\documentclass{amsart}
\usepackage[ps2pdf]{hyperref}
\usepackage{amsmath,amssymb}
\usepackage{psfrag}
\usepackage{graphicx}
\usepackage{subfigure}
\usepackage{leftidx}
\usepackage[all,poly]{xy}
\usepackage[latin1]{inputenc}
\usepackage[dvipsnames]{xcolor}

\newtheorem{thm}{Theorem}[section]

\newtheorem{lem}[thm]{Lemma}

\theoremstyle{definition}

\DeclareSymbolFont{bboldforun}{U}{bbold}{m}{n}
\DeclareSymbolFontAlphabet{\mathbboldforun}{bboldforun}
\newcommand{\un}{\mathbboldforun{1}}
\newcommand{\CCC}{\mathbboldforun{C}}
\newcommand{\co}{\colon}
\newcommand{\id}{\mathrm{id}}

\newcommand{\cc}{\mathcal{C}}
\newcommand{\cch}{\mathcal{C}_{\mathrm{hom}}}

\newcommand{\dd}{\mathcal{D}}

\newcommand{\ee}{\mathcal{E}}

\newcommand{\ZZ}{\mathbb{Z}}
\newcommand{\RR}{\mathbb{R}}

\newcommand{\Ker}{\mathrm{Ker}}

\newcommand{\iso}{\stackrel{\sim}{\longrightarrow}}

\newcommand{\kk}{\Bbbk}
\newcommand{\kt}{$\Bbbk$\nobreakdash-\hspace{0pt}}

\newcommand{\Aut}{\mathrm{Aut}}

\newcommand{\End}{\mathrm{End}}
\newcommand{\Hom}{\mathrm{Hom}}

\newcommand{\tr}{\mathrm{tr}}

\newcommand{\rank}{\mathrm{rank}}

\newcommand{\lev}{\mathrm{ev}}

\newcommand{\rev}{\widetilde{\mathrm{ev}}}
\newcommand{\lcoev}{\mathrm{coev}}
\newcommand{\rcoev}{\widetilde{\mathrm{coev}}}
\newcommand{\ldual}[1]{#1^{*}}

\newcommand{\scaledraw}[1]{A}
\newcommand{\scaleraisedraw}[2]{A}
\newcommand{\rsdraw}[3]{\raisebox{-#1\height}{\scalebox{#2}{\includegraphics{#3.eps}}}}
\newcommand{\labela}{\renewcommand{\labelenumi}{{\rm (\alph{enumi})}}}
\newcommand{\labeli}{\renewcommand{\labelenumi}{{\rm (\roman{enumi})}}}

\begin{document}

\title{Surgery HQFT}
\author[V. Turaev]{Vladimir Turaev}
\address{%
Vladimir Turaev\newline
\indent Department of Mathematics, \newline
\indent Indiana University \newline
\indent Bloomington IN47405 \newline
\indent USA \newline
\indent e-mail: vtouraev@indiana.edu}
\author[A. Virelizier]{Alexis Virelizier}
\address{%
Alexis Virelizier\newline
\indent Department of Mathematics, \newline
\indent University Lille 1\newline
\indent 59655 Villeneuve d'Ascq \newline
\indent France \newline
\indent e-mail:  alexis.virelizier@math.univ-lille1.fr}
\date{\today}

\begin{abstract}
Homotopy Quantum Field
Theories (HQFTs)
generalize more familiar Topological Quantum Field Theories (TQFTs). In generalization of the surgery construction of    3-dimensional TQFTs from modular categories, we use surgery to derive  3-dimensional HQFTs from $G$-modular categories.
\end{abstract}
\maketitle

\setcounter{tocdepth}{1} \tableofcontents

\section{Introduction}\label{sec-Intro} Homotopy Quantum Field
  Theories
  (HQFTs)
were introduced in \cite{Tu1} as   generalizations  of
  Topological Quantum Field Theories (TQFTs). An
  HQFT   produces
\lq\lq quantum" invariants of   manifolds    endowed with   homotopy classes of
maps   to   a fixed   space. Such homotopy classes    represent additional structures on  manifolds whose nature depends on the choice of the target space. We shall focus on HQFTs whose target space is an Eilenberg-MacLane space $K(G,1)$ where $G$ is a  discrete  group.
The maps to $K(G,1)$ encode flat principal $G$-bundles over manifolds.
When  $G=1$,
 we recover  the
usual TQFTs.

The study of TQFTs is
 primarily motivated by their interest for theoretical physics, and
 their main  applications outside of pure mathematics lie  in   physics and in the theory of   quantum
 computations, see, for example, \cite{Ca, Wa}.
The study of TQFTs has also  found
 applications  in   knot theory, low dimensional topology, and in the theory of Hopf algebras and    monoidal categories.
 One  expects similar applications for HQFTs.

The    construction of    Reshe\-ti\-khin and Turaev  \cite{RT2} derives  a   3-dimensional TQFT  from  a modular  category.
The principal aim of the present paper is to extend this construction to HQFTs.   Specifically,   we
  show that every $G$-modular category in the sense of \cite{TVi3} gives rise to a 3-dimensional   HQFT   with target  $K(G,1)$. The construction is based on surgery  presentations of 3-manifolds by links in Euclidean 3-space $\RR^3$, and the resulting HQFT is called the {\it surgery HQFT}.

  The definition of a $G$-modular category will be fully recalled in the body of the paper. Note here that such a category is $G$-graded and $G$-ribbon.
A $G$-graded category   is a  monoidal category   whose  objects have
a   multiplicative $G$-grading.   Such a  category is $G$-ribbon if it carries
   additional  structures:   a $G$-braiding,  a $G$-twist, and  an action of
$G$  called the   crossing.
For  $G=1$,    we recover the standard notions of
 braided/ribbon/modular categories.

An alternative approach to 3-dimensional  HQFTs based on the technique of state sums starts with so-called spherical $G$-graded categories, see
 \cite{TVi2}. In a sequel to this  paper, we will relate these  approaches via the following theorem:
the  state sum HQFT
associated with a spherical $G$-graded category is isomorphic to the
surgery HQFT associated with the $G$-center of
that   category.  This theorem is highly non-trivial already for $G=1$ (that is, for TQFTs); in this case  it was first established in \cite{TVi1} and slightly later  - but independently  -  in \cite{KB,Ba1,Ba2}.  For   the   notion of the $G$-center of a $G$-graded category, we refer the reader to \cite{TVi3}; we will not use   $G$-centers here.

  A surgery construction for HQFTs  was  first suggested  in   \cite{Tu1}.  However,   the class of $G$-modular categories   studied in \cite{Tu1} is very narrow. For example, the $G$-centers of  spherical $G$-graded categories only rarely belong to this class which makes it inadequate for the above-mentioned theorem. The   notion of  a $G$-modular category used here is considerably more general and does include  the $G$-centers. Working in this generality requires us to to introduce quite a number of   new   algebraic and geometric techniques
essential for this work.

The key ingredient in the definition of the surgery TQFT associated with a modular category $\cc$ is a certain  functor  from the category
   of $\cc$-colored ribbon graphs in $\RR^2 \times [0,1]$ to $\cc$, see \cite{RT1, RT2, Sh, Tu0}.  A $\cc$-coloring of a
ribbon graph labels the edges of the
graph with objects of $\cc$ and labels   the vertices of the graph with morphisms in $\cc$. Note that the
vertices of a ribbon graph are rectangles called \lq\lq coupons".
We need  to define a   version of the functor above for
ribbon graphs
whose exteriors are equipped
with homotopy classes of maps to $K(G,1)$.
  Such a homotopy class is   determined by a  homomorphism  from the fundamental group
of the   graph   exterior  to~$G$, and we rather work with homomorphisms. In the role of the base point  of the graph
exterior  we take  any point  with big second coordinate. In the role of   $\cc$ we take a $G$-ribbon $G$-graded category.  A
$\cc$-coloring attributes an object of $\cc$ to each path in the
graph exterior leading from the base point to an edge of the graph.
This  object must be preserved under homotopies of the path and must
behave in a \lq\lq controlled" way under changes of the path. In
particular, under multiplication of the path by a loop at the base
point, the object should be modified via the crossing automorphism
of $\cc$ determined by the element of $G$ represented by the loop.
Furthemore, a $\cc$-coloring attributes morphisms in $\cc$ to paths from the base point  to the
coupons.   Again, a   \lq\lq controlled behavior" is required.

We define a monoidal category $\mathcal G_{\cc}$ of $\cc$-colored ribbon graphs. The objects of $\mathcal G_{\cc}$
 are finite sequences of pairs      (an object of $\cc$, a sign $\pm $).
 These objects encode the colors and the orientations of   $\cc$-colored ribbon graphs near
  the inputs and the outputs. The morphisms of $\mathcal G_{\cc}$ are appropriate equivalence classes of
   $\cc$-colored ribbon graphs in $\RR^2 \times [0,1]$ having no circle components.   It should be stressed that the category $\mathcal G_{\cc}$ does not include knots or links. It does
     include    $\cc$-colored string links (and in particular, $\cc$-colored  braids) which are viewed as ribbon graphs without coupons.

For any $G$-ribbon $G$-graded category $\cc$, we define a
monoidal functor  $F_\cc \co  \mathcal G_{\cc} \to \cc$. As in the
classical case, the construction of $F_\cc$ uses graph diagrams and
Reidemeister moves though, in our setting, the correspondence between  graphs and   diagrams becomes quite delicate. For  $G=1$,  the functor $F_\cc $ is the
restriction of the functor of \cite{RT1, Sh, Tu0} to ribbon graphs without
circle components.

We then show how to transform links in $\RR^3$ into $\cc$-colored ribbon graphs. This transformation, called \lq\lq insertion of coupons",  allows  us to apply $F_\cc$ to  links and   leads  to the surgery   HQFT associated with $\cc$.


The paper consists of 13 sections. Sections \ref{Preliminaries on categories and functors}--\ref{sect-braided-ribbon} are  devoted to an algebraic discussion
 of $G$-graded categories, $G$-braidings, twists, etc.    In Sections \ref{Colored   $G$-graphs}--\ref{Invariants of special colored $G$-graphs}  we define and study the functor $F_\cc$. In  Sections \ref{$G$-modular categories} and \ref{sect-modular-Gcat}  we define $G$-graded modular categories and construct the associated HQFTs.

 Throughout the paper,  we fix  a  (discrete) group $G$.

  {\it Acknowledgements.}   The work of V.\ Turaev was partially supported
by the NSF grant DMS-1202335.   A. Virelizier gratefully
acknowledges the support and hospitality of the Max Planck Institute for Mathematics in  Bonn  where a   part of this work was carried out.

\section{Preliminaries on categories}\label{Preliminaries on categories and functors}

We recall the basic definitions of the theory of monoidal categories.

\subsection{Conventions}\label{Conventions}  The symbol $\cc$ will denote  a monoidal category with unit object~$\un=\un_\cc$.
Notation  $X\in \cc$    means    that $X$ is an object of $\cc$.
To simplify the formulas, we  will always pretend that   $\cc$   is
strict. Consequently, we  omit brackets in the monoidal products and
suppress the associativity constraints $(X\otimes Y)\otimes Z\cong
X\otimes (Y\otimes Z)$ and the unitality constraints $X\otimes \un
\cong X\cong \un \otimes X$.    By the   monoidal   product $X_1 \otimes
X_2 \otimes \cdots \otimes X_n$ of $n\geq 2$ objects $X_1,...,
X_n\in \cc$ we mean $(... ((X_1\otimes X_2) \otimes X_3) \otimes
\cdots \otimes X_{n-1}) \otimes X_n$.

\subsection{Pivotal    categories}\label{pivotall}
A monoidal category $\cc=(\cc,\otimes,\un)$  is \emph{pivotal} (see
\cite{Malt}) if   for each
object $X$ of $\cc$, we have a \emph{dual
object}~$X^*\in \cc$ and four morphisms
\begin{align*}
& \lev_X \co X^*\otimes X \to\un,  \qquad \lcoev_X\co \un  \to X \otimes X^*,\\
&   \rev_X \co X\otimes X^* \to\un, \qquad   \rcoev_X\co \un  \to X^* \otimes X,
\end{align*}
such that
\begin{enumerate}
  \renewcommand{\labelenumi}{{\rm (\alph{enumi})}}
   \item for any object $X\in \cc$,
  \begin{gather*}
(\id_X \otimes \lev_X)(\lcoev_X \otimes \id_X)=\id_X \quad \text{and} \quad (\lev_X \otimes \id_{X^*})(\id_{X^*} \otimes \lcoev_X)=\id_{X^*}, \\
(\rev_X \otimes \id_X)(\id_X \otimes \rcoev_X)=\id_X \quad \text{and} \quad (\id_{X^*} \otimes \rev_X)(\rcoev_X \otimes \id_{X^*})=\id_{X^*};
\end{gather*}
\item for every morphism $f\co X \to Y$ in $\cc$,
the \emph{left dual}
$$
f^*= (\lev_Y \otimes  \id_{X^*})(\id_{Y^*}  \otimes f \otimes \id_{X^*})(\id_{Y^*}\otimes \lcoev_X) \colon Y^*\to X^*$$ is equal to the
 \emph{right  dual}
$$
f^*= (\id_{X^*} \otimes \rev_Y)(\id_{X^*} \otimes f \otimes \id_{Y^*})(\rcoev_X \otimes \id_{Y^*}) \colon Y^*\to X^*;$$
\item for all   $X,Y\in
\cc$, the  \emph{left  monoidal constraint}
$$
(\lev_X  \otimes \id_{(Y \otimes X)^*})(\id_{X^*}  \otimes \lev_Y \otimes  \id_{X \otimes (Y \otimes X)^*})(\id_{X^* \otimes Y^*}\otimes \lcoev_{Y \otimes
X})\colon \ldual{X} \otimes \ldual{Y} \to (Y \otimes X)^*
$$
is equal to the  \emph{right  monoidal constraint}
$$
(\id_{(Y \otimes X)^*} \otimes \rev_Y)(\id_{(Y \otimes X)^*\otimes Y} \otimes
\rev_X \otimes \id_{Y^*})(\rcoev_{Y \otimes X}\otimes \id_{X^*
\otimes Y^*}) \colon \ldual{X} \otimes \ldual{Y} \to (Y \otimes
X)^*;
$$
\item  $\lev_\un=\rev_\un \colon \un^* \to \un$ (or, equivalently, $\lcoev_\un=\rcoev_\un\colon \un  \to \un^*$).
\end{enumerate}

In what follows, for a pivotal category $\cc$, we will suppress the duality constraints
$\un^* \cong \un$ and $X^* \otimes Y^*\cong (Y\otimes X)^* $. For
example, we will write $(f \otimes g)^*=g^* \otimes f^*$ for
morphisms $f,g$ in~$\cc$.

\subsection{Traces and dimensions}\label{sec-traces}
Let $\cc$ be a pivotal category. For any endomorphism $f$ of an object
$X\in \cc$, one   defines the {\it left} and  {\it right traces}
$$\tr_l(f)=\lev_X(\id_{\ldual{X}} \otimes f) \rcoev_X  \quad {\text {and}}\quad \tr_r(f)=  \rev_X( f \otimes
\id_{\ldual{X}}) \lcoev_X .$$ Both traces take values in   the commutative monoid   $\End_\cc(\un)$ and are symmetric: $\tr_l
(gh)=\tr_l(hg)$ for any morphisms $g\co X\to Y$, $h\co Y\to X$ in $\cc$
and similarly
 for $\tr_r$.  Also $\tr_{l/r}(f)=\tr_{r/l}( {f}^*) $ for any
 endomorphism $f$ of an object. The  {\it left} and  {\it right}
 dimensions of an object $X\in \cc$ are defined by
 $\dim_{l/r}(X)=\tr_{l/r}(\id_X)$. Clearly,
 $\dim_{l/r}(X)=\dim_{r/l}(X^*) $ for all $X$.


\subsection{Monoidal  functors}\label{sect-monofunctor}
Let $\cc$ and $\dd$ be    monoidal categories. A \emph{monoidal
functor} from $\cc$ to $\dd$ is a triple $(F,F_2,F_0)$, where $F\co
\cc \to \dd$ is a functor, $$ F_2=\{F_2(X,Y) \co F(X) \otimes F(Y)
\to F(X \otimes Y)\}_{X,Y \in \cc} $$ is a natural transformation
from $F\otimes F$ to $F \otimes$, and $F_0\co\un_\dd \to F(\un_\cc)$ is a
morphism in~$\dd$, such that the diagrams
\begin{equation}\label{stmonoidal1}
\begin{split}
    \xymatrix@R=1cm @C=3cm { F(X) \otimes F(Y) \otimes F(Z) \ar[r]^-{\id_{F(X)} \otimes F_2(Y,Z)}  \ar[d]_-{F_2(X,Y) \otimes \id_{F(Z)}}
    & F(X) \otimes F(Y \otimes Z) \ar[d]^-{F_2(X,Y \otimes Z)} \\
    F(X \otimes Y) \otimes F(Z) \ar[r]_-{F_2(X \otimes Y, Z)}
    & F(X \otimes Y \otimes Z),
    }
\end{split}
\end{equation}
\begin{equation}\label{stmonoidal2}
\begin{split}
\xymatrix@R=1cm @C=2,5cm { F(X)  \ar[rd]^-{\id_{F(X)}} \ar[r]^-{\id_{F(X)} \otimes F_0}  \ar[d]_-{F_0 \otimes \id_{F(X)}}
    & F(X) \otimes F(\un_\cc) \ar[d]^-{F_2(X,\un_\cc)} \\
    F(\un_\cc) \otimes F(X) \ar[r]_-{F_2(\un_\cc,X)}
    & F(X)
    }
\end{split}
\end{equation}
%
commute for all objects $X,Y,Z \in \cc$ (see \cite{ML1}).  Composing
  monoidal    products of  $F_1=\id_F$, and $F_2$, we can define for
every   integer $n \geq 3$, a natural transformation
$$
F_n=\{F_n(X_1,\dots,X_n) \co F(X_1) \otimes \cdots \otimes F(X_n) \to F(X_1 \otimes \cdots \otimes X_n)\}_{X_1,\dots,X_n \in \cc}.
$$
For instance, $F_3(X,Y,Z)=F_2(X,Y \otimes Z) (F_1(X) \otimes F_2(Y,Z))$.
The commutativity of the diagrams~\eqref{stmonoidal1}
and~\eqref{stmonoidal2} ensures that $F_n$ does not depend on the
way it is built from $F_1$ and $F_2$.

A monoidal functor $(F,F_2,F_0)$ is   \emph{strong}   if $F_2$ and
$F_0$ are isomorphisms. A monoidal functor $(F,F_2,F_0)$ is   \emph{strict}
if $F_2$ and $F_0$ are   identity morphisms.

If $F\co \cc \to \dd$ and $G\co \dd \to \ee$ are two monoidal
functors between monoidal categories, then their composition $GF\co
\cc \to \ee$
is a monoidal functor with 
$$
(GF)_0=G(F_0)G_0 \quad \text{and} \quad
(GF)_2=\{G(F_2(X,Y))\, G_2(F(X),F(Y))\}_{X,Y \in \cc}.
$$

Let $F\co\cc \to \dd$ and $G\co\cc \to \dd$ be two monoidal
functors. A natural transformation $\varphi=\{\varphi_X \co F(X) \to
G(X)\}_{X \in \cc}$ from $F$ to $G$ is \emph{monoidal} if it
satisfies
\begin{equation}\label{monoidalnattrans} G_0=\varphi_\un F_0
 \quad \text{and} \quad  \varphi_{X \otimes Y} F_2(X,Y)= G_2(X,Y) (\varphi_X \otimes
\varphi_Y)
\end{equation}
for all objects $X,Y$ of $\cc$. A \emph{monoidal natural
isomorphism} between $F$ and $G$ is a monoidal natural
transformation
 $\varphi$ from $F$ to $G$ which is an isomorphism in the sense that each $\varphi_X$ is an isomorphism. The inverse $\varphi^{-1}=\{\varphi_X^{-1} \co G(X) \to F(X)\}_{X \in \cc}$ is then a monoidal natural transformation from $G$ to $F$.


 \subsection{Pivotal functors}\label{sect-pivotal-functor}

 Given
  a strong monoidal functor $F\co \cc \to \dd$ between pivotal categories,  we define for each $X\in \cc$ a
  morphism $F^l(X)\co F(X^*) \to F(X)^*$ by
$$
F^l(X)=(F_0^{-1}F(\lev_X)F_2(X^*,X) \otimes \id_{F(X)^*})(\id_{F(X^*)} \otimes \lcoev_{F(X)}).
$$
It is well-known that $F^l=\{F^l(X) \co F(X^*) \to F(X)^*\}_{X \in
\cc}$ is a monoidal natural isomorphism
 which   preserves the left
duality in the sense that for all $X \in \cc$,
\begin{align}
& F(\lev_X)=F_0\lev_{F(X)}(F^l(X) \otimes \id_{F(X)})F_2(X^*,X)^{-1}, \label{pivotal-lev} \\
& F(\lcoev_X)=F_2(X,X^*)(\id_{F(X)} \otimes F^l(X)^{-1} )\lcoev_{F(X)}F_0^{-1}. \label{pivotal-lcoev}
\end{align}
The monoidality of $F^l$ means that $(F_0^{-1})^*=F^l(\un)F_0$ and for all $X,Y \in \cc$,
$$
F^l(X \otimes Y)F_2(Y^*,X^*)=(F_2(X,Y)^{-1})^*(F^l(Y) \otimes F^l(X)).
$$

Likewise, the morphisms $ \{F^r(X) \co F(X^*) \to F(X)^*\}_{X \in
\cc}$, defined by
$$
F^r(X)=(\id_{F(X)^*} \otimes F_0^{-1}F(\rev_X)F_2(X,X^*))(\rcoev_{F(X)} \otimes \id_{F(X^*)}),
$$
form a monoidal natural isomorphism $F^r$  preserving the right
duality: for all $X \in \cc$,
\begin{align}
& F(\rev_X)=F_0\rev_{F(X)}(\id_{F(X)} \otimes F^r(X))F_2(X,X^*)^{-1},  \label{pivotal-rev} \\
& F(\rcoev_X)=F_2(X^*,X)(F^r(X)^{-1}\otimes \id_{F(X)})\rcoev_{F(X)}F_0^{-1}. \label{pivotal-rcoev}
\end{align}
One can check  that    $F^l$ and $F^r$ are related by
\begin{equation}\label{FrFlcomp}
F^l(X^*)F(\phi_X)=F^r(X)^*\phi_{F(X)}
\end{equation}
for all $X \in \cc$ where $ \{\phi_X \co X \to X^{**}\}_{X\in\cc} $ is the
 {\it pivotal structure} in $\cc$ defined by
\begin{equation}\label{pivotal-struct}
\phi_X=(\rev_X \otimes \id_{X^{**}})(\id_X \otimes \lcoev_{X^*})\co X \to X^{**}.
\end{equation}

The functor $F\co \cc \to \dd$ is  said to be  \emph{pivotal} if
$F^l(X)=F^r(X)$ for any $X \in \cc$. In this case, $F^l=F^r$ is
denoted by $F^1$.

\subsection{Penrose graphical calculus}\label{sect-penrose} We will represent morphisms in a category $\cc$ by plane   diagrams to be read from the bottom to the top.
The  diagrams are made of   oriented arcs colored by objects of
$\cc$  and of boxes colored by morphisms of~$\cc$.  The arcs connect
the boxes and   have no   intersections or self-intersections.
The identity $\id_X$ of $X\in  \cc $, a morphism $f\co X \to Y$,
and the composition of two morphisms $f\co X \to Y$ and $g\co Y \to
Z$ are represented as follows:
\begin{center}
\psfrag{X}[Bc][Bc]{\scalebox{.7}{$X$}} \psfrag{Y}[Bc][Bc]{\scalebox{.7}{$Y$}} \psfrag{h}[Bc][Bc]{\scalebox{.8}{$f$}} \psfrag{g}[Bc][Bc]{\scalebox{.8}{$g$}}
\psfrag{Z}[Bc][Bc]{\scalebox{.7}{$Z$}} $\id_X=$ \rsdraw{.45}{.9}{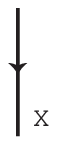}\,,\quad $f=$ \rsdraw{.45}{.9}{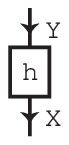} ,\quad \text{and} \quad $gf=$ \rsdraw{.45}{.9}{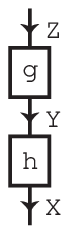}\,.
\end{center}
  If $\cc$ is
monoidal, then the monoidal product of two morphisms $f\co X \to Y$
and $g \co U \to V$ is represented by juxtaposition:
\begin{center}
\psfrag{X}[Bc][Bc]{\scalebox{.7}{$X$}} \psfrag{h}[Bc][Bc]{\scalebox{.8}{$f$}}
\psfrag{Y}[Bc][Bc]{\scalebox{.7}{$Y$}}  $f\otimes g=$ \rsdraw{.45}{.9}{morphism} \psfrag{X}[Bc][Bc]{\scalebox{.8}{$U$}} \psfrag{g}[Bc][Bc]{\scalebox{.8}{$g$}}
\psfrag{Y}[Bc][Bc]{\scalebox{.7}{$V$}} \rsdraw{.45}{.9}{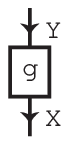}\,.
\end{center}
If $\cc$ is pivotal, then we allow arcs directed  upwards. Such an arc, colored with $X\in \cc$, contributes $X^*$ to the source/target of the associated morphism. For example, $\id_{X^*}$  and a morphism $f\co X^* \otimes
Y \to U \otimes V^* \otimes W$  may be depicted as:
\begin{center}
 $\id_{X^*}=$ \, \psfrag{X}[Bl][Bl]{\scalebox{.7}{$X$}}
\rsdraw{.45}{.9}{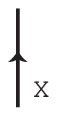} $=$  \,
\psfrag{X}[Bl][Bl]{\scalebox{.7}{$\ldual{X}$}}
\rsdraw{.45}{.9}{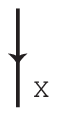}  \quad and \quad
\psfrag{X}[Bc][Bc]{\scalebox{.7}{$X$}}
\psfrag{h}[Bc][Bc]{\scalebox{.8}{$f$}}
\psfrag{Y}[Bc][Bc]{\scalebox{.7}{$Y$}}
\psfrag{U}[Bc][Bc]{\scalebox{.7}{$U$}}
\psfrag{V}[Bc][Bc]{\scalebox{.7}{$V$}}
\psfrag{W}[Bc][Bc]{\scalebox{.7}{$W$}} $f=$
\rsdraw{.45}{.9}{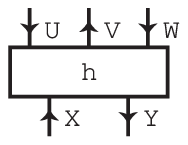} \,.
\end{center}
The duality morphisms   are depicted as follows:
\begin{center}
\psfrag{X}[Bc][Bc]{\scalebox{.7}{$X$}} $\lev_X=$ \rsdraw{.45}{.9}{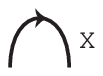}\,,\quad
 $\lcoev_X=$ \rsdraw{.45}{.9}{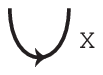}\,,\quad
$\rev_X=$ \rsdraw{.45}{.9}{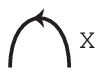}\,,\quad
\psfrag{C}[Bc][Bc]{\scalebox{.7}{$X$}} $\rcoev_X=$
\rsdraw{.45}{.9}{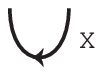}\,.
\end{center}
The dual of a morphism $f\co X \to Y$ and the
  traces of a morphism $g\co X \to X$ can be depicted as
follows:
\begin{center}
\psfrag{X}[Bc][Bc]{\scalebox{.7}{$X$}} \psfrag{h}[Bc][Bc]{\scalebox{.8}{$f$}}
\psfrag{Y}[Bc][Bc]{\scalebox{.7}{$Y$}} \psfrag{g}[Bc][Bc]{\scalebox{.8}{$g$}}
$f^*=$ \rsdraw{.45}{.9}{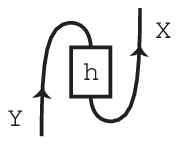}$=$ \rsdraw{.45}{.9}{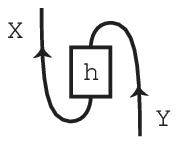}\quad \text{and} \quad
$\tr_l(g)=$ \rsdraw{.45}{.9}{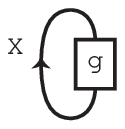}\,,\quad  $\tr_r(g)=$ \rsdraw{.45}{.9}{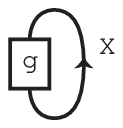}\,.
\end{center}
  If $\cc$ is pivotal, then    the morphism  represented by a plane diagram
is invariant under isotopies of the diagram  in the plane keeping
  the bottom and   top endpoints.

\section{$G$-graded categories}

\subsection{$G$-graded categories}\label{G-cat-deb} A \emph{$G$-graded category}  is a
  monoidal category   $(\cc,\otimes,\un)$   endowed with a system of
pairwise disjoint full   subcategories $\{{\mathcal
C}_{\alpha}\}_{{\alpha}\in G}$ such that
\begin{enumerate}
  \labela

  \item   all $\Hom$-sets in $\cc$ are modules over a (fixed) commutative ring $\kk$ and the composition and the    monoidal    product of morphisms  are $\kk$-bilinear;

  \item  $\un \in  {\mathcal C}_1$ and   if $U\in {\mathcal C}_{\alpha}$ and  $V\in {\mathcal C}_{\beta}$, then
$U\otimes V\in {\mathcal C}_{{\alpha}{\beta}}$;

\item if $U\in {\mathcal C}_{\alpha}$ and $V\in
{\mathcal C}_{\beta}$ with ${\alpha}\neq {\beta}$, then $\Hom_{\mathcal C}
(U,V)=0$.

\end{enumerate}

  The monoidal category ${\mathcal C}_1$ corresponding to the
neutral element $1\in G$ is called the {\it neutral component}
of~${\mathcal C}$.

 An object $X$ of a $G$-graded category ${\cc} $ is   {\it homogeneous} if   $X\in \cc_\alpha \subset \cc$ for some $\alpha\in
G$. Such an $\alpha$
 is then uniquely determined by $X$ and  denoted $|X|$. It is allowed for
  objects $X \in \cc_\alpha$, $ Y\in \cc_\beta $ with $\alpha\neq \beta$ to be
 isomorphic. However, in this case, $X$ and $Y$ are {\it zero objects} in the sense that  $\End_\cc(X)=\End_\cc(Y)=0$.


The definition of $G$-graded categories given above is  more
general than the corresponding definition in \cite{TVi2} where we
additionally require  the existence of direct sums and the splitting
of arbitrary objects into direct sums of homogeneous objects. These
conditions will not be needed in the present paper.

\subsection{$G$-crossed categories}
Given a monoidal   category $\cc$,   denote by
$\Aut(\cc)$ the category of  strong monoidal auto-equivalences of
$\cc$. Its objects are   strong monoidal functors
 $\cc\to \cc$ that are   $\kk$-linear on the $\Hom$-sets and are equivalences of categories. The  morphisms in  $\Aut(\cc)$ are
 monoidal natural isomorphisms. The category $\Aut(\cc)$ has a canonical structure of a strict monoidal category, in which
the monoidal product is the composition of monoidal functors and the
monoidal unit is the identity endofunctor $1_\cc$ of~$\cc$.

Denote by $\overline{G}$ the category whose objects are elements of
the group $G$ and morphisms are identities. We view $\overline{G}$ as a strict monoidal category with
monoidal product
$\alpha\otimes \beta= \beta \alpha$ for all $\alpha, \beta \in G$.

By a \emph{$G$-crossed category} we mean a $G$-graded category $\cc$   endowed with a \emph{crossing}, that is, a
strong monoidal functor  $\varphi\co \overline{G} \to
\Aut(\cc)$  such that $ \varphi_\alpha(\cc_\beta) \subset
\cc_{\alpha^{-1} \beta \alpha} $ for all $\alpha,\beta \in G$.
For each $\alpha\in G$, the crossing $\varphi$ provides a
strong monoidal equivalence $\varphi_\alpha\co \cc\to \cc$.  By
definition, $\varphi$ comes equipped with isomorphisms
$(\varphi_\alpha)_0 \co \un \iso \varphi_\alpha(\un)$ in $\cc$  and
with natural isomorphisms
\begin{align*}
&(\varphi_\alpha)_2=\{(\varphi_\alpha)_2(X,Y) \co \varphi_\alpha(X) \otimes \varphi_\alpha(Y) \iso \varphi_\alpha(X \otimes Y)\}_{X,Y \in \cc},\\
&\varphi_2=\bigl\{\varphi_2(\alpha,\beta)=\{\varphi_2(\alpha,\beta)_X  \co \varphi_\alpha\varphi_\beta(X)\iso \varphi_{\beta\alpha}(X)\}_{X \in \cc}\bigr\}_{\alpha,\beta \in G},\\
&\varphi_0=\{(\varphi_0)_X \co X \iso \varphi_1(X)\}_{X \in \cc},
\end{align*}
such that  $(\varphi_0)_\un=(\varphi_1)_0$ and, for all
$\alpha,\beta,\gamma \in G$ and   $X,Y,Z \in \cc$,   the
following diagrams commute:
\begin{equation}\label{crossing1}
\begin{split}
\xymatrix@R=1cm @C=3cm {
\varphi_\alpha(X) \otimes \varphi_\alpha(Y) \otimes \varphi_\alpha(Z) \ar[r]^-{\id_{\varphi_\alpha(X)} \otimes (\varphi_\alpha)_2(Y,Z)} \ar[d]^{(\varphi_\alpha)_2(X,Y) \otimes \id_{\varphi_\alpha(Z)}} & \varphi_\alpha(X) \otimes \varphi_\alpha(Y \otimes  Z)
\ar[d]^{(\varphi_\alpha)_2(X,Y \otimes Z)}  \\
\varphi_\alpha(X \otimes Y) \otimes \varphi_\alpha(Z) \ar[r]_-{(\varphi_\alpha)_2(X \otimes Y, Z)} & \varphi_\alpha(X \otimes Y \otimes Y)
}
\end{split}
\end{equation}

\begin{equation}\label{crossing2}
\begin{split}
\xymatrix@R=1cm @C=3cm {
\varphi_\alpha(X) \ar[rd]^-{\id_{\varphi_\alpha(X)}} \ar[r]^-{\id_{\varphi_\alpha(X)} \otimes (\varphi_\alpha)_0} \ar[d]_{(\varphi_\alpha)_0
\otimes \id_{\varphi_\alpha(X)}} & \varphi_\alpha(X) \otimes \varphi_\alpha(\un) \ar[d]^{(\varphi_\alpha)_2(X,\un)}\\
\varphi_\alpha(\un) \otimes \varphi_\alpha(X) \ar[r]_-{(\varphi_\alpha)_2(\un,X)} & \varphi_\alpha(X)
}\end{split}
\end{equation}

\def\MyNode{\ifcase\xypolynode\or
      \varphi_{\beta\alpha}(X \otimes Y)
    \or
      \varphi_{\beta\alpha}(X) \otimes \varphi_{\beta\alpha} (Y)
    \or
      \varphi_\alpha\varphi_\beta(X) \otimes \varphi_\alpha\varphi_\beta(Y)
    \or
      \varphi_\alpha(\varphi_\beta(X) \otimes \varphi_\beta(Y))
    \or
      \varphi_\alpha\varphi_\beta(X \otimes Y)
    \fi
  }%
\begin{equation}\label{crossing3}
\begin{split}
  \xy/r10pc/: (0,.3)::
    \xypolygon5{~>{}\txt{\ \ \strut\ensuremath{\MyNode}}} 
    \ar "2";"1" ^-{\qquad (\varphi_{\beta\alpha})_2(X,Y)}
    \ar "3";"2" ^-{\varphi_2(\alpha,\beta)_X \otimes \varphi_2(\alpha,\beta)_Y \qquad\qquad}
    \ar "3";"4" _-{(\varphi_\alpha)_2(\varphi_\beta(X),\varphi_\beta(Y))}
    \ar "4";"5" _-{\varphi_\alpha((\varphi_\beta)_2(X,Y))}
    \ar "5";"1" _-{\varphi_2(\alpha,\beta)_{X\otimes Y}}
  \endxy
\end{split}
\end{equation}

\begin{equation}\label{crossing4}
\begin{split}
\xymatrix@R=1cm @C=2.5cm { \un \ar[d]_-{(\varphi_\alpha)_0} \ar[r]^-{(\varphi_{\beta\alpha})_0}  & \varphi_{\beta\alpha}(\un) \\
\varphi_\alpha(\un) \ar[r]_-{\varphi_\alpha\bigl((\varphi_\beta)_0\bigr) } & \varphi_\alpha\varphi_\beta(\un) \ar[u]_-{\varphi_2(\alpha,\beta)_\un}
}
\end{split}
\end{equation}

\begin{equation}\label{crossing7}
\begin{split}
\xymatrix@R=1cm @C=.5cm { X \otimes Y \ar[rd]_-{(\varphi_0)_X \otimes (\varphi_0)_Y}\ar[rr]^-{(\varphi_0)_{X \otimes Y}} && \varphi_1(X \otimes Y)\\
& \varphi_1(X) \otimes \varphi_1(Y) \ar[ru]_-{\;(\varphi_1)_2(X,Y)} &
}
\end{split}
\end{equation}

\begin{equation}\label{crossing5}
\begin{split}
\xymatrix@R=1cm @C=2.5cm {
\varphi_\alpha\varphi_\beta \varphi_\gamma(X) \ar[r]^-{\varphi_\alpha(\varphi_2(\beta, \gamma)_X)}\ar[d]_{\varphi_2(\alpha,\beta)_{\varphi_\gamma(X)}} & \varphi_\alpha\varphi_{\gamma\beta}(X)
\ar[d]^{\varphi_2(\alpha, \gamma \beta)_X} \\
\varphi_{\beta\alpha} \varphi_\gamma(X) \ar[r]_-{\varphi_2( \beta\alpha, \gamma)_X} & \varphi_{\gamma\beta\alpha}(X)}
\end{split}
\end{equation}

\begin{equation}\label{crossing6}
\begin{split}
\xymatrix@R=1cm @C=2.5cm {
\varphi_\alpha(X) \ar[rd]^-{\id_{\varphi_\alpha(X)}} \ar[r]^-{\varphi_\alpha((\varphi_0)_X)} \ar[d]_{(\varphi_0)_{\varphi_\alpha(X)}} & \varphi_\alpha\varphi_1(X)  \ar[d]^{\varphi_2(\alpha,1)_X}\\
\varphi_1\varphi_\alpha(X) \ar[r]_-{\varphi_2(1,\alpha)_X} & \varphi_\alpha(X).
}
\end{split}
\end{equation}
The commutativity of the diagrams  \eqref{crossing1} and
\eqref{crossing2} means that
$(\varphi_\alpha,(\varphi_\alpha)_2,(\varphi_\alpha)_0)$ is a
monoidal endofunctor of~$\cc$.   The   diagrams
\eqref{crossing3} and \eqref{crossing4} indicate  that the natural
transformation $\varphi_2(\alpha,\beta)$ is monoidal. The
diagram \eqref{crossing7} and the equality
$(\varphi_0)_\un=(\varphi_1)_0$ indicate that the natural
transformation $\varphi_0$ is monoidal. The diagrams
\eqref{crossing5} and \eqref{crossing6} indicate  that
$(\varphi,\varphi_2,\varphi_0)$ is a monoidal functor.

Crossings in   $G$-graded categories were   introduced in \cite{Tu1}
in the special case  where    $\varphi$ and all $\varphi_\alpha$'s are strict monoidal functors, that is,   all the morphisms
$\varphi_2(\alpha,\beta)_X$, $(\varphi_0)_X$, $(\varphi_\alpha)_2(X,Y)$, and $(\varphi_\alpha)_0$  are identity morphisms.

\subsection{Pivotality}\label{sect-pivot-cross-Gcat} A $G$-graded category $\cc  $ is {\emph{pivotal} if the underlying  monoidal category of $\cc$ is pivotal  and  for all  $\alpha\in
G$ and $X \in \cc_\alpha$ we have $X^* \in  \cc_{\alpha^{-1}}$. A  crossing $\varphi$ in a  pivotal $G$-graded category $\cc  $ is {\emph{pivotal} if  all the
   functors $\{\varphi_\alpha
\co \cc \to \cc\}_{\alpha \in G}$ are pivotal   (see
Section~\ref{sect-pivotal-functor}).    Then  we have  for each $\alpha \in G$ a   monoidal   natural
isomorphism
$$
\varphi_\alpha^1=\{\varphi_\alpha^1(X) \co \varphi_\alpha(X^*) \iso \varphi_\alpha(X)^*\}_{X \in \cc}
$$
which  preserves both left and right duality.

 We shall use the crossing $\varphi$ to define the following transformations of morphisms in $\cc$.  Given
   an
 isomorphism $\psi\colon X \to \varphi_\alpha (Y)$ with $  X, Y\in \mathcal C$ and $\alpha\in G$,   we let
 $\overline \psi \colon Y\to \varphi_{\alpha^{-1}} (X)$ be the
 following composition of isomorphisms:
 \begin{equation*}\label{overlinepsi}
 \xymatrix@R=1cm @C=2cm { Y \ar[r]^-{(\varphi_0)_{Y}} & \varphi_{1} (Y) \ar[r]^-{\varphi_2 ( \alpha^{-1}, \alpha)_{Y}^{-1}} & \varphi_{\alpha^{-1}} \varphi_{\alpha} (Y) \ar[r]^-{ \varphi_{\alpha^{-1}} (\psi^{-1}) } &\varphi_{\alpha^{-1}} (X). }
\end{equation*}
If $\cc$ and $\varphi$ are pivotal, we  let    $  \psi^- \colon X^*\to
\varphi_{\alpha} (Y^*)$ be the following composition of
isomorphisms:
\begin{equation*}\label{psi-}
\xymatrix@R=1cm @C=1.7cm {  X^* \ar[r]^-{  (\psi^{-1})^* }     & \varphi_\alpha (Y)^*    \ar[r]^-{ \varphi_\alpha^1 (Y)^{-1} }     &
 \varphi_\alpha (Y^*). }
\end{equation*}
We also  sometimes write $\psi^+$ for $\psi$ itself.

\begin{lem}\label{lem-psibar}
In the above notation,  $\overline{\psi^-}=\bigl(\overline{\psi}\bigr)^-\co Y^*\to
\varphi_{\alpha^{-1}} (X^*)$.
\end{lem}
\begin{proof}
We begin with three observations  which are direct consequences of the definitions given in Section~\ref{sect-pivotal-functor}. Firstly, if $F\co \cc \to \dd$ is a pivotal functor  between pivotal categories and $f\co X \to Y$ is a morphism in $\cc$, then
\begin{equation}\label{eq-psibar1}
F(f)^*=F^1(X)F(f^*)F^1(Y)^{-1}.
\end{equation}
Secondly, if $F\co \cc \to \dd$ and $G\co \dd \to \ee$ are pivotal functors between pivotal categories, then for any $X \in \cc$,
\begin{equation}\label{eq-psibar2}
(GF)^1(X)=G^1(F(X)) \,  G(F^1(X)).
\end{equation}
Thirdly, if $F,G\co \cc \to \dd$ are pivotal functors  between pivotal categories, then any  monoidal natural transformation  $\lambda=\{\lambda_X\co F(X) \to G(X)\}_{X \in \cc}$ is  invertible and
\begin{equation}\label{eq-psibar3}
\lambda_X^*G^1(X)=F^1(X)\lambda_{X^*}^{-1}
\end{equation}
for all $X\in \cc$. Now,
\begin{eqnarray*}
\overline{\psi^-} & \overset{(i)}{=}
&  \varphi_{\alpha^{-1}}(\psi^*) \varphi_{\alpha^{-1}}(\varphi_\alpha^1(Y))\varphi_2(\alpha^{-1},\alpha)^{-1}_{Y^*} (\varphi_0)_{Y^*}\\
& \overset{(ii)}{=}
& \varphi_{\alpha^{-1}}^1(X)^{-1}  \varphi_{\alpha^{-1}}(\psi)^*  \varphi_{\alpha^{-1}}^1(\varphi_\alpha(Y)) \varphi_{\alpha^{-1}}(\varphi_\alpha^1(Y))\varphi_2(\alpha^{-1},\alpha)^{-1}_{Y^*} (\varphi_0)_{Y^*}\\
& \overset{(iii)}{=} &  \varphi_{\alpha^{-1}}^1(X)^{-1}  \varphi_{\alpha^{-1}}(\psi)^*  (\varphi_{\alpha^{-1}}\varphi_\alpha)^1(Y)\varphi_2(\alpha^{-1},\alpha)^{-1}_{Y^*} (\varphi_0)_{Y^*} \\
& \overset{(iv)}{=} & \varphi_{\alpha^{-1}}^1(X)^{-1}  \varphi_{\alpha^{-1}}(\psi)^* \varphi_2(\alpha^{-1},\alpha)^{*}_{Y} \varphi_1^1(Y) (\varphi_0)_{Y^*} \\
& \overset{(v)}{=} & \varphi_{\alpha^{-1}}^1(X)^{-1}  \varphi_{\alpha^{-1}}(\psi)^* \varphi_2(\alpha^{-1},\alpha)^{*}_{Y} ((\varphi_0)_{Y}^{-1})^*  \overset{(vi)}{=}   \bigl(\overline{\psi}\bigr)^-.
\end{eqnarray*}
Here, the equalities $(i)-(vi)$ follow respectively from the definition  of $\overline{\psi^-}$, \eqref{eq-psibar1}, \eqref{eq-psibar2},  \eqref{eq-psibar3} with $\lambda=\varphi_2(\alpha^{-1},\alpha)$,   \eqref{eq-psibar3} with $\lambda=\varphi_0^{-1}$,  the definition  of $\bigl(\overline{\psi}\bigr)^-$.
\end{proof}

\subsection{The fusion algebra}\label{sect-verlinde-algebra}
With every   $G$-graded category $\cc$ over $\kk$, one associates
a   $G$-graded  \kt algebra  $L({\mathcal  C})$  called the   {\it
fusion algebra} or  the {\it Verlinde algebra}  of~${\mathcal  C}$.  Specifically, for each
$\alpha\in G$, set $\widetilde L_\alpha=\oplus_{X } \,
\End_{\mathcal  C} (X)$ where $X$ runs over all objects of
${\mathcal  C}_\alpha$. The element of the \kt module $\widetilde
L_\alpha$ represented by $f\in \End_{\mathcal C} (X)$ is denoted
$\langle X, f \rangle$ or briefly $\langle f \rangle$.   Let
$L_\alpha$ be the quotient of $\widetilde L_\alpha$ by all relations
of type $\langle X, fg \rangle=\langle Y, gf \rangle$ for morphisms
$f\colon Y\to X$ and $g\colon X\to Y$ in ${\mathcal C}_\alpha$.   We
provide the \kt module  $L =\oplus_{\alpha\in G} L_\alpha$ with
multiplication $\langle f \rangle \,\langle f' \rangle= \langle
f\otimes f'\rangle$.
  This turns $L=L({\mathcal  C})$ into an associative  $G$-graded  \kt algebra with unit $\langle \id_{\un}
\rangle    $.   Every homogeneous object $X\in {\mathcal  C} $ determines a
vector $ \langle X \rangle =\langle \id_X \rangle\in L $.

If $\cc$ is crossed, then its crossing $\varphi$ induces a group homomorphism $\varphi\co G \to \mathrm{Aut}(L)$, where $\mathrm{Aut}(L)$ denotes the group of algebra automorphisms of $L$. For $\alpha \in G$, $\varphi_\alpha$  carries any generator $\langle f \rangle$ to $\langle \varphi_\alpha(f) \rangle$. Clearly, $\varphi_\alpha(L_\beta)\subset L_{\alpha^{-1}\beta\alpha}$.

If $\cc$ is pivotal, then $L$ is endowed with a canonical \kt linear endomorphism $\ast\co L \to L$ defined by sending a generator $\langle X, f \rangle$ to the generator $\langle X^*, f^*\rangle$.  One can prove
that $\ast$ is an involutive  algebra anti-automorphism of    $L$
carrying each $L_\alpha$ onto $L_{\alpha^{-1}}$.

  If $\cc$ is crossed and pivotal, and its crossing $\varphi$ is pivotal, then \eqref{eq-psibar1} implies that $\ast\co L \to L$ commutes with the crossing:  $\varphi_\alpha \ast=\ast \varphi_\alpha$ for all $\alpha \in G$.

\section{$G$-braided and $G$-ribbon categories}\label{sect-braided-ribbon}

Given a $G$-graded category $\cc$, we let $\cch=\amalg_{\alpha\in G}\,  \cc_\alpha$ be
the full subcategory of homogeneous objects of $\cc$, cf.\
Section~\ref{G-cat-deb}.   Note that $\cch$ is itself a $G$-graded category in the sense of Section~\ref{G-cat-deb} such that $(\cch)_\mathrm{hom}=\cch$.

\subsection{$G$-braided  categories}%
A \emph{$G$-braided  category} is a $G$-crossed category
$(\cc,\varphi)$ endowed with a \emph{$G$-braiding}, that is, a family
of isomorphisms $$\tau=\{\tau_{X,Y} \co X \otimes Y \to Y \otimes
\varphi_{|Y|}(X)\}_{X \in \cc, Y \in \cch}$$ natural in   $X$,
  $Y$   and such that:
\begin{enumerate}
\labela
  \item for all $X \in \cc$ and $Y,Z \in \cch$, the following diagram commutes:
\begin{equation}\label{eq-braiding1}
\begin{split}
\xymatrix@R=1cm @C=3cm {
X \otimes Y \otimes Z \ar[r]^-{\tau_{X, Y \otimes Z}}\ar[d]_{\tau_{X,Y} \otimes \id_Z} & Y \otimes Z \otimes \varphi_{|Y \otimes Z|}(X) \\
Y \otimes \varphi_{|Y|}(X) \otimes Z \ar[r]_-{\id_Y \otimes \tau_{\varphi_{|Y|}(X),Z}} & Y \otimes Z \otimes \varphi_{|Z|}\varphi_{|Y|}(X); \ar[u]_{\id_{Y \otimes Z} \otimes \varphi_2(|Z|,|Y|)_X}
}
\end{split}
\end{equation}
  \item for all $X,Y \in \cc$ and $Z \in \cch$, the following diagram commutes:
\begin{equation}\label{eq-braiding2}
\begin{split}
\xymatrix@R=1cm @C=3cm {
X \otimes Y \otimes Z \ar[r]^-{\tau_{X \otimes Y, Z}}\ar[d]_{\id_X\otimes \tau_{Y,Z}} & Z \otimes \varphi_{|Z|}(X \otimes Y) \\
X \otimes Z \otimes \varphi_{|Z|}(Y) \ar[r]_-{\tau_{X,Z} \otimes \id_{\varphi_{|Z|}(Y)}} & Z \otimes \varphi_{|Z|}(X) \otimes \varphi_{|Z|}(Y); \ar[u]_{\id_Z \otimes (\varphi_{|Z|})_2(X,Y)}
}
\end{split}
\end{equation}
   \item for all $\alpha \in G$, $X \in \cc$, and $Y \in \cch$, the following diagram commutes:
\begin{equation}\label{eq-braiding3}
\begin{split}
\xymatrix@R=1cm @C=3cm{
\varphi_\alpha(X) \otimes \varphi_\alpha(Y) \ar[r]^-{(\varphi_\alpha)_2(X,Y)}\ar[d]_{\tau_{\varphi_\alpha(X),\varphi_\alpha(Y)}} &\varphi_\alpha(X \otimes Y) \ar[d]^-{\varphi_\alpha(\tau_{X,Y})} \\
\varphi_\alpha(Y) \otimes \varphi_{\alpha^{-1}|Y|\alpha}\varphi_\alpha(X)\ar[d]^-{\id_{\varphi_\alpha(Y)} \otimes \varphi_2(\alpha^{-1}|Y|\alpha,\alpha)_X} & \varphi_\alpha(Y \otimes \varphi_{|Y|}(X)) \\
 \varphi_\alpha(Y) \otimes \varphi_{|Y|\alpha}(X) \ar[r]_-{\id_{\varphi_\alpha(Y)} \otimes \varphi_2(\alpha,|Y|)_X^{-1}} &\varphi_\alpha(Y) \otimes \varphi_{\alpha}\varphi_{|Y|}(X). \ar[u]_-{(\varphi_\alpha)_2(Y,\varphi_{|Y|}(X))} \\
}
\end{split}
\end{equation}
\end{enumerate}

The diagrams \eqref{eq-braiding1} and \eqref{eq-braiding2} are the analogues of the usual braiding relations in braided categories while \eqref{eq-braiding3} expresses the \lq\lq invariance" of $\tau$  under $\varphi$.
We will depict the $G$-braiding $\tau$ and its inverse as follows
\begin{center}
 \psfrag{U}[Br][Br]{\scalebox{.9}{$Y$}}
 \psfrag{V}[Bl][Bl]{\scalebox{.9}{$\varphi_{|Y|}(X)$}}
 \psfrag{A}[Br][Br]{\scalebox{.9}{$X$}}
 \psfrag{B}[Bl][Bl]{\scalebox{.9}{$Y$}}
 $\tau_{X,Y}=$\!\!\!\rsdraw{.45}{.9}{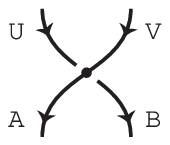}
 \quad \text{and} \quad
  \psfrag{U}[Br][Br]{\scalebox{.9}{$X$}}
 \psfrag{V}[Bl][Bl]{\scalebox{.9}{$Y$}}
 \psfrag{A}[Br][Br]{\scalebox{.9}{$Y$}}
 \psfrag{B}[Bl][Bl]{\scalebox{.9}{$\varphi_{|Y|}(X)$}}
$\tau^{-1}_{Y,X}=$\!\!\!\rsdraw{.45}{.9}{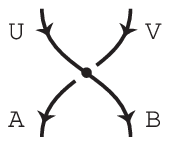}
\end{center}

\begin{lem} \label{lem-braiding}
Let $(\cc,\varphi,\tau)$ be a  $G$-braided  category. Then:
\begin{enumerate}
  \labela
     \item $\tau_{X, \un}=(\varphi_0)_X$ for all $X \in \cc$;
     \item $\tau_{\un,X}=\id_X \otimes (\varphi_{|X|})_0$ for all $X \in \cch$;
      \item The $G$-braiding $\tau$ satisfies the following quantum Yang-Baxter equation:
for all $X \in \cc$, $Y  \in \cc_\beta$, $Z\in \cc_\gamma$ with
$\beta,\gamma\in G$,
$$
\psfrag{X}[Br][Br]{\scalebox{.9}{$X$}}
\psfrag{Y}[Br][Br]{\scalebox{.9}{$Y$}}
\psfrag{Z}[Bl][Bl]{\scalebox{.9}{$Z$}}
\psfrag{R}[Br][Br]{\scalebox{.9}{$Z$}}
\psfrag{P}[Bl][Bl]{\scalebox{.9}{$\varphi_{\beta\gamma}(X)$}}
\psfrag{T}[Bl][Bl]{\scalebox{.9}{$\varphi_\gamma(Y)$}}
\psfrag{u}[Bc][Bc]{\scalebox{.9}{$\varphi_2(\gamma^{-1}\beta\gamma,\gamma)_X$}}
\psfrag{v}[Bc][Bc]{\scalebox{.9}{$\varphi_2(\gamma,\beta)_X$}}
\psfrag{e}[Bc][Bc]{\scalebox{1.11}{$=$}}
\rsdraw{.45}{.9}{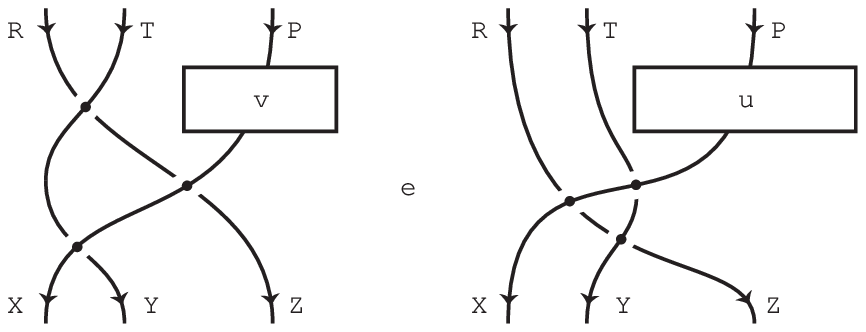} \;.
$$
\item If $\cc$ is pivotal, then for all $X \in \cc$ and $Y \in \cc_\beta$ with $\beta \in G$,
$$
\psfrag{X}[Bl][Bl]{\scalebox{.9}{$\varphi_\beta(X)$}}
\psfrag{Y}[Bl][Bl]{\scalebox{.9}{$Y$}}
\psfrag{R}[Bl][Bl]{\scalebox{.9}{$X$}}
\psfrag{u}[Bc][Bc]{\scalebox{.9}{$\varphi_{\beta}^l(X)$}}
\psfrag{v}[Bc][Bc]{\scalebox{.9}{$(\varphi_0)_X^{-1}\varphi_2(\beta^{-1},\beta)_X$}}
\tau_{X,Y}^{-1}=\;\,\rsdraw{.45}{.9}{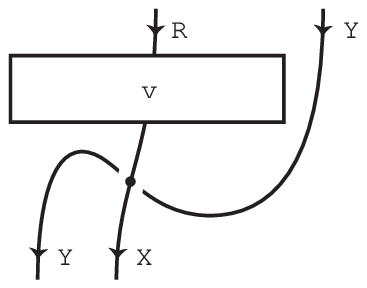} \!= \;\;\rsdraw{.45}{.9}{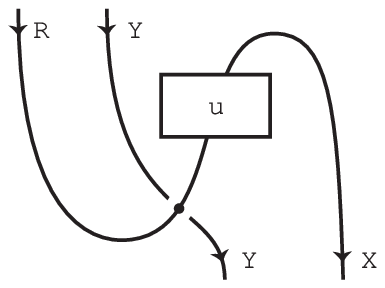} \;.
$$
\end{enumerate}
\end{lem}
\begin{proof}
By the naturality of $\tau$ and $\varphi_0$, if (a) holds for  some
$X\in \cc$, then it holds for all objects of $\cc$  isomorphic to
$X$. Since $\varphi_1$ is an auto-equivalence of $\cc$, it is enough
to check (a) for the objects  of type  $\varphi_1(X)$ where  $X \in
\cc$. Setting $Y=Z=\un$ in \eqref{eq-braiding1} we obtain $\tau_{X,
\un}=\varphi_2(1,1)_X\tau_{\varphi_1(X),\un} \tau_{X, \un}$.
Multiplying by $\tau_{X, \un}^{-1}$ and using \eqref{crossing6}, we
obtain  that $\tau_{\varphi_1(X),
\un}=(\varphi_2(1,1)_X)^{-1}=(\varphi_0)_{\varphi_1(X)}$.

Let us prove (b). Since $\tau_{\un,Z}$ is an isomorphism, putting $X=Y=\un$ in \eqref{eq-braiding2} gives $\id_{Z \otimes \varphi_{|Z|}(\un)}=(\id_Z \otimes (\varphi_{|Z|})_2(\un,\un))(\tau_{\un,Z} \otimes \id_{\varphi_{|Z|}(\un)})$. We conclude by composing on the right with $\id_Z \otimes (\varphi_{|Z|})_0$, since $(\varphi_{|Z|})_2(\un,\un)(\id_{\varphi_{|Z|}(\un)} \otimes (\varphi_{|Z|})_0)=\id_{\varphi_{|Z|}(\un)}$ by \eqref{crossing2}.

Let us prove (c). The naturality of the $G$-braiding gives
$$
\tau_{Y\otimes \varphi_{|Y|}(X),Z}(\tau_{X,Y} \otimes \id_Z)=\bigl(\id_Z \otimes \varphi_{|Z|}(\tau_{X,Y})\bigr)\tau_{X\otimes Y,Z}.
$$
We use \eqref{eq-braiding2} to expand $\tau_{X \otimes Y,Z}$ on the right and $\tau_{Y\otimes \varphi_{|Y|}(X),Z}$ on the left. Then we use \eqref{eq-braiding3} to replace on the right $\varphi_{|Z|}(\tau_{X,Y})(\varphi_{|Z|})_2(X,Y)$ by a composition of 4 arrows. This gives a formula equivalent to (c).

We now prove (d). Depicting $\id_\un$ by a dotted line, we  obtain
$$
\psfrag{X}[Bl][Bl]{\scalebox{.9}{$X$}}
\psfrag{Y}[Bl][Bl]{\scalebox{.9}{$Y$}}
\rsdraw{.25}{.9}{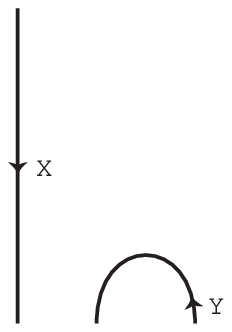}
\overset{(i)}{=}\,
\psfrag{X}[Bl][Bl]{\scalebox{.9}{$X$}}
\psfrag{Y}[Bl][Bl]{\scalebox{.9}{$Y$}}
\psfrag{v}[Bc][Bc]{\scalebox{.9}{$(\varphi_0)_X^{-1}$}}
\rsdraw{.25}{.9}{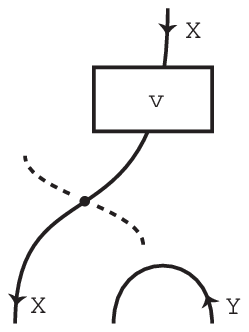}
\overset{(ii)}{=}
\psfrag{X}[Bl][Bl]{\scalebox{.9}{$X$}}
\psfrag{Y}[Bl][Bl]{\scalebox{.9}{$Y$}}
\psfrag{u}[Bc][Bc]{\scalebox{.9}{$\id_{Y \otimes Y^*}$}}
\psfrag{v}[Bc][Bc]{\scalebox{.9}{$(\varphi_0)_X^{-1}$}}
\rsdraw{.25}{.9}{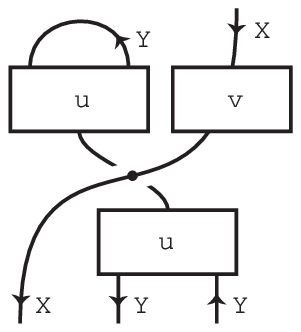}
\overset{(iii)}{=}\;
\psfrag{X}[Bl][Bl]{\scalebox{.9}{$X$}}
\psfrag{Y}[Bl][Bl]{\scalebox{.9}{$Y$}}
\psfrag{Z}[Br][Br]{\scalebox{.9}{$Y$}}
\psfrag{v}[Bc][Bc]{\scalebox{.9}{$(\varphi_0)_X^{-1}\varphi_2(\beta^{-1},\beta)_X$}}
\rsdraw{.25}{.9}{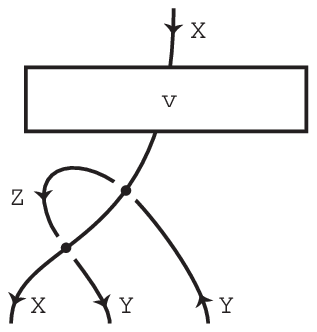},
$$
where the equality $(i)$ is obtained by applying (a), $(ii)$ follows from the naturality of $\tau$, and $(iii)$  is obtained by applying \eqref{eq-braiding1} to $Z=Y^*$. Substituting the resulting   equality in the   dotted box below, we obtain
$$
\psfrag{X}[Bl][Bl]{\scalebox{.9}{$\varphi_\beta(X)$}}
\psfrag{Y}[Bl][Bl]{\scalebox{.9}{$Y$}}
\psfrag{R}[Bl][Bl]{\scalebox{.9}{$X$}}
\rsdraw{.35}{.9}{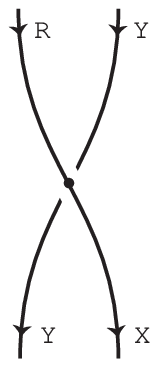}
\;\;=\;
\psfrag{X}[Bl][Bl]{\scalebox{.9}{$\varphi_\beta(X)$}}
\psfrag{Y}[Bl][Bl]{\scalebox{.9}{$Y$}}
\psfrag{R}[Bl][Bl]{\scalebox{.9}{$X$}}
\rsdraw{.35}{.9}{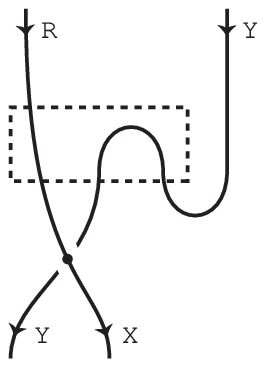}
=
\psfrag{X}[Bl][Bl]{\scalebox{.9}{$\varphi_\beta(X)$}}
\psfrag{R}[Bl][Bl]{\scalebox{.9}{$X$}}
\psfrag{Z}[Br][Br]{\scalebox{.9}{$Y$}}
\psfrag{v}[Bc][Bc]{\scalebox{.9}{$(\varphi_0)_X^{-1}\varphi_2(\beta^{-1},\beta)_X$}}
\rsdraw{.35}{.9}{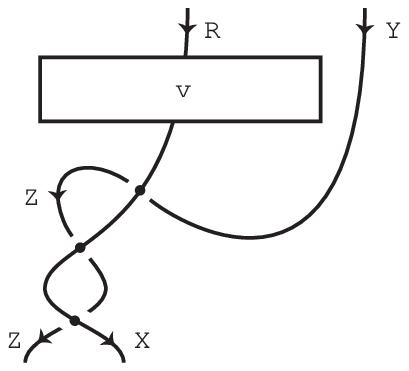}
\!\!\!\!\!=
\psfrag{X}[Bl][Bl]{\scalebox{.9}{$\varphi_\beta(X)$}}
\psfrag{Y}[Bl][Bl]{\scalebox{.9}{$Y$}}
\psfrag{R}[Bl][Bl]{\scalebox{.9}{$X$}}
\psfrag{u}[Bc][Bc]{\scalebox{.9}{$\varphi_{\beta}^l(X)$}}
\psfrag{v}[Bc][Bc]{\scalebox{.9}{$(\varphi_0)_X^{-1}\varphi_2(\beta^{-1},\beta)_X$}}
\rsdraw{.35}{.9}{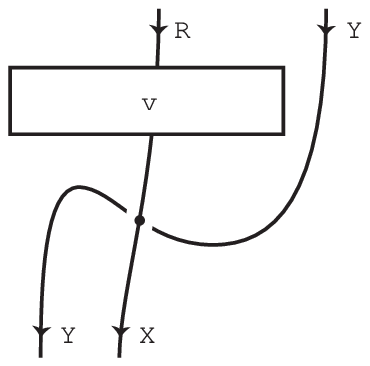}\!\!.
$$
Similarly,
\begin{center}
$ \displaystyle
\psfrag{X}[Bl][Bl]{\scalebox{.9}{$X$}}
\psfrag{Y}[Bl][Bl]{\scalebox{.9}{$Y$}}
\rsdraw{.25}{.9}{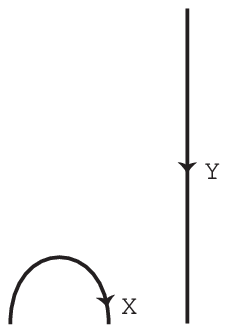}
\;\overset{(i)}{=}\;
\psfrag{X}[Bl][Bl]{\scalebox{.9}{$X$}}
\psfrag{Y}[Bl][Bl]{\scalebox{.9}{$Y$}}
\psfrag{v}[Bc][Bc]{\scalebox{.9}{$(\varphi_\beta)_0^{-1}$}}
\rsdraw{.25}{.9}{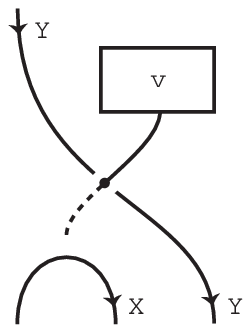}
\;\overset{(ii)}{=}\;
\psfrag{X}[Bl][Bl]{\scalebox{.9}{$X$}}
\psfrag{Y}[Bl][Bl]{\scalebox{.9}{$Y$}}
\psfrag{u}[Bc][Bc]{\scalebox{.9}{$\id_{X^* \otimes X}$}}
\psfrag{v}[Bc][Bc]{\scalebox{.9}{$(\varphi_\beta)_0^{-1}\varphi_\beta(\lev_X)$}}
\rsdraw{.25}{.9}{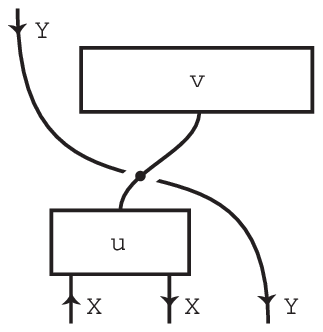}
$\\[1em]
$ \displaystyle
\;\overset{(iii)}{=}\;
\psfrag{X}[Bl][Bl]{\scalebox{.9}{$X$}}
\psfrag{Y}[Bl][Bl]{\scalebox{.9}{$Y$}}
\psfrag{v}[Bc][Bc]{\scalebox{.9}{$(\varphi_\beta)_0^{-1}\varphi_\beta(\lev_X)(\varphi_\beta)_2(X^*,X)$}}
\rsdraw{.25}{.9}{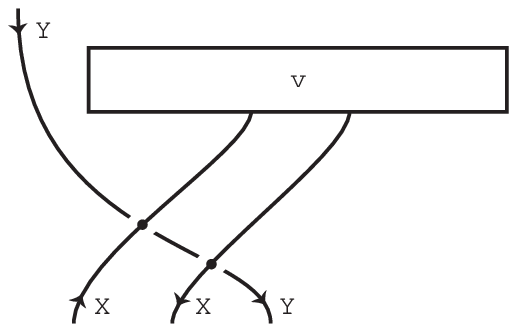}
\;\overset{(iv)}{=}\;
\psfrag{X}[Bl][Bl]{\scalebox{.9}{$X$}}
\psfrag{Y}[Bl][Bl]{\scalebox{.9}{$Y$}}
\psfrag{v}[Bc][Bc]{\scalebox{.9}{$\varphi_\beta^l(X)$}}
\rsdraw{.25}{.9}{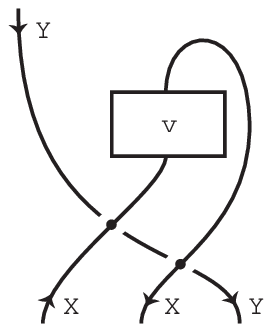}
$
\end{center}\vspace{1em}
where  $(i)$ is obtained by applying (a),  $(ii)$ follows from the naturality of $\tau$, $(iii)$  follows from \eqref{eq-braiding2}, and $(iv)$ is obtained by applying \eqref{pivotal-lev} to $F=\varphi_\beta$. Finally, substituting the resulting   equality in the   dotted box below, we obtain
$$
\psfrag{X}[Bl][Bl]{\scalebox{.9}{$\varphi_\beta(X)$}}
\psfrag{Y}[Bl][Bl]{\scalebox{.9}{$Y$}}
\psfrag{R}[Bl][Bl]{\scalebox{.9}{$X$}}
\rsdraw{.35}{.9}{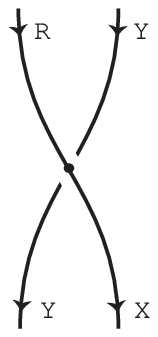}
\!\!=
\psfrag{X}[Bl][Bl]{\scalebox{.9}{$\varphi_\beta(X)$}}
\psfrag{Y}[Bl][Bl]{\scalebox{.9}{$Y$}}
\psfrag{R}[Bl][Bl]{\scalebox{.9}{$X$}}
\rsdraw{.35}{.9}{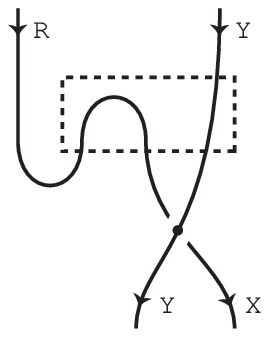}
\!=
\psfrag{X}[Bl][Bl]{\scalebox{.9}{$\varphi_\beta(X)$}}
\psfrag{R}[Bl][Bl]{\scalebox{.9}{$X$}}
\psfrag{Z}[Br][Br]{\scalebox{.9}{$Y$}}
\psfrag{v}[Bc][Bc]{\scalebox{.9}{$\varphi_\beta^l(X)$}}
\rsdraw{.35}{.9}{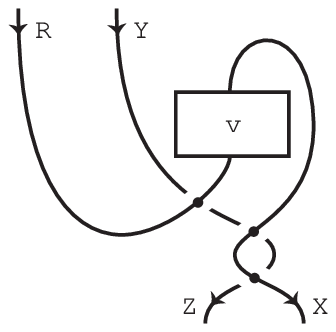}
\!\!\!\!=\!
\psfrag{X}[Bl][Bl]{\scalebox{.9}{$\varphi_\beta(X)$}}
\psfrag{R}[Bl][Bl]{\scalebox{.9}{$X$}}
\psfrag{Z}[Br][Br]{\scalebox{.9}{$Y$}}
\psfrag{v}[Bc][Bc]{\scalebox{.9}{$\varphi_\beta^l(X)$}}
\rsdraw{.35}{.9}{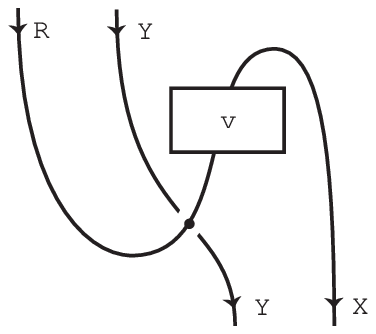}\!\!.
$$
This concludes the proof of the lemma.
%
\end{proof}

\subsection{Twists}
Consider a   pivotal   $G$-braided  category $(\cc,\varphi,\tau)$. The \emph{twist} of $\cc$ is the family
of morphisms $\theta=\{\theta_X \co X\to \varphi_{|X|}(X)\}_{X \in
\cch}$ defined by
\begin{equation}\label{eq-def-twist}
\theta_X =(\lev_X \otimes \id_{\varphi_{|X|}(X)})(\id_{X^*} \otimes \tau_{X,X})(\rcoev_X \otimes \id_X)
\psfrag{B}[Bl][Bl]{\scalebox{.9}{$\varphi_{|X|}(X)$}}
\psfrag{X}[Bl][Bl]{\scalebox{.9}{$X$}}
=\,\rsdraw{.45}{.9}{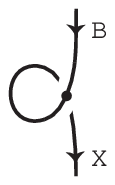}\;.
\end{equation}
The naturality of $\tau$ implies that   $ \theta_X $ is natural
in $X$.  Lemma~\ref{lem-braiding}(a) implies that $\theta_\un=(\varphi_0)_\un=(\varphi_{1})_0$.

\begin{lem}\label{lem-twist}
For any $X \in \cc_\alpha$ with $\alpha\in G$, the  twist  $\theta_X$ is invertible and
\begin{equation*}
\theta_X^{-1}=( \id_X \otimes\rev_{\varphi_{\alpha}(X)})(\tau_{X,\varphi_{\alpha}(X)}^{-1} \otimes \id_{\varphi_{\alpha}(X)^*})( \id_{\varphi_{\alpha}(X)} \otimes \lcoev_{\varphi_{\alpha}(X)})
\psfrag{X}[Bl][Bl]{\scalebox{.9}{$\varphi_{\alpha}(X)$}}
\psfrag{B}[Bl][Bl]{\scalebox{.9}{$X$}}
=\,\rsdraw{.45}{.9}{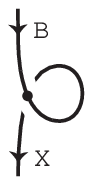}\;\;.
\end{equation*}
\end{lem}
\begin{proof}
Denote the right-hand side by $\vartheta_X$. Firstly,
\begin{center}
$ \displaystyle
\vartheta_X\theta_X \;
\overset{(i)}{=}\;
\psfrag{X}[Bl][Bl]{\scalebox{.9}{$X$}}
\psfrag{u}[Bc][Bc]{\scalebox{.9}{$\theta_X$}}
\psfrag{v}[Bc][Bc]{\scalebox{.9}{$(\varphi_0)_X^{-1}\varphi_2(\alpha^{-1},\alpha)_X$}}
\rsdraw{.35}{.9}{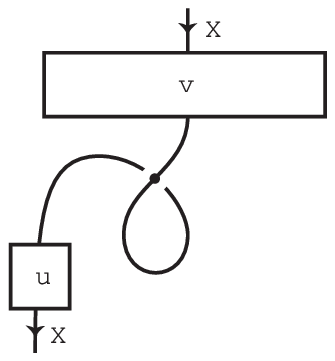}
\overset{(ii)}{=}\;
\psfrag{X}[Bl][Bl]{\scalebox{.9}{$X$}}
\psfrag{u}[Bc][Bc]{\scalebox{.9}{$\theta_X$}}
\psfrag{v}[Bc][Bc]{\scalebox{.9}{$(\varphi_0)_X^{-1}\varphi_2(\alpha^{-1},\alpha)_X$}}
\rsdraw{.35}{.9}{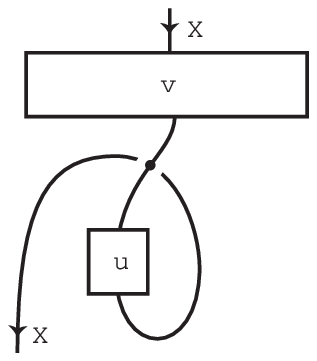}
\overset{(iii)}{=}\;
\psfrag{X}[Bl][Bl]{\scalebox{.9}{$X$}}
\psfrag{v}[Bc][Bc]{\scalebox{.9}{$(\varphi_0)_X^{-1}\varphi_2(\alpha^{-1},\alpha)_X$}}
\rsdraw{.35}{.9}{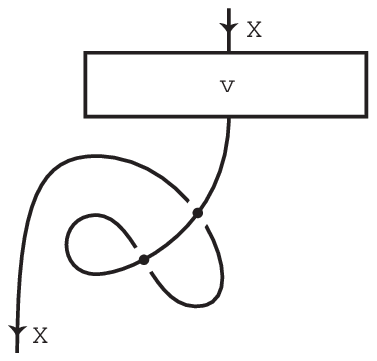}
$\\[.8em]
$ \displaystyle \overset{(iv)}{=}\;
\psfrag{X}[Bl][Bl]{\scalebox{.9}{$X$}}
\psfrag{u}[Bc][Bc]{\scalebox{.9}{$\id_{X \otimes X^*}$}}
\psfrag{v}[Bc][Bc]{\scalebox{.9}{$(\varphi_0)_X^{-1}$}}
\rsdraw{.45}{.9}{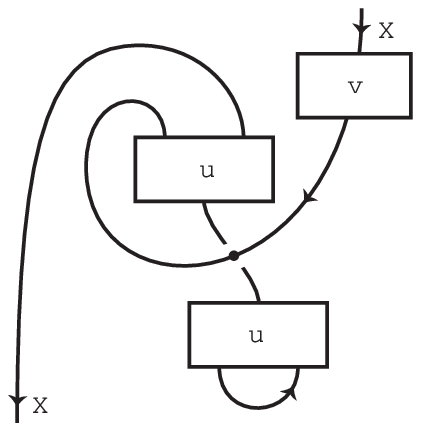}
\overset{(v)}{=}\;
\psfrag{X}[Bl][Bl]{\scalebox{.9}{$X$}}
\psfrag{v}[Bc][Bc]{\scalebox{.9}{$(\varphi_0)_X^{-1}$}}
\rsdraw{.45}{.9}{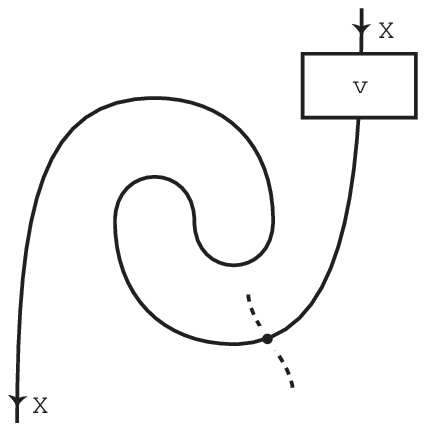}
\overset{(vi)}{=}\; \id_X.
$
\end{center}\vspace{.8em}
\noindent Here, $(i)$ is obtained from the definition of $\vartheta_X$ by applying the first expression for $\tau^{-1}$   in Lemma~\ref{lem-braiding}(d),  $(ii)$ and $(v)$ follow  from the naturality of $\tau$ and $(iii)$ from the definition of $\theta_X$, $(iv)$ is obtained by applying \eqref{eq-braiding1}, and $(vi)$ follows from Lemma~\ref{lem-braiding}(a). Secondly,
\begin{center}
$ \displaystyle
\theta_X\vartheta_X \;
=\;
\psfrag{X}[Bl][Bl]{\scalebox{.9}{$\varphi_\alpha(X)$}}
\psfrag{u}[Bc][Bc]{\scalebox{.9}{$\vartheta_X$}}
\rsdraw{.35}{.9}{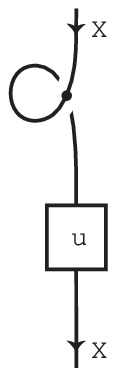}
\;\;\;\overset{(i)}{=}\;\;\;
\psfrag{X}[Bl][Bl]{\scalebox{.9}{$\varphi_\alpha(X)$}}
\psfrag{u}[Bc][Bc]{\scalebox{.9}{$\vartheta_X$}}
\rsdraw{.35}{.9}{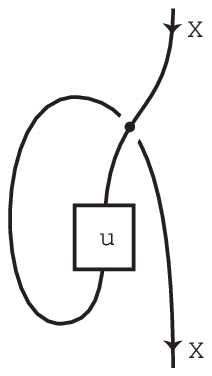}
\overset{(ii)}{=}\;
\psfrag{X}[Bl][Bl]{\scalebox{.9}{$\varphi_\alpha(X)$}}
\psfrag{v}[Bc][Bc]{\scalebox{.9}{$(\varphi_0)^{-1}_X\varphi_2(\alpha^{-1},\alpha)_X$}}
\rsdraw{.35}{.9}{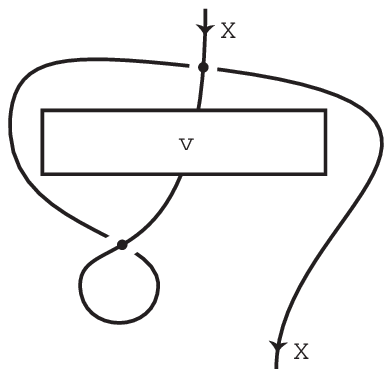}
$\\[.8em]
$ \displaystyle
\overset{(iii)}{=}
\psfrag{X}[Bl][Bl]{\scalebox{.9}{$\varphi_\alpha(X)$}}
\psfrag{v}[Bc][Bc]{\scalebox{.9}{$(\varphi_0)^{-1}_{\varphi_\alpha(X)}\varphi_2(\alpha,\alpha^{-1})_{\varphi_\alpha(X)}$}}
\rsdraw{.25}{.9}{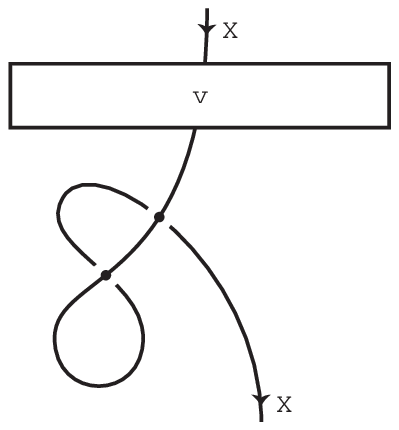}
\!\!\!\!\!\!\!\!\!\!\overset{(iv)}{=}
\psfrag{X}[Bl][Bl]{\scalebox{.9}{$\varphi_\alpha(X)$}}
\psfrag{u}[Bc][Bc]{\scalebox{.9}{$\id_{\varphi_\alpha(X)^* \otimes \varphi_\alpha(X)}$}}
\psfrag{v}[Bc][Bc]{\scalebox{.9}{$(\varphi_0)^{-1}_{\varphi_\alpha(X)}$}}
\rsdraw{.25}{.9}{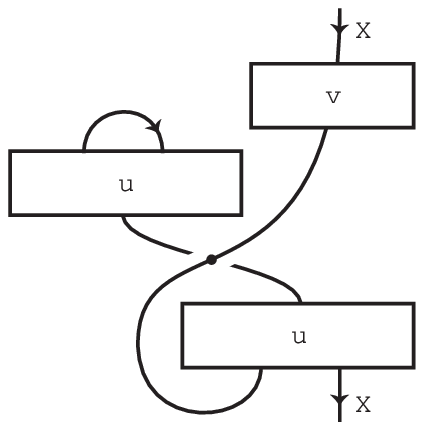}
\;\;\overset{(v)}{=}\;\,
\psfrag{X}[Bl][Bl]{\scalebox{.9}{$\varphi_\alpha(X)$}}
\rsdraw{.25}{.9}{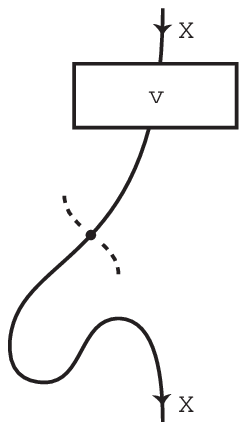}
\!\!\!\overset{(vi)}{=}\; \id_X.
$
\end{center}\vspace{.8em}
Here, $(i)$ and $(v)$ follow  from the naturality of $\tau$,  $(ii)$ is obtained from the definition of $\vartheta_X$ by applying the first expression for   $\tau^{-1}$ in Lemma~\ref{lem-braiding}(d),  $(iii)$ follows from the naturality of $\tau$ and the equality
\begin{equation}\label{eq-dem-twist}
\varphi_\alpha\bigl ( (\varphi_0)^{-1}_X\varphi_2(\alpha^{-1},\alpha)_X \bigr )= (\varphi_0)^{-1}_{\varphi_\alpha(X)}\varphi_2(\alpha,\alpha^{-1})_{\varphi_\alpha(X)}
\end{equation}
which is a consequence of \eqref{crossing5} and \eqref{crossing6}, $(iv)$ is obtained by applying \eqref{eq-braiding1}, and $(vi)$ follows from Lemma~\ref{lem-braiding}(a).
Thus $\theta_X$ is invertible and
$\theta_X^{-1}=\vartheta_X$.
\end{proof}

\begin{lem}\label{lem-twist-mult}
For any $X \in \cc_\alpha$, $Y \in \cc_\beta$  with $\alpha, \beta \in G$,
\begin{equation*}
\psfrag{e}[Bc][Bc]{\scalebox{.9}{$\theta_X$}}
\psfrag{a}[Bc][Bc]{\scalebox{.9}{$\theta_Y$}}
\psfrag{u}[Bc][Bc]{\scalebox{.9}{$\varphi_2(\beta^{-1}\alpha\beta,\beta)_Y$}}
\psfrag{v}[Bc][Bc]{\scalebox{.9}{$\varphi_2(\beta,\alpha)_X$}}
\psfrag{s}[Bc][Bc]{\scalebox{.9}{$(\varphi_{\alpha\beta})_2(X,Y)$}}
\psfrag{T}[Bl][Bl]{\scalebox{.9}{$\varphi_{\alpha\beta}(X\otimes Y)$}}
\psfrag{Y}[Bl][Bl]{\scalebox{.9}{$Y$}}
\psfrag{X}[Br][Br]{\scalebox{.9}{$X$}}
\theta_{X\otimes Y}=\;\rsdraw{.45}{.9}{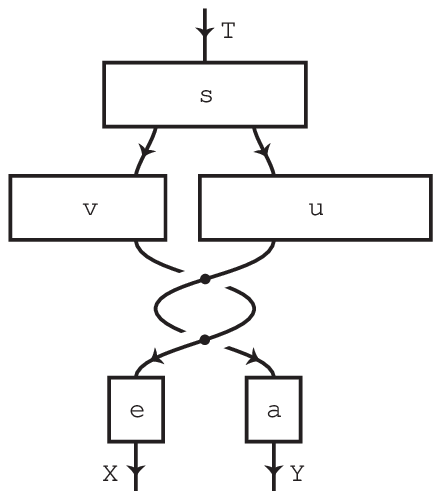}\;\;.
\end{equation*}
\end{lem}
\begin{proof}
  We have
\begin{center}
$ \displaystyle
\theta_{X\otimes Y}
\;\;\overset{(i)}{=}\;\;
\psfrag{T}[Bl][Bl]{\scalebox{.9}{$\varphi_{\alpha\beta}(X\otimes Y)$}}
\psfrag{Y}[Bl][Bl]{\scalebox{.9}{$X \otimes Y$}}
\rsdraw{.45}{.9}{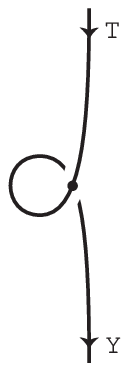}
\overset{(ii)}{=}\;
\psfrag{u}[Bc][Bc]{\scalebox{.9}{$\id_{X \otimes Y}$}}
\psfrag{s}[Bc][Bc]{\scalebox{.9}{$(\varphi_{\alpha\beta})_2(X,Y)$}}
\psfrag{T}[Bl][Bl]{\scalebox{.9}{$\varphi_{\alpha\beta}(X\otimes Y)$}}
\psfrag{Y}[Bl][Bl]{\scalebox{.9}{$X \otimes Y$}}
\rsdraw{.45}{.9}{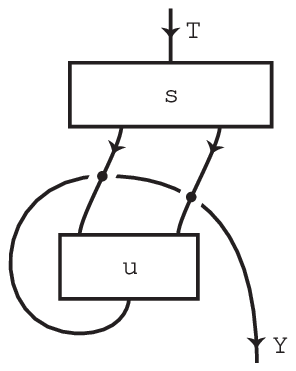}
$\\[1.2em]
$ \displaystyle
\overset{(iii)}{=}\;\;
\psfrag{e}[Bc][Bc]{\scalebox{.9}{$\theta_X$}}
\psfrag{a}[Bc][Bc]{\scalebox{.9}{$\theta_Y$}}
\psfrag{u}[Bc][Bc]{\scalebox{.9}{$\varphi_2(\beta,\alpha)_Y$}}
\psfrag{v}[Bc][Bc]{\scalebox{.9}{$\varphi_2(\beta,\alpha)_X$}}
\psfrag{s}[Bc][Bc]{\scalebox{.9}{$(\varphi_{\alpha\beta})_2(X,Y)$}}
\psfrag{T}[Bl][Bl]{\scalebox{.9}{$\varphi_{\alpha\beta}(X\otimes Y)$}}
\psfrag{Y}[Bl][Bl]{\scalebox{.9}{$Y$}}
\psfrag{X}[Br][Br]{\scalebox{.9}{$X$}}
\rsdraw{.45}{.9}{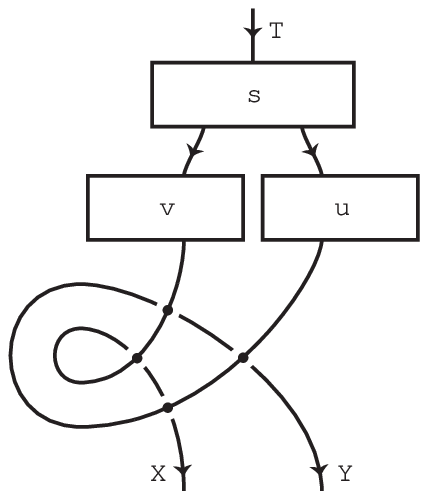}
\;\;\overset{(iv)}{=}\;\;
\psfrag{e}[Bc][Bc]{\scalebox{.9}{$\theta_X$}}
\psfrag{a}[Bc][Bc]{\scalebox{.9}{$\theta_Y$}}
\psfrag{u}[Bc][Bc]{\scalebox{.9}{$\varphi_2(\beta,\alpha)_Y$}}
\psfrag{v}[Bc][Bc]{\scalebox{.9}{$\varphi_2(\beta,\alpha)_X$}}
\psfrag{s}[Bc][Bc]{\scalebox{.9}{$(\varphi_{\alpha\beta})_2(X,Y)$}}
\psfrag{T}[Bl][Bl]{\scalebox{.9}{$\varphi_{\alpha\beta}(X\otimes Y)$}}
\psfrag{Y}[Bl][Bl]{\scalebox{.9}{$Y$}}
\psfrag{X}[Br][Br]{\scalebox{.9}{$X$}}
\rsdraw{.45}{.9}{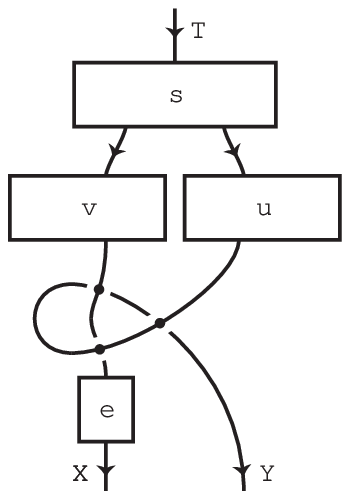}
$\\[1.2em]
$ \displaystyle
\overset{(v)}{=}\;\;
\psfrag{e}[Bc][Bc]{\scalebox{.9}{$\theta_X$}}
\psfrag{a}[Bc][Bc]{\scalebox{.9}{$\theta_Y$}}
\psfrag{u}[Bc][Bc]{\scalebox{.9}{$\varphi_2(\beta^{-1}\alpha\beta,\beta)_Y$}}
\psfrag{v}[Bc][Bc]{\scalebox{.9}{$\varphi_2(\beta,\alpha)_X$}}
\psfrag{s}[Bc][Bc]{\scalebox{.9}{$(\varphi_{\alpha\beta})_2(X,Y)$}}
\psfrag{T}[Bl][Bl]{\scalebox{.9}{$\varphi_{\alpha\beta}(X\otimes Y)$}}
\psfrag{Y}[Bl][Bl]{\scalebox{.9}{$Y$}}
\psfrag{X}[Br][Br]{\scalebox{.9}{$X$}}
\rsdraw{.45}{.9}{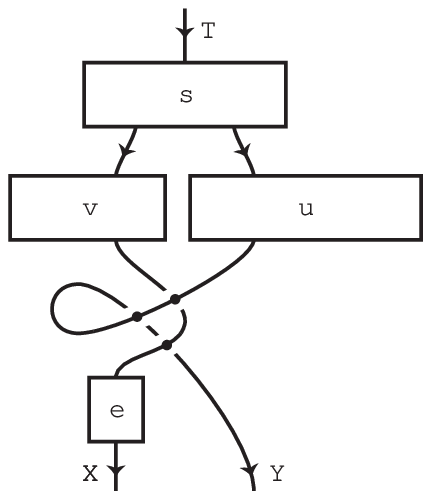}
\;\;\overset{(vi)}{=}\;\;
\psfrag{e}[Bc][Bc]{\scalebox{.9}{$\theta_X$}}
\psfrag{a}[Bc][Bc]{\scalebox{.9}{$\theta_Y$}}
\psfrag{u}[Bc][Bc]{\scalebox{.9}{$\varphi_2(\beta^{-1}\alpha\beta,\beta)_Y$}}
\psfrag{v}[Bc][Bc]{\scalebox{.9}{$\varphi_2(\beta,\alpha)_X$}}
\psfrag{s}[Bc][Bc]{\scalebox{.9}{$(\varphi_{\alpha\beta})_2(X,Y)$}}
\psfrag{T}[Bl][Bl]{\scalebox{.9}{$\varphi_{\alpha\beta}(X\otimes Y)$}}
\psfrag{Y}[Bl][Bl]{\scalebox{.9}{$Y$}}
\psfrag{X}[Br][Br]{\scalebox{.9}{$X$}}
\rsdraw{.45}{.9}{theta-def3}\;\;.
$
\end{center}
Here, $(i)$ follows from the definition of $\theta_{X \otimes Y}$, $(ii)$ is obtained from \eqref{eq-braiding2}, $(iii)$ is obtained by applying \eqref{eq-braiding1} twice, $(iv)$ follows from the definition of $\theta_X$ and the naturality of $\tau$, $(v)$ is obtained by Lemma~\ref{lem-braiding}(c), and $(vi)$ follows from the definition of $\theta_Y$ and the naturality of $\tau$.
\end{proof}

\begin{lem}\label{lem-twist-2}
If the crossing $\varphi$ in   $  \cc $    is pivotal, then   for  all $\alpha, \beta \in G$ and   $X \in \cc_\alpha$,
$$
\varphi_\beta(\theta_X)= \varphi_2(\beta,\alpha)_X^{-1} \varphi_2(\beta^{-1}\alpha\beta,\beta)_X \theta_{\varphi_\beta(X)}.
$$

\end{lem}
\begin{proof}
Set $\psi=\varphi_2(\beta,\alpha)_X^{-1} \varphi_2(\beta^{-1}\alpha\beta,\beta)_X$.   Then
$$
\varphi_\beta(\theta_X)\;
\overset{(i)}{=}\;
\psfrag{Y}[Bl][Bl]{\scalebox{.9}{$\varphi_\beta\varphi_\alpha(X)$}}
\psfrag{X}[Bl][Bl]{\scalebox{.9}{$\varphi_\beta(X)$}}
\psfrag{a}[Bc][Bc]{\scalebox{.9}{$(\varphi_\beta)_2(\un,\varphi_\alpha(X))$}}
\psfrag{c}[Bc][Bc]{\scalebox{.9}{$\varphi_\beta(\lev_X)$}}
\psfrag{u}[Bc][Bc]{\scalebox{.9}{$(\varphi_\beta)_2(X^*,X)$}}
\psfrag{v}[Bc][Bc]{\scalebox{.9}{$(\varphi_\beta)_2(X,\varphi_\alpha(X))^{-1}$}}
\psfrag{e}[Bc][Bc]{\scalebox{.9}{$\varphi_\beta(\tau_{X,X})$}}
\psfrag{n}[Bc][Bc]{\scalebox{.9}{$(\varphi_\beta)_2(X,X)$}}
\psfrag{s}[Bc][Bc]{\scalebox{.9}{$(\varphi_\beta)_2(X^*,X)^{-1}$}}
\psfrag{z}[Bc][Bc]{\scalebox{.9}{$\varphi_\beta(\rcoev_X)$}}
\psfrag{d}[Bc][Bc]{\scalebox{.9}{$(\varphi_\beta)_2(\un,X)^{-1}$}}
\rsdraw{.35}{.9}{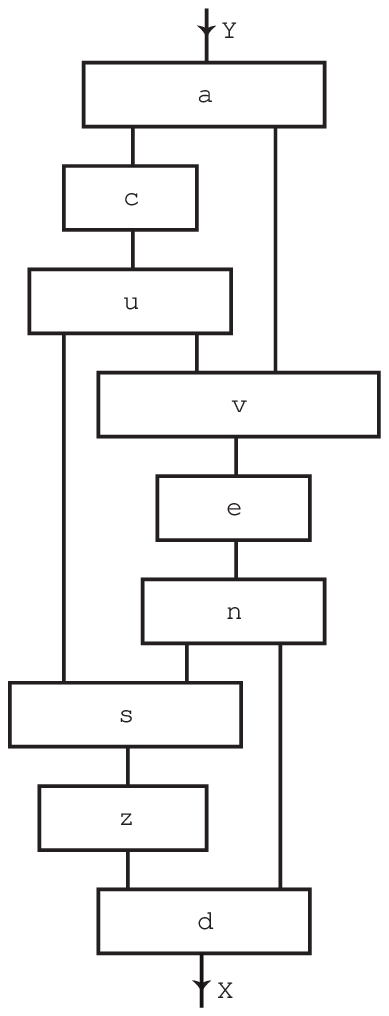}
\overset{(ii)}{=}\;\;\;
\psfrag{Y}[Bl][Bl]{\scalebox{.9}{$\varphi_\beta\varphi_\alpha(X)$}}
\psfrag{X}[Bl][Bl]{\scalebox{.9}{$\varphi_\beta(X)$}}
\psfrag{a}[Bc][Bc]{\scalebox{.9}{$(\varphi_\beta)_2(\un,\varphi_\alpha(X))$}}
\psfrag{c}[Bc][Bc]{\scalebox{.9}{$(\varphi_\beta)_0$}}
\psfrag{u}[Bc][Bc]{\scalebox{.9}{$\psi$}}
\psfrag{e}[Bc][Bc]{\scalebox{.9}{$\varphi_\beta^1(X)$}}
\psfrag{n}[Bc][Bc]{\scalebox{.9}{$\varphi_\beta^1(X)^{-1}$}}
\psfrag{z}[Bc][Bc]{\scalebox{.9}{$(\varphi_\beta)_0^{-1}$}}
\psfrag{d}[Bc][Bc]{\scalebox{.9}{$(\varphi_\beta)_2(\un,X)^{-1}$}}
\rsdraw{.35}{.9}{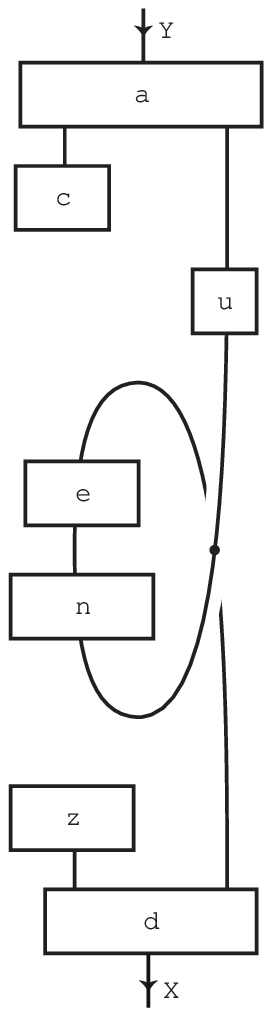}
\overset{(iii)}{=}\; \psi \, \theta_{\varphi_\beta(X)}.
$$
Here, $(i)$ is obtained by writing
$$
\theta_X=(\lev_X \otimes \id_{\varphi_\alpha(X)})(\id_{X^*} \otimes \tau_{X,X})(\rcoev_X \otimes \id_X)
$$
and  applying the monoidality of  $\varphi_\beta$, $(ii)$ is obtained by applying \eqref{eq-braiding3}, \eqref{pivotal-lev},
 and \eqref{pivotal-rcoev}, and $(iii)$ follows from \eqref{crossing2}.
\end{proof}

\subsection{$G$-ribbon categories}\label{sect-rib-def}
A \emph{$G$-ribbon  category} is a   pivotal   $G$-braided  category $\cc$   such that   its crossing   $\varphi$ is   pivotal and    its  twist $\theta$  is
\emph{self-dual} in the sense that for all $\alpha\in G$ and all $X \in \cc_\alpha$,
\begin{equation}\label{eq-ribbon}
(\theta_X)^*= (\varphi_0)_X^* (\varphi_2(\alpha^{-1},\alpha)_X^{-1})^* \varphi_{\alpha^{-1}}^1(\varphi_{\alpha}(X)) \theta_{\varphi_{\alpha}(X)^*}.
\end{equation}

For examples of $G$-ribbon categories, see \cite{Tu1}, \cite{TVi3}. The following   lemmas yield a useful consequence of self-duality for twists.

\begin{lem}\label{lem-twist-ribbon}
If  the twist $\theta$ in a  pivotal   $G$-braided  category $\cc$ is self-dual, then  for all   $\alpha\in G$ and   $X \in \cc_\alpha$,
$$
\theta_X
\psfrag{B}[Bl][Bl]{\scalebox{.9}{$\varphi_{\alpha}(X)$}}
\psfrag{X}[Bl][Bl]{\scalebox{.9}{$X$}}
=\,\rsdraw{.45}{.9}{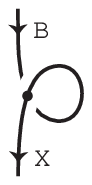}
\qquad \text{and} \qquad
\psfrag{X}[Bl][Bl]{\scalebox{.9}{$\varphi_{\alpha}(X)$}}
\psfrag{B}[Bl][Bl]{\scalebox{.9}{$X$}}
\theta_X^{-1}=\,\rsdraw{.45}{.9}{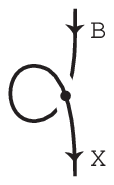}\;\;.
$$
\end{lem}
\begin{proof}
  We have
\begin{center}
$ \displaystyle
\theta_X \;\overset{(i)}{=}\;\,
\psfrag{X}[Bl][Bl]{\scalebox{.9}{$X$}}
\psfrag{Y}[Bl][Bl]{\scalebox{.9}{$\varphi_\alpha(X)$}}
\psfrag{u}[Bc][Bc]{\scalebox{.9}{$\theta_X^*$}}
\rsdraw{.35}{.9}{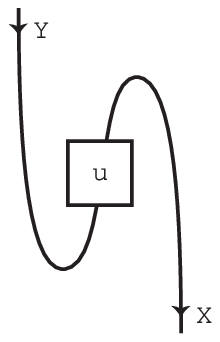}
 \overset{(ii)}{=}\;
\psfrag{X}[Bl][Bl]{\scalebox{.9}{$X$}}
\psfrag{Y}[Bl][Bl]{\scalebox{.9}{$\varphi_\alpha(X)$}}
\psfrag{u}[Bc][Bc]{\scalebox{.9}{$\varphi_{\alpha^{-1}}^1(\varphi_\alpha(X))$}}
\psfrag{v}[Bc][Bc]{\scalebox{.9}{$\varphi_2(\alpha^{-1},\alpha)_X^{-1}(\varphi_0)_X$}}
\rsdraw{.35}{.9}{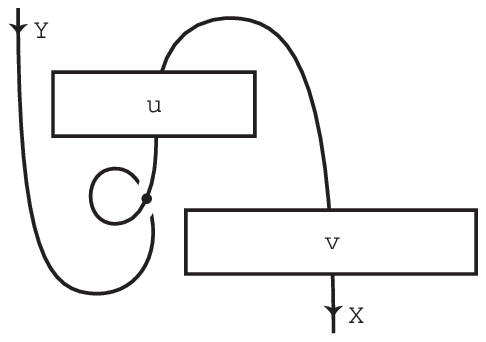}
$\\[.8em]
$ \displaystyle
\overset{(iii)}{=}\;
\psfrag{X}[Bl][Bl]{\scalebox{.9}{$X$}}
\psfrag{Y}[Bl][Bl]{\scalebox{.9}{$\varphi_\alpha(X)$}}
\psfrag{u}[Bc][Bc]{\scalebox{.9}{$(\varphi_0)_{\varphi_\alpha(X)}^{-1}\varphi_2(\alpha,\alpha^{-1})_{\varphi_\alpha(X)}$}}
\psfrag{v}[Bc][Bc]{\scalebox{.9}{$\varphi_2(\alpha^{-1},\alpha)_X^{-1}(\varphi_0)_X$}}
\rsdraw{.35}{.9}{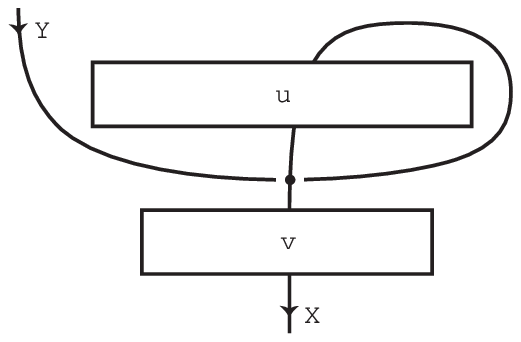}
\;\overset{(iv)}{=}\;\;
\psfrag{X}[Bl][Bl]{\scalebox{.9}{$X$}}
\psfrag{Y}[Bl][Bl]{\scalebox{.9}{$\varphi_\alpha(X)$}}
\rsdraw{.35}{.9}{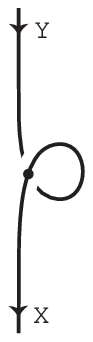}
$
\end{center}\vspace{.8em}
Here,  $(i)$ follows from the pivotality of $\cc$, $(ii)$ is obtained from \eqref{eq-ribbon} and the definition of $\theta_{\varphi_\alpha(X)}$,  $(iii)$ is obtained by applying   the second equality of Lemma~\ref{lem-braiding}(d), and  $(iv)$ follows from the naturality of $\tau$ and \eqref{eq-dem-twist}.

The proof of the second equality of the lemma uses the first equality and is similar to
the proof  of Lemma~\ref{lem-twist}.
\end{proof}

\subsection{The category $\cc_1$}\label{sect-cas-C1}
Given a   $G$-ribbon  category $(\cc,\varphi,\tau )$ with  twist $\theta$,   the category   $\cc_1$ is   a ribbon category   in the usual sense of the word  with braiding
$$
\{c_{X,Y}=(\id_Y \otimes (\varphi_0)^{-1}_X)\tau_{X,Y} \co X \otimes Y \to Y \otimes X\}_{X,Y \in \cc}
$$
and   twist $ \{v_X=(\varphi_0)^{-1}_X \theta_X\co X \to X\}_{X \in \cc}$.

For
  $G=1$,   the definitions of $G$-braided/$G$-ribbon categories  are
equivalent to the standard definitions of braided/ribbon
categories.

\section{Colored $G$-graphs}\label{Colored   $G$-graphs}

From now on, unless explicitly stated to the
contrary,  the symbol $\cc$ denotes  a pivotal $G$-crossed category
with pivotal crossing $\varphi $.

In this section, we  introduce ribbon graphs in $\RR^3$
and their colorings over $\cc$.

\subsection{Ribbon graphs}\label{ribbredGgraphs}   We   recall  the
notion of a ribbon graph following \cite{Tu0}. A {\it
coupon}\index{coupon} is an oriented  rectangle   with a
distinguished  side called the bottom base; the opposite side is
called   the top base. A {\it ribbon graph}  $\Omega$  with $k\geq
0$ {\it inputs} $((r,0,0))_{ r=1}^k$ and $ l\geq 0$ {\it outputs}
$((s,0,1))_{ s=1}^l$ consists of a finite family of coupons,
oriented circles, and  oriented segments  embedded in $\RR
\times (-\infty, 1] \times [0,1]$. The circles and the segments in
question are called the {\it circle components} and the {\it  edges}
of $\Omega$, respectively.
The inputs and outputs of $\Omega$ should be among the endpoints
of the edges, all the other endpoints of the edges should lie on
the bases  of the coupons. Otherwise, the edges, the circle
components, and the coupons  of $\Omega$ are disjoint. They are
also supposed to carry a    framing, i.e.,   a continuous
nonsingular vector field
  on $\Omega$ transversal to $\Omega$. It is required that near the inputs
and outputs of $\Omega$, the edges are  straight segments parallel to the axis $ \{(0,0)\}\times \RR
 $ and  the framing is  given by the vector $(0, \delta,0)$
with small $\delta>0$.
 The orientation of each coupon together with the framing should
yield  the negative (left-handed) orientation of $\RR^3$.

In the pictures   we will use the following
conventions: the first axis in $\RR^2\times [0,1]$ is a
horizontal line  on the page of the picture directed to the
right, the second axis is orthogonal to the plane of the picture
and is  directed from the eye of the reader towards this plane,
the third axis is a vertical line on the plane of the picture
directed from the bottom to the top.   Note that   points with
positive second coordinate lie behind the plane of the picture.
The distinguished bases of the coupons
in the pictures are the bottom horizontal sides.

\subsection{Tracks and meridians}\label{ribbredGgraphsNEWNEW}   Fix a base point $z\in \RR \times [2, \infty ) \times [0,1]$.   Given a  ribbon graph   $\Omega$, we   consider its complement $ C_\Omega=(\RR^2\times
[0,1]) \setminus \Omega$   in $\RR^2\times
[0,1]$ and observe that $z \in  C_\Omega$.
 We shall write
$\pi_1(C_{\Omega} )$ for $\pi_1(C_{\Omega}, z)$. Pushing
$\Omega$ along the framing we obtain a disjoint   copy
$\widetilde \Omega$ of $\Omega$. Pushing an   edge/coupon ${e}$
of $\Omega$ along the framing we obtain an   edge/coupon
$\widetilde {e}$
  of $\widetilde \Omega$. A path  $\gamma\colon [0,1] \to
C_{\Omega}$ from the base point $z=\gamma(0) $ to a point  of  $
\widetilde  {{e}} $  is called a {\it $z$-path} for ${e}$. By a
{\it homotopy} of a $z$-path $\gamma$ we  mean a deformation of
$\gamma$ in the class of $z$-paths   in $C_{\Omega}$ fixing
$\gamma(0)=z  $ and keeping $\gamma(1)$ on $\widetilde  {{e}}$.
The   homotopy classes of  $z$-paths for ${e}$  are called  {\it
tracks} of ${e}$  (with respect to $z$). Multiplication of loops
based at $z$ with $z$-paths defines a left action of
$\pi_1(C_{\Omega} )$ on the set of tracks of~${e}$. Since
$\widetilde  {{e}}$ is contractible, this action is transitive
and faithful. The tracks of edges (resp.\ coupons) of $\Omega$
are called {\it edge-tracks} (resp.\ {\it coupon-tracks}) of
$\Omega$.  We do not define tracks for circle components of
$\Omega$.

   For a  $z$-path $\gamma$ of  an edge/coupon  ${e}$, denote by
$\mu_{\gamma}\in \pi_1(C_{\Omega}) $ the (negative) meridian of
${e}$ represented by the loop $\gamma l_{e} \gamma^{-1}$, where
$l_{e} $ is a small loop in $C_{\Omega}$ encircling ${e}$ with
linking number $-1$, see Figure~\ref{fig-meridian}.  The meridian $\mu_{\gamma}$ depends only on the
track represented by $\gamma$. Clearly, $\mu_{\beta \gamma}=\beta \mu_\gamma \beta^{-1}$ for any $\beta\in \pi_1(C_\Omega )$.

\begin{figure}[t]
\begin{center}
\psfrag{u}[Br][Br]{\scalebox{.9}{$1$}}
\psfrag{n}[Bl][Bl]{\scalebox{.9}{$n$}}
\psfrag{x}[Bl][Bl]{\scalebox{.9}{$m$}}
\psfrag{e}[Bc][Bc]{\scalebox{.9}{$e$}}
\psfrag{z}[Bc][Bc]{\scalebox{.9}{$z$}}
\psfrag{g}[Bc][Bc]{\scalebox{.9}{$\gamma$}}
\psfrag{m}[Bc][Bc]{\scalebox{.9}{$\ell_e$}}
\rsdraw{.45}{.9}{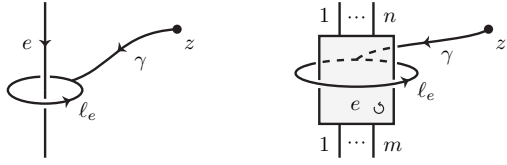}
\end{center}
\caption{Meridian $\mu_{\gamma}=\gamma l_{e} \gamma^{-1}$ of an edge/coupon $e$}
\label{fig-meridian}
\end{figure}

\subsection{Colorings of   graphs}\label{coloredGgraphs}  By a  {\it $G$-graph} we mean a ribbon
 graph $\Omega $ endowed with a
 group homomorphism $g\colon   \pi_1(C_\Omega ) \to G$.
 For brevity, we shall sometimes write $\Omega$ for the pair $(\Omega, g)$.   
  A {\it   ${\mathcal C}$-pre-coloring} or shorter a {\it
pre-coloring} $u$  of  a $G$-graph $(\Omega,g )$ comprises
two functions. The first function   assigns to every
edge-track  $\gamma $ of      $\Omega$ a  non-zero object $u_\gamma\in
{{\mathcal C}}_{g(\mu_{\gamma})}$ called the {\it color} of
$\gamma$.    The second function   assigns
 to every    edge-track  $\gamma $ of     $\Omega$  and to every
$\beta\in \pi_1(C_\Omega )$ an isomorphism
$$u_{\beta, \gamma}\colon  u_{\beta \gamma} \to \varphi_{g(\beta^{-1})} (u_\gamma) $$
so that for all $\gamma$, we have $u_{1,
\gamma}=(\varphi_0)_{u_\gamma}\colon u_{ \gamma} \to \varphi_{1}
(u_\gamma)$ and for all $ \beta, \delta \in  \pi_1(C_\Omega )$,
the following diagram   commutes:
\begin{equation}\begin{split}\label{condweak}\xymatrix@R=1cm @C=3cm {
u_{ \beta \delta \gamma} \ar[r]^-{u_{  \beta \delta, \gamma}}\ar[d]_{u_{\beta, \delta \gamma}} & \varphi_{g(\delta^{-1}\beta^{-1})} (u_{  \gamma}) \\
\varphi_{g(\beta^{-1})} (u_{\delta \gamma}) \ar[r]^-{\varphi_{g(\beta^{-1})}  (u_{\delta, \gamma}) } & \varphi_{g(\beta^{-1})} \varphi_{g(\delta^{-1} )} (u_{  \gamma})   \ar[u]_{\varphi_2(g(\beta^{-1}) , g(\delta^{-1}))_{u_\gamma}}.
} \end{split}\end{equation}

One can   extend this definition by allowing zero objects for   colors of tracks. This however does not lead to interesting invariants of graphs, and we shall not do it.

A {\it   ${\mathcal C}$-coloring} or shorter a {\it coloring}
of  a $G$-graph $(\Omega,g )$ consists of   a
pre-coloring $u$ and a   function  $v$ assigning to every
coupon-track  $\gamma$ of  $\Omega$  a    morphism $v_\gamma$ in
$ {{\mathcal C}}_{g(\mu_{\gamma})}$. To state our requirements
on $v_\gamma$, we need more terminology. By {\it entries}
(resp.\ {\it exits}) of a coupon $Q$ of $\Omega$, we mean the
endpoints of edges of $\Omega$ lying on the bottom  (resp.\ top)
side of $Q$.
 Let $m$   be the number of entries  of $Q$; the   direction of
 the bottom  side induced by the orientation of    $Q$
 determines an order in the set of the entries. Let  ${e}_i $
 be the   edge  of $\Omega$
incident to the $i$-th entry where $i=1, \ldots, m$. Set
  $\varepsilon_i =+$    if ${e}_i$ is directed
  out of $Q$ near the $i$-th entry and set   $\varepsilon_i =-$
  otherwise. Composing  a  $z$-path representing  $\gamma$ with
  a   path in
$\widetilde  Q$ leading   to the $i$-th entry, we obtain a track
$\gamma_i$  of $ {e}_i$ depending only on $\gamma$ and $i$.
Similarly, let $n$   be the number of exits  of $Q$; the
direction of the top side of $Q$ induced by the opposite
orientation of $Q$ determines an order in the set of the exits.
Let  ${e}^j$ be the edge  of $\Omega$ incident to the $j$-th
exit where $j=1, \ldots, n$. Set
   $\varepsilon^j =-$     if ${e}^j$ is directed out of $Q$ and
    $\varepsilon^j =+$   otherwise.
     Composing  a  $z$-path  representing $\gamma$ with a   path
in $\widetilde  Q$ leading  to the $j$-th exit we obtain  a
well-defined track $\gamma^j$  of $ {e}^j$.   Clearly,
$$
\mu_\gamma=\mu_{\gamma_1}^{\varepsilon_1} \cdots  \mu_{\gamma_m}^{\varepsilon_m} =
\mu_{\gamma^1}^{\varepsilon^1} \cdots  \mu_{\gamma^n}^{\varepsilon^n}  \in \pi_1(C_\Omega).
$$
We require that
\begin{enumerate}
\labeli
\item for any  coupon-track $\gamma$ of  $ \Omega$,  we have (in
the notation above)
\begin{equation}\label{vgamma}
v_\gamma \in \Hom_{{\mathcal C}}({{\otimes}}_{i=1}^m
u_{\gamma_i}^{\varepsilon_i},\,
 {{\otimes}}_{j=1}^n u_{\gamma^j}^{\varepsilon^j})
\end{equation}
where for an object $U$ of $\mathcal C$, we set $U^+=U$ and $U^-=U^*$;

\item for any  $  \gamma$ as in (i) and   any $\beta\in \pi_1(C_\Omega)$, the following diagram   commutes:
\begin{equation}\begin{split}\label{weakiso---}
\xymatrix@R=1cm @C=1.7cm {
{{\otimes}}_{i=1}^m
u_{\beta \gamma_i}^{\varepsilon_i}  \ar[r]^-{\otimes_{i=1}^m
u_{\beta, \gamma_i}^{\varepsilon_i}}\ar[d]_{v_{ \beta \gamma}  } &
{{\otimes}}_{i=1}^m \varphi_{g(\beta^{-1})}  (u_{ \gamma_i}^{\varepsilon_i}  ) \ar[r]^-{(\varphi_{g(\beta^{-1})})_m}
 &\varphi_{g(\beta^{-1})} ( {{\otimes}}_{i=1}^m    u_{ \gamma_i}^{\varepsilon_i})  \ar[d]^-{  \varphi_{g(\beta^{-1})} (v_\gamma) }\\
 {{\otimes}}_{j=1}^n u_{\beta \gamma^j}^{\varepsilon^j}
 \ar[r]^{\otimes_{j=1}^n  u_{\beta, \gamma^j}^{\varepsilon^j}}
 &{{\otimes}}_{j=1}^n  \varphi_{g(\beta^{-1})} (u_{\gamma^j}^{\varepsilon^j})  \ar[r]^-{(\varphi_{g(\beta^{-1})})_n}
& \varphi_{g(\beta^{-1})}( {{\otimes}}_{j=1}^n     u_{  \gamma^j}^{\varepsilon^j}  )}
\end{split}\end{equation}
where $u_{\beta, \gamma}^+= u_{\beta, \gamma} $,   $u_{\beta,
\gamma}^- \colon  u_{\beta \gamma}^* \to \varphi_{g(\beta^{-1})}
(u_\gamma^*)$ is defined in Section~\ref{sect-pivot-cross-Gcat}, and the notation $(\varphi_\alpha)_n$ is defined in Section~\ref{sect-monofunctor}.
In the case where $m=n=1$ and $\varepsilon_1=\varepsilon^1=+$,
the diagram \eqref{weakiso---} simplifies to
 \begin{equation}\begin{split}\label{weakiso---next}\xymatrix@R=1cm @C=1.7cm {
 u_{\beta \gamma_1}  \ar[r]^-{u_{\beta, \gamma_1}} \ar[d]_{v_{ \beta \gamma}  } &
  \varphi_{g(\beta^{-1})} (    u_{ \gamma_1} )  \ar[d]^-{  \varphi_{g(\beta^{-1})} (v_\gamma) }\\
 u_{\beta \gamma^1}
 \ar[r]^{u_{\beta, \gamma^1} \,\,\,\,\,\,\,\,\,\,\,\,\,}
 &      \varphi_{g(\beta^{-1})}(    u_{  \gamma^1}  ).
}\end{split}\end{equation}
\end{enumerate}

When $m=0$ and/or $n=0$, we use in \eqref{vgamma} and in similar
 formulas below the   convention that an empty monoidal product
 of objects is the
  unit object.

Any pre-coloring $u$ of a $G$-graph $ \Omega $ can be extended
to a coloring  of $ \Omega $ as follows:  for each coupon $Q$ of
$\Omega$ pick a  track $\gamma $ of $Q$ and  a    morphism
$v_\gamma$ as in \eqref{vgamma}.  For all $\beta\in
\pi_1(C_\Omega)\setminus \{1\}$,  the  morphism  $ v_{\beta
\gamma}$ is determined uniquely from the   diagram
\eqref{weakiso---}. One may check that this gives a  coloring
$(u,v)$ of $ \Omega $.

\subsection {The source and the target}
Consider a $G$-graph $\Omega=(\Omega, g)$ with $k\geq 0$
inputs and $l\geq 0$ outputs. For $r=1, \ldots, k$ consider the
path in $C_\Omega$   obtained as   the product of the linear
paths from    the base point $z=(z_1,z_2,z_3)$  to $(r,z_2,0)$
and  from $(r,z_2,0)$ to $(r, \delta,0)$, where $\delta$ is a
small positive real number.   This product is a   $z$-path of
the  edge of $\Omega$ incident to the $r$-th input. The
corresponding  track is denoted $\gamma_r$ and called the  {\it
$r$-th input track} of $ \Omega$. Also   we define a   sign
$\varepsilon_r $
  to be   $ +$   if the edge of $\Omega$ incident to the $r$-th
  input is directed down (into   $ \RR^2\times
  (-\infty, 0]$)  and to be    $-$   otherwise.  Similarly, for
$s=1, \ldots, l$,  the   product of the linear paths from $z$
  to $(s,z_2,1)$ and
from $(s,z_2,1)$ to $(s,\delta,1)$  is a   $z$-path of the  edge
of $\Omega$ incident to the $s$-th output.  The corresponding
track is denoted $\gamma^s$ and called the  {\it $s$-th output
track} of $ \Omega$. Set   $\varepsilon^s=+$  if the  edge of
$\Omega$ incident to the $s$-th output     is directed down
(into $ \RR^2\times [0,1]$)  and set   $\varepsilon^s=
-$   otherwise.  Given a pre-coloring $u$ of $\Omega$, the
sequences
   $( u_{\gamma_1}, \varepsilon_1),\ldots ,( u_{\gamma_k},
   \varepsilon_k)$ and $( u_{\gamma^1}, \varepsilon^1),\ldots ,(
   u_{\gamma^l},
\varepsilon^l)$ are called respectively   the {\it source} and
the {\it target} of the pre-colored $G$-graph $(\Omega, u
)$. Here $u_{\gamma_r} \in {\mathcal C}_{g(\mu_{\gamma_r})}$ and
$u_{\gamma^s} \in {\mathcal C}_{g(\mu_{\gamma^s})}$ for all
$r,s$. Clearly,
$$\mu_{\gamma_1}^{\varepsilon_1} \cdots \mu_{\gamma_k}^{\varepsilon_k}= \mu_{\gamma^1}^{\varepsilon^1}\cdots \mu_{\gamma^l}^{\varepsilon^l} .$$

 \subsection{Isomorphisms of colorings}\label{sect-iso-quasi-iso}   Let $\Omega=(\Omega, g)$ be a $G$-graph.
  An  {\it  isomorphism}      $u\approx u'$ of pre-colorings
 of  $\Omega $   is a system of isomorphisms $f=\{f_\gamma
 \colon u_\gamma \to u'_\gamma\}_\gamma$, where $\gamma$ runs
 over all   edge-tracks of   $\Omega$,  such that
  for any   $\gamma$  and any $ \beta \in  \pi_1(C_\Omega)$, the
  following diagram   commutes:
\begin{equation}\begin{split}\label{weakiso}\xymatrix@R=1cm @C=3cm {
u_{  \beta \gamma} \ar[r]^-{u_{   \beta, \gamma}}\ar[d]_{f_{\beta  \gamma}} & \varphi_{g(\beta^{-1} )} (u_{  \gamma}) \ar[d]^{ \varphi_{g(\beta^{-1} )} (f_{  \gamma})} \\
u'_{  \beta \gamma} \ar[r]^-{u'_{   \beta, \gamma}}  &
\varphi_{g(\beta^{-1} )} (u'_{
\gamma}).}\end{split}\end{equation} Note that for $\beta=1$ the
commutativity of \eqref{weakiso} follows from the  definition of
pre-colorings and the naturality of $\varphi_0$.

Isomorphisms of pre-colorings may be used to replace the
 colors of edges with isomorphic objects. Specifically, suppose
 that $u$ is a pre-coloring of $ \Omega $ and   that for each
 edge-track $\gamma$ of   $\Omega$ we have  an object $u'_\gamma  \in {{\mathcal
 C}}_{g(\mu_{\gamma})}$ and an isomorphism $
 f_\gamma \colon u_\gamma \to u'_\gamma  $.  Then the system  $
 \{u'_\gamma\}_\gamma$ extends uniquely to a pre-coloring $ u'
 $ of $ \Omega $  such that $f=\{f_\gamma  \}_\gamma\colon u\to
 u'$ is an isomorphism of pre-colorings. Indeed, the morphisms
 $u'_{\beta, \gamma}$ can be uniquely recovered from
 \eqref{weakiso}.
For example, given $u$, we can replace the color of any
    edge-track $\gamma_0$   via any isomorphism $f_0\colon
    u_{\gamma_0} \to V\in {{\mathcal C}}_{g(\mu_{\gamma_0})}$
    keeping the  colors of all the other   edge-tracks. This is
    achieved by applying the procedure above to
  the system     $u'_{\gamma_0} = V$,
  $f_{\gamma_0}=f_0$ and $u'_\gamma=u_\gamma$,  $f_\gamma=\id$ for
  $\gamma\neq \gamma_0$.

 The following lemma  shows that   to specify a pre-coloring
 it is  essentially  enough to color one track  for every edge.

\begin{lem}\label{lem-weakcolorings}    Let  ${E}$ be the set of edges of
  $\Omega $.  Pick a      track $ \gamma_{e} $   of  ${e}$ for
  all    ${e}\in {E} $.
\begin{enumerate}
\labeli
\item  For any family of   non-zero   objects $\{ u_{e} \in {{\mathcal
C}}_{g(\mu_{\gamma_{e}})}\}_{{e}\in {E}}$, there is a
pre-coloring $u$ of $\Omega $ such that
$u_{\gamma_{e}}=u_{e}$ for all ${e}\in {E}$.
\item Given      pre-colorings  $u$,  $u'$ of $\Omega$,  any
system of  isomorphisms $\{ u_{\gamma_e}\to   u'_{\gamma_e}
\}_{e\in E}$ extends uniquely to an isomorphism  $ u \approx
u'$.
\end{enumerate}
\end{lem}

\begin{proof}
Let us prove (i).  For any     track $\gamma$ of ${e}\in {E}$, there is a
unique element of $\pi_1(C_\Omega)$  denoted
$\gamma_{e}\gamma^{-1}$ such that $(\gamma_{e}\gamma^{-1})
\gamma=\gamma_e$.  Set $u_\gamma=
\varphi_{g(\gamma_{e}\gamma^{-1})} (u_{e})$. For any $\beta\in
\pi_1(C_\Omega)$, consider the isomorphism
$$  \varphi_2( g(\beta^{-1}), g(\gamma_{e} \gamma^{-1}))_{u_{e}}\co    \varphi_{g(\beta^{-1})} \varphi_{g(\gamma_{e} \gamma^{-1})} (u_{e}) \to
\varphi_{g( \gamma_{e} \gamma^{-1} \beta^{-1})} (u_{e})  .$$ The
source and the target of this isomorphism  are the objects
$\varphi_{g(\beta^{-1})} (u_\gamma)$ and $u_{\beta \gamma}$,
respectively. Let $u_{\beta, \gamma}$ be the inverse isomorphism
$u_{\beta \gamma} \to  \varphi_{g(\beta^{-1})} (u_\gamma)$. The
commutativity of the diagrams \eqref{crossing5} and
\eqref{crossing6} implies that the functions $\gamma\mapsto
u_\gamma$ and $(\gamma, \beta) \mapsto u_{\beta, \gamma}$ form
a pre-coloring  of $ \Omega $.

Clearly, $u_{\gamma_{e}}= \varphi_{1} (u_{e})$ is isomorphic to
$u_{e}$ for all ${e}\in {E}$. Replacing the colors inductively
as described before the lemma, we can ensure that
$u_{\gamma_{e}}= u_{e}$ for all ${e} $.

Let us prove (ii). Fix a system of  isomorphisms $\{f_e\colon u_{\gamma_e}\to
u'_{\gamma_e} \}_{e\in E}$. Consider an isomorphism $f\co u \to u'$
such that $f_{\gamma_{e}}=f_e$ for all $e\in E$.   Replacing
$\gamma$ and $\beta$  in \eqref{weakiso} by $\gamma_{e}$ and
$\gamma \gamma_{e}^{-1}$, respectively, we obtain that for any
track $\gamma$ of an edge ${e}\in {E}$,
 \begin{equation}\begin{split}\label{defiso}f_{\gamma}= (u'_{\gamma \gamma_{e}^{-1}, \gamma_e})^{-1}
\varphi_{g(\gamma_e \gamma^{-1})} (f_e ) \,  u_{\gamma \gamma_{e}^{-1}, \gamma_e}\co u_\gamma \to u'_\gamma.\end{split}\end{equation}  This proves the uniqueness of $f=\{f_\gamma\}_\gamma$. To prove the existence of $f$, we define each $f_\gamma$ by \eqref{defiso}. The naturality of $\varphi_0$ implies that $f_{\gamma_e}=f_e$ for all $e\in E$.  It remains to verify the commutativity of the diagram \eqref{weakiso} for all $\beta$,  $ \gamma$. For $\gamma=\gamma_{e}$, the commutativity of \eqref{weakiso} follows from the  definition of $f_\gamma$. Any    track $\gamma$ of ${e}\in {E}$ expands as $\delta \gamma_{e}$ with $\delta\in \pi_1(C_\Omega )$ and the commutativity of \eqref{weakiso}
follows from the commutativity of the cubic diagram in which two
horizontal squares are the diagram  \eqref{condweak} with
$\gamma$ replaced by $\gamma_{e}$ and a similar diagram with $u$
replaced by  $u'$, while the vertical isomorphisms relating
these two squares  are   induced by   $f$.
\end{proof}

Consider two  colorings $(u,   v)$ and $(u',  v')$ of $\Omega $
  with  the same source and target so that $u_{\gamma
  }=u'_{\gamma }$ for all input/output   tracks $\gamma$ of
  $\Omega$.  By an   {\it   isomorphism}
   $(u,   v) \approx (u',  v')$, we mean an isomorphism of
 pre-colorings $f \colon u \to u'$   such that  for any
 input/output   track  $\gamma$ of $\Omega$, we have
 $f_\gamma=\id\co  u_\gamma\to u'_\gamma $ and for any coupon $Q$
 of $\Omega$ and any track $\gamma$ of $Q$, the following
 diagram commutes:
\begin{equation}\begin{split}\label{weakiso+}\xymatrix@R=1cm @C=3cm {
{{\otimes}}_{i=1}^m
u_{  \gamma_i}^{\varepsilon_i}  \ar[r]^-{v_{   \gamma}}\ar[d]_{\otimes_{i=1}^m
f_{  \gamma_i}^{\varepsilon_i}} &  {{\otimes}}_{j=1}^n u_{  \gamma^j}^{\varepsilon^j}
 \ar[d]^{\otimes_{j=1}^n f_{  \gamma^j}^{\varepsilon^j}}\\
 {{\otimes}}_{i=1}^m (u'_{  \gamma_i})^{\varepsilon_i}  \ar[r]^-{v'_{   \gamma}} &  {{\otimes}}_{j=1}^n (u'_{  \gamma^j})^{\varepsilon^j}  .
 }\end{split}\end{equation} Here  we use the notation of Section
\ref{coloredGgraphs} and set  $$f_{  \gamma}^+=  f_\gamma \colon
u_\gamma \to u'_\gamma \quad \quad {\text {and}}\quad \quad f_{
\gamma}^- =(f_{ \gamma}^*)^{-1} \colon    u_{ \gamma}^* \to
(u'_\gamma)^* .$$ Note that if the diagram \eqref{weakiso+}
commutes for one track $\gamma$ of $Q$, then it commutes for all
tracks of $Q$. This follows from the commutativity of
\eqref{weakiso---}.

\subsection {Color-equivalence}\label{color-equivalence1}   By a {\it self-homeomorphism} of   $\RR^2\times [0,1]$ we   mean a  homeomorphism $\RR^2\times
[0,1]\to \RR^2\times [0,1]$  which is   the identity outside
a compact subset of     $\RR \times (-\infty, 1]\times
(0,1) $.   Self-homeomorphisms of $\RR^2\times [0,1]$ fix
the base point $z$ (for any choice of $z$ as   in Section~\ref{ribbredGgraphsNEWNEW}) and act on
colored $G$-graphs in the obvious way. Two colored
$G$-graphs related by a self-homeomorphism of $\RR^2\times [0,1]$ are {\it isotopic}.  It is clear that colored
$G$-graphs are isotopic if and only if there is a
color-preserving deformation of one into the other    in the class
of colored $G$-graphs.

Two colored $G$-graphs are {\it color-equivalent} if they
can  be obtained from each other through  isotopy and
isomorphism of colorings. Color-equivalent  colored
$G$-graphs necessarily have the same source and the same
target.

We  define a \lq\lq stable" version of the color-equivalence using the following transformation of colored $G$-graphs.
Pick an edge  of a colored $G$-graph $(\Omega,g, u,v)$ and
  insert in this edge a new coupon $Q$ with one entry and one
  exit, see Figure \ref{stabi} where the framing is orthogonal
  to the
 plane of the picture and is directed behind the picture. This
 gives a ribbon graph $\Omega'$ containing $\Omega$ as a subset.
 The inclusion of the graph complements   $i \colon C_{\Omega'} \hookrightarrow C_\Omega$
 is a homotopy equivalence. We
  now derive from the coloring $(u,v)$ of $(\Omega,g)$ a
 coloring $(u',v')$ of the $G$-graph $\Omega'=(\Omega', g i_*\co
\pi_1(C_{\Omega'})\to G)$. Composition with $i$   transforms any
edge-track $\gamma$ of $\Omega'$    into an edge-track $
i\gamma$ of $\Omega$, and we set $u'_\gamma=u_{i\gamma}$.
Similarly, set $u'_{\beta, \gamma}=u_{i_*(\beta), i \gamma}$ for
$\beta\in \pi_1(C_{\Omega'})$. For any track $\gamma$ of a
coupon of $\Omega'$ distinct from $Q$, set $v'_\gamma=
v_{i\gamma}$. For any track $\gamma$ of $Q$, we have
$i\gamma^1=i\gamma_1$ so that $u'_{\gamma^1} =u'_{\gamma_1}$ and
we   set $v'_\gamma=\id\co  u'_{\gamma_1}\to u'_{\gamma^1}$. We
obtain in this way a colored $G$-graph $\Omega'$ with the same
source and target as $\Omega$. We call this construction {\it
stabilization}. Two colored $G$-graphs   are   {\it stably
color-equivalent} if stabilizing   them several times we can
obtain color-equivalent colored $G$-graphs.

  \begin{figure}[t]
\begin{center}
\rsdraw{.45}{1}{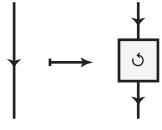}
\end{center}
\caption{Stabilization}
\label{stabi}
\end{figure}

\subsection {Base points re-examined}
 The structure  of   a  pre-colored  $G$-graph  on  a
ribbon graph $\Omega $ depends
on the choice of a  base point $z$ in    $Z= \RR \times
[2,\infty)\times [0,1] $.  We can transfer this structure along
any path $\rho  $ in $Z $ from $z $ to $z'\in Z $. Given a homomorphism $g \colon \pi_1(C_\Omega,z)\to
G$ and a pre-coloring $u$ of $(\Omega, g)$, we define a
homomorphism $g'\colon \pi_1(C_\Omega,z')\to G$ and a
pre-coloring $u'$ of $(\Omega, g')$ by  $g' (\beta)= g( \rho
\beta \rho^{-1})$, $u'_\gamma=u_{\rho \gamma} $, and $u'_{\beta,
\gamma}=u_{\rho \beta \rho^{-1}, \rho \gamma} $   for any
$\beta\in \pi_1(C_\Omega,z')$ and any   edge-track $\gamma$ of
$\Omega$ with respect to $z'$.
This gives a pre-colored $G$-graph $(\Omega, g', u')$, the
  {\it transfer} of $(\Omega, g,u)$ along $\rho$. Clearly,
transfers  along homotopic paths   are equal. Since
  $ Z$ is
contractible, we can  thus move between   base points in a canonical
way. Alternatively, we can consider  $ Z $ as a \lq\lq
big base point". As a consequence, we shall suppress the base point
from the notation for pre-colored  $G$-graphs. Similar remarks
apply to colored $G$-graphs.

\section{The category ${\mathcal G}_{\cc} $}
In this section, we organize   $\cc$-colored $G$-graphs into a monoidal  category ${\mathcal G}={\mathcal
G}_{\cc} $.

\subsection{The category ${\mathcal G}$}\label{sect-category-of-graphs}
The objects of ${\mathcal  G}$ are finite sequences $ ((U_r,
\varepsilon_r))_{r=1}^k$ where $k\geq 0, \varepsilon_r=\pm$, and
$U_r$ is a non-zero homogeneous object of $ {\mathcal C} $ for
$r=1,\ldots ,k$.     A morphism $((U_r, \varepsilon_r))_{r=1}^k
\to ((U^s, \varepsilon^s))_{s=1}^l$ in $\mathcal  G$ is  the
stable color-equivalence class of  a colored $G$-graph having no
circle components and having the target $((U_r,
 \varepsilon_r))_{r=1}^k$ and the source $((U^s,
\varepsilon^s))_{s=1}^l$.   We now define composition of
morphisms in ${\mathcal G}$. Consider two colored $G$-graphs
$(\Omega^t=(\Omega^t, g^t, u^t, v^t))_{t=1,2}$ such that   ${\rm
{source}} (\Omega^1)= {\rm {target}} (  \Omega^2)=((U_r,
\varepsilon_r))_{r=1}^k$. Stabilizing
if necessary these graphs, we can assume that  they have no
edges with both endpoints lying in the set of inputs and
outputs. Let ${\iota}^t$ be the embedding of the strip
 $\RR^2\times [0,1] $ into itself carrying any point $(x_1, x_2,
x_3)$ to $(x_1, x_2, (x_3+4-2t)/3)$. Then $ {\iota}^1 (\Omega^1)$
(resp. ${\iota}^2 (\Omega^2)$) lies in the upper (resp.\ lower) third
of the strip. In the middle third $\RR^2\times [1/3,2/3] $ we
insert a row of $k$ copies of the graph shown on the right
picture in Figure \ref{stabi}   (with the orientation of the two edges in the $r$-th copy reversed to the upward direction whenever $\varepsilon_r=-$).   The
union, $\Omega$, of these $k$ copies with $ {\iota}^1 (\Omega^1)\cup
{\iota}^2 (\Omega^2)$  is a ribbon graph without circle components.   The van Kampen theorem
  implies that there is a unique
homomorphism $g\co \pi_1(C_\Omega)\to G$ such that $g {\iota}^t_*= g^t\co
\pi_1 (C_{\Omega^t})\to G$ for $t=1,2$. Observe   that any
coloring $(u,v)$ of  the $G$-graph $\Omega=(\Omega, g)$  induces
a coloring $(u {\iota}^t, v {\iota}^t)$  of the $G$-graph
$\Omega^t=(\Omega^t, g^t)$ for $t=1,2$. We show now that the
given colorings of $ \Omega^1 $  and $\Omega^2$ determine a
  coloring of $\Omega $.

\begin{lem}\label{lem-composcolorings}  Let $\gamma (r)$ be the
linear path in $C_{\Omega}$ leading from the base point to the
$r$-th coupon in $\RR^2\times [1/3,2/3] $, where $r=1,...,k$. There
is a coloring $(u,v)$ of the $G$-graph $\Omega $ such that
\begin{equation}\label{gluglu} u_{\gamma(r)_1}=u_{\gamma(r)^1}=U_r, \quad   v_{\gamma(r)}=
\id\co  {U_r^{\varepsilon_r}}\to U_r^{\varepsilon_r} \quad {  {for}}
\quad r=1,...,k,\end{equation} and $(u   {\iota}^t, v {\iota}^t) \approx (u^t,
v^t)$ for $t=1,2$. Such a coloring   of $\Omega$ is unique up to
isomorphism and for this coloring, ${\rm {source}} (\Omega)={\rm
{source}} (\Omega^2)$, ${\rm {target}} (\Omega)={\rm {target}}
(\Omega^1)$.
\end{lem}

\begin{proof} For $t=1, 2$, denote by   $E^t$   the set of
edges of $\Omega^t$. For every edge $e\in E^t$ fix a track
$\gamma_e$ of $e$. We assume that if $e$ is incident to an input
(resp.\ output) of $\Omega^t$, then $\gamma_e$ is the
corresponding input (resp.\ output) track of~$\Omega^t$. (Here
  we use the assumption that no edge of $\Omega^t$ has both
endpoints among inputs and outputs.)  Note that the set of edges
of $\Omega$ can be identified with  $ E^1\amalg E^2$.

We first prove the existence of $(u,v)$.   For   $e\in E^t$, the
composition  ${\iota}^t \gamma_e$ is a track of the edge of $\Omega$
containing ${\iota}^t(e) $. Clearly,  $g(\mu_{{\iota}^t
\gamma_e})=g^t(\mu_{\gamma_e})$. Lemma
  \ref{lem-weakcolorings}(i) implies that there is a
  pre-coloring $u$ of $(\Omega, g)$ such that
$u_{{\iota}^t\gamma_{e}}=u^t_{\gamma_e}$ for all ${e}\in {E^t}$ and
$t=1,2$. The choice of $\{\gamma_e\}_e$ ensures that if $e\in E^1$
is incident to the $r$-th input of $\Omega^1$, then
$u_{{\iota}^1\gamma_{e}}=u^1_{\gamma_e}=U_r$, and   if $e\in E^2$ is
incident to the $r$-th output of $\Omega^2$, then
$u_{{\iota}^2\gamma_{e}}=u^2_{\gamma_e}=U_r$.  For $t=1,2$,
consider the pre-coloring $  u {\iota}^t $ of $\Omega^t$. By
definition, $(u {\iota}^t)_{\gamma_e}=
u_{{\iota}^t\gamma_{e}}=u^t_{\gamma_e}$ for all $e\in E^t$. By Lemma
\ref{lem-weakcolorings}(ii), there is an isomorphism of
pre-colorings $f^t\co u {\iota}^t \to u^t$ extending the identity
morphisms $\{\id\co (u {\iota}^t)_{\gamma_e}\to u^t_{\gamma_e}\}_{e\in
E^t}$. Next, we extend   $u$   to a coloring $(u,v)$ of
$\Omega$ as follows. Fix a track $\gamma_Q$ for every coupon $Q$ of
$\Omega^t$   with $t=1,2$.   The  morphism $v_{{\iota}^t \gamma_Q}$
is uniquely determined by the condition
 that  the isomorphism of pre-colorings $f^t $
carries $v_{{\iota}^t \gamma_Q}$ to $v^t_{\gamma_Q}$, i.e., we have
the commutative
 diagram \eqref{weakiso+} where $u, u', v'_\gamma, f $ are
replaced by $ u{\iota}^t, u^t, v^t_{\gamma_Q}, f^t$, respectively.
This and \eqref{gluglu} yields a value of $v$ on one track
for each coupon of $\Omega$. The last remark of Section
\ref{coloredGgraphs} shows that these values extend to a coloring
$(u,v)$ of $\Omega$. Since  $f^t\co  u {\iota}^t \to u^t$ carries
$v_{{\iota}^t \gamma_Q}$ to $v^t_{\gamma_Q}$, the last remark of
Section \ref{sect-iso-quasi-iso} implies that $f^t$
carries $v{\iota}^t$ to $v^t$. Thus, the coloring $(u,v)$
satisfies
 all the
requirements of the lemma. The equalities ${\rm {source}}
(\Omega)={\rm {source}} (\Omega^2)$ and ${\rm {target}}
(\Omega)={\rm {target}} (\Omega^1)$ follow from the definition
of $(u,v)$.

Let us prove the uniqueness of $(u,v)$. Suppose that $(u,v)$ and
$(\overline u, \overline  v)$ are two colorings of $\Omega$
satisfying the conditions of the lemma. Pick isomorphisms $h^t\co
(u   {\iota}^t, v {\iota}^t) \to (u^t, v^t)$ and $\overline h^t\co  (\overline
u {\iota}^t, \overline v {\iota}^t) \to (u^t, v^t)$ for $t=1,2$.  For $e\in
E^t$ consider the   induced isomorphisms
$$ h^t_{{\iota}^t\gamma_e}\co    u_{{\iota}^t \gamma_e}=  (u
  {\iota}^t)_{\gamma_e} \to u^t_{\gamma_e}\quad {\rm {and}}\quad
\overline h^t_{{\iota}^t\gamma_e}\co  \overline u_{{\iota}^t
\gamma_e}=(\overline u {\iota}^t)_{\gamma_e} \to u^t_{\gamma_e}.$$
Consider the composed isomorphisms $$ \{H_{ e}= (
h^t_{{\iota}^t\gamma_e})^{-1} \overline h^t_{{\iota}^t\gamma_e}\co
\overline u_{{\iota}^t \gamma_e} \to u_{{\iota}^t \gamma_e}\}_{e\in
E^t, \, t=1,2}.$$ By Lemma \ref{lem-weakcolorings}(ii), this system
of isomorphisms extends to an isomorphism of pre-colorings
$H\co \overline u \to u $. We claim that $H$ is an isomorphism
$(\overline u, \overline v)\approx (u,v)$. Note that all the
input/output tracks of $\Omega$ belong to the system
$\{\gamma_e\}_{e}$, and the values of $h^t$ and $\overline h^t$ on
these tracks are the identity morphisms of the corresponding objects
of~$\cc$. Therefore the same is true for~$H$. A similar argument
involving the inputs of $\Omega^1$, the outputs of $\Omega^2$, and
the assumption $\overline v_{\gamma(r)}= v_{\gamma(r)}= \id$
implies that the values of $H$ on the coupon-tracks
$\{\gamma(r)\}_{r=1}^k$ carry $\overline v$  to
$v$. By the last remark of Section \ref{sect-iso-quasi-iso}, the same is true for all coupon-tracks
of the $k$ coupons of $\Omega$ lying in $\RR^2\times [1/3,
2/3]$. Finally, the assumption that  $h^1$ carries $v{\iota}^1$ to
$v^1$ and $\overline h^1$ carries $\overline v{\iota}^1$ to $v^1$
implies that the values of $H$ on all coupon-tracks of $\Omega$
entirely lying in $\RR^2 \times [2/3,1]$  carry $\overline v$
to $v$. Using again the last remark of Section \ref{sect-iso-quasi-iso}, we deduce the same for all coupon-tracks of the
coupons of $\Omega$ lying in $\RR^2\times [2/3, 1]$. The coupons
lying in $\RR^2\times [0, 1/3]$ are treated similarly. This proves
our claim.
\end{proof}

We define   composition of the morphisms in $\mathcal G$
represented by $\Omega^1$, $\Omega^2$ to be the stable
color-equivalence class of $(\Omega,g,u,v)$. This composition is
well-defined and associative. The identity morphisms are
represented by   colored $G$-graphs   formed by oriented
vertical segments with constant framing (and no coupons).

\subsection{Monoidal product in ${\mathcal G}$}\label{sect-category-of-graphs+}     We define   a   monoidal    product
 in ${\mathcal G}$. The    monoidal   product of the objects of
 ${\mathcal G}$ is the juxtaposition of sequences.
The unit object is the empty sequence. To define the    monoidal
product of morphisms represented by colored $G$-graphs
  $(\Omega^t=(\Omega^t, g^t, u^t, v^t))_{t=1,2}$ we proceed as
follows. Positioning a copy of $\Omega^1$ to the left of a vertical
band $\{\ast\} \times \RR\times [0,1]$  (with $\ast\in
\RR$) and a copy of $\Omega^2$ to the right of this band,
and taking the union, we obtain a ribbon graph,~$\Omega$. The
complement $C_\Omega$ of $\Omega$ deformation retracts onto the
wedge $C_{\Omega^1} \bigvee C_{\Omega^2}$. The van Kampen theorem
yields a homomorphism $g\co \pi_1(C_\Omega)\to G$ whose restriction to
$\pi_1(C_{\Omega^t})$ is equal to $g^t$ for $t=1,2$. An analogue of
Lemma~\ref{lem-composcolorings} says that there is a unique (up to
isomorphism) coloring $(u,v)$ of $\Omega$ whose restriction to
$C_{\Omega^t}$ gives a coloring of $\Omega^t$ isomorphic to $(u^t,
v^t)$ for $t=1,2$. The stable color-equivalence class of
$(\Omega,g,u,v)$ is the   monoidal    product of the morphisms represented
by $\Omega^1$ and $\Omega^2$. This    monoidal    product is well defined
and   turns ${\mathcal G}$ into a strict
 monoidal category.

 \subsection{Remarks}\label{remarkoncomposition}  1. The coloring   of $\Omega$ provided by Lemma
\ref{lem-composcolorings} is defined only up to isomorphism. Using
the replacement technique of Section \ref{sect-iso-quasi-iso}, we can find a representative $(u,v)$ in this
isomorphism class such that $(u {\iota}^t, v {\iota}^t) = (u^t,
v^t)$ for $t=1,2$. The values of   $(u,v)$ on the tracks and loops
lying in the upper (resp.\ lower) third of $\RR^2\times [0,1]$ are
given directly by $(u^t, v^t)_{t=1,2}$.  A similar remark
applies to the construction of    monoidal    product in Section
\ref{sect-category-of-graphs+}.

2. A useful  class of ribbon graphs without circle components is formed by string links (which generalize braids). By a \emph{$k$-string link}
  with $k\geq 1$ we mean a system of $k$ framed oriented segments embedded in $\RR^2\times [0,1]$ and meeting the boundary planes
   at the points   $\{(r,0,0), (r,0,1)\}_{r=1,...,k}$. The framing should be given by the vector $(0,\delta, 0)$ at the endpoints where $\delta$ is a small positive real number. Such a string link  is a ribbon graph without coupons. All the definitions given above for ribbon graphs apply to string links.  To turn a string link $L$ (equipped with a principal $G$-bundle on the exterior) into a $\cc$-colored ribbon graph it is enough
to color the  input  tracks  of $L$ with objects of $\cc$. This
determines a $\cc$-coloring of $L$ uniquely up to color-preserving
isomorphism.


 \section{Colored   diagrams}\label{Colored   diagrams}

 We introduce colored   diagrams which will be used in the next sections  to
   represent colored ribbon graphs.

\subsection{Graph diagrams}
A {\it graph diagram}    is
a finite family of embedded coupons, immersed
segments,  and immersed circles in $\RR\times [0,1] $.
The segments and circles are called the {\it 1-strata} of the
diagram. We require that
\begin{enumerate}
\labeli
\item  the coupons  are  oriented counterclockwise, disjoint, and
  lie in $\RR \times (0,1)$;
\item the 1-strata are oriented and have only double
  transversal crossings   in $\RR \times (0,1)$ with
  over/under-data at all crossings;
\item the set of the endpoints of the 1-strata consists of the
points $((r,0))_{ r=1}^k$ (the inputs) and  $((s,1))_{ s=1}^l$ (the
outputs) for some $k, l\geq 0$ together with certain points lying on
the distinguished (bottom) sides of the coupons and the opposite
(top) sides. The 1-strata do not meet the coupons other than at the
endpoints and meet $\RR \times \{0,1\}$ orthogonally at the
inputs and outputs.
\end{enumerate}

We do not require the sides of the coupons in the diagrams  to
be parallel to the horizontal and vertical axes in $\RR^2$.
However,  in the pictures below, we will have only such coupons. By
convention,  the distinguished sides of the coupons in our pictures
are their bottom horizontal sides.

  Each crossing  $c$ of a graph diagram $D$ gives rise to two
  points on the 1-strata of $D$: the undercrossing  $c_{\rm
  {un}}$ and  the overcrossing  $
    c_{\rm {ov}}$. The overcrossings lying on  a 1-stratum $d$
    of $D$  split $d$ into  consecutive   segments called  {\it
    underpasses}. If $d$  contains no overcrossings (i.e., $d$
    is embedded and lies below  all the other 1-strata), then by
  definition, $d$ has   one underpass  equal to $ d$.

A crossing   $c$ of $D$ determines three underpasses of
(1-strata  of) $D$: the underpass $\underline c$ containing the
point   $c_{\rm {un}}$ and   two underpasses ${c^{-}}, {c^{+}}$
separated by the  point  $c_{\rm {ov}}$. One of the underpasses
${c^{-}}, {c^{+}}$  is directed towards $c_{\rm {ov}}$ and the
other one is directed away from $c_{\rm {ov}}$. We choose
notation so that  ${c^{+}}$ is directed towards $c_{\rm {ov}}$
if the crossing $c$ is   positive  and away from $c_{\rm {ov}}$
if $c$ is   negative, see Figure~\ref{fig-underpass-crossing}.

\begin{figure}[t]
\begin{center}
\psfrag{u}[Bl][Bc]{\scalebox{.9}{$\underline{c}$}}
\psfrag{x}[Bl][Bl]{\scalebox{.9}{$\underline{c}$}}
\psfrag{c}[Bl][Bl]{\scalebox{.9}{$c^-$}}
\psfrag{r}[Bc][Bc]{\scalebox{.9}{$c^+$}}
\rsdraw{.45}{.9}{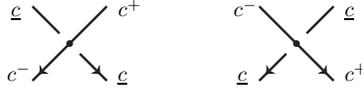}
\end{center}
\caption{The underpasses associated with a crossing $c$}
\label{fig-underpass-crossing}
\end{figure}

\subsection{Colorings of diagrams}
  A {\it   ${\mathcal C}$-pre-coloring} or shorter a  {\it
 pre-coloring} $U$  of  a
graph diagram  $D$  comprises two functions. The first function
assigns to every     underpass $p$ of   (a 1-stratum of) $D$ a non-zero
homogeneous object $U_p$ of $ {{\mathcal C}} $ called the {\it
color} of $p$.    The second function   assigns
 to every   crossing   $c$ of   $D$  an isomorphism
\begin{equation}\label{weakiso-}
U_{c}\colon  U_{   {{c^{+}}}} \to \varphi_{ \vert U_{  {\underline c}} \vert  } (U_{  {{c^{- }}}})
\end{equation}
called   the {\it color} of $c$. The existence of such an
isomorphism   together with the fact that the colors are non-zero objects   implies that $ \vert U_{   {{c^{+}}}} \vert =
  \vert U_{  {\underline c}} \vert^{-1} \vert U_{  {{c^{-
  }}}}\vert \vert U_{  {\underline c}} \vert $ for all $c$.  A
  pre-colored diagram $D$ has a source/target defined similarly
  to the source/target of a pre-colored ribbon graph  but
  using  the orientations and the colors of the underpasses  of
  $D$ adjacent to the inputs/outputs.

A {\it   ${\mathcal C}$-coloring} or shorter a {\it coloring} of
$D$ consists of   a  pre-coloring $U$ and a   function $V$
assigning to every coupon $Q $ of $D$    a     morphism $ V_Q$
in $ {{\mathcal C}} $ satisfying  the following conditions. Let
${p}_1,\ldots , {p}_m$   be the   underpasses of $D$ incident to
the bottom  side of $Q$ enumerated from the left to the right
(i.e., in the order determined by the direction on $\partial Q$
induced by the orientation of $Q$). Set   $\varepsilon_i =+$   if
${p}_i$ is directed
  out of $Q$ and   $\varepsilon_i =-$   otherwise. Let
${p}^1,\ldots ,{p}^n$ be the    underpasses of $D$ incident to
the top  side of $Q$ enumerated from the left to the right
(i.e.,  in the order determined by the direction on $\partial Q$
opposite to the one induced by the orientation of $Q$). Set
  $\varepsilon^j =-$     if ${p}^j$ is directed out of $Q$ and
   $\varepsilon^j =+$    otherwise.
We require that
$$
\prod_{i=1}^m \vert
U_{p_i}\vert^{\varepsilon_i}=\prod_{j=1}^n \vert
U_{p^j}\vert^{\varepsilon^j}\quad {\text {and}} \quad V_Q \in
\Hom_{{\mathcal C}}({{\otimes}}_{i=1}^m
U_{p_i}^{\varepsilon_i},\, {{\otimes}}_{j=1}^n
 U_{p^j}^{\varepsilon^j})  .
$$
As above, $X^+=X$ and $X^-=X^*$ for any $X\in \cc$.   By a \emph{$\cc$-colored diagram} or shorter a \emph{colored diagram}, we mean a graph diagram endowed with a $\cc$-coloring.

\begin{figure}[t]
\begin{center}
       \subfigure[{$\protect \overrightarrow{\cap}_X$}]{\quad\psfrag{X}[Bc][Bc]{\scalebox{.9}{$X$}} \rsdraw{.45}{.9}{leval}\quad} \qquad\;
       \subfigure[{$\protect \overrightarrow{\cup}_X$}]{\quad\psfrag{X}[Bc][Bc]{\scalebox{.9}{$X$}} \rsdraw{.45}{.9}{lcoeval}\quad} \qquad\;
       \subfigure[{$\protect \overleftarrow{\cap}_X$}]{\quad\psfrag{X}[Bc][Bc]{\scalebox{.9}{$X$}} \rsdraw{.45}{.9}{reval}\quad} \qquad\;
       \subfigure[{$\protect \overleftarrow{\cup}_X$}]{\quad\psfrag{X}[Bc][Bc]{\scalebox{.9}{$X$}} \rsdraw{.45}{.9}{rcoeval}\quad} \\[.4em]
       \subfigure[{$\sigma_+(X,Y,X',\psi)$}]{\psfrag{U}[Br][Br]{\scalebox{.9}{$Y$}}
       \psfrag{V}[Bl][Bl]{\scalebox{.9}{$X'$}}\psfrag{A}[Br][Br]{\scalebox{.9}{$X$}}
       \psfrag{B}[Bl][Bl]{\scalebox{.9}{$Y$}} \psfrag{p}[Bl][Bl]{\scalebox{.9}{$\psi\co X' \to \varphi_{|Y|}(X)$}}\rsdraw{.45}{.9}{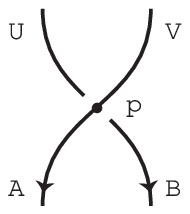}\quad} \qquad\;
       \subfigure[{$\sigma_-(X,Y,Y',\psi)$}]{\quad\psfrag{U}[Br][Br]{\scalebox{.9}{$Y'$}}
       \psfrag{V}[Bl][Bl]{\scalebox{.9}{$X$}}\psfrag{A}[Br][Br]{\scalebox{.9}{$X$}}
       \psfrag{B}[Bl][Bl]{\scalebox{.9}{$Y$}} \psfrag{p}[Bl][Bl]{\scalebox{.9}{$\psi\co Y \to \varphi_{|X|}(Y')$}}\rsdraw{.45}{.9}{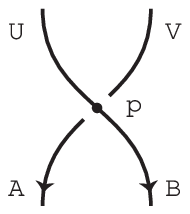}\quad} \qquad\qquad
       \subfigure[{$D(v)$}]{\quad\psfrag{U}[Br][Br]{\scalebox{.9}{$(U^1,\varepsilon^1)$}}
       \psfrag{V}[Bl][Bl]{\scalebox{.9}{$(U^n,\varepsilon^n)$}}\psfrag{A}[Br][Br]{\scalebox{.9}{$(U_1,\varepsilon_1)$}}
       \psfrag{B}[Bl][Bl]{\scalebox{.9}{$(U_m,\varepsilon_m)$}} \psfrag{v}[Bc][Bc]{\scalebox{1}{$v$}}\rsdraw{.45}{.9}{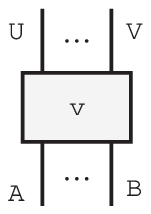}\quad}\qquad\qquad
\end{center}
\caption{Elementary colored diagrams}
\label{fig-elem-diags}
\end{figure}

Examples of colored diagrams (and notation for them) are given
in   Figure~\ref{fig-elem-diags}.   Here we mark  the
overcrossings by dots and indicate the colors of the underpasses
and of the crossings. In the first six  diagrams,  $X, Y, X',
Y'$ are any homogeneous objects of $\mathcal C$ and $\psi$ is
any isomorphism. The seventh diagram   is formed by a coupon   colored by a morphism $v \in \Hom_{\cc}\bigl(\otimes_{i=1}^m
U_i^{\varepsilon_i}, \otimes_{j=1}^n  (U^j)^{\varepsilon^j}\bigr)$   and
$m+n$ vertical segments. The  diagrams in
Figure~\ref{fig-elem-diags}   are called {\it elementary
diagrams}. The elementary diagrams $\sigma_+$ and $\sigma_-$ should not be confused with the pictures used    in   Section~\ref{sect-braided-ribbon} to represent the $G$-braiding and its inverse. In the latter, the crossings are not decorated with isomorphisms.

An   {\it   isomorphism}      $U\approx U'$ of pre-colorings
 of $D$   is a system of isomorphisms $f=\{f_p \colon U_p \to
 U'_p\}_p$, where $p$ runs over all   underpasses of   $D$, such
 that
  for any   crossing   $c$, the following diagram   commutes
\begin{equation}
\begin{split}\label{weakis+}\xymatrix@R=0.7cm @C=2cm {
U_{   {{c^{+}}}} \ar[r]^-{U_{  c}}\ar[d]_{f_{{c^{+}}}} & \varphi_{\vert U_{  {\underline c}} \vert } (U_{  {{c^{-}}}})\ar[d]^{\varphi_{\vert U_{  {\underline c}} \vert } (f_{  {{c^{-}}}})} \\
U'_{   {{c^{+}}}} \ar[r]^-{U'_{  c} } & \varphi_{\vert U'_{
{\underline c}} \vert } (U'_{ {{c^{-}}}}).}\end{split}
\end{equation}
Here $\vert U_{ {\underline c}} \vert=\vert U'_{ {\underline
c}} \vert$ because isomorphic non-zero homogeneous objects of $\cc$ have the
same grading.

Let $(U,V)$ and $(U', V')$  be colorings of $D$ with the same source
and target. Thus, $U_p = U'_p$ for any underpass $p$ of $D$ adjacent
to an input
 or an output of $D$.  An {\it isomorphism}     $(U,V)\to (U',
 V')$  is an isomorphism of pre-colorings $f \colon U \to U'$
 such that for any underpass $p$   adjacent to an input or an
 output of $D$, we have $f_p=\id\colon U_p \to U'_p$  and for
 any coupon $Q$ of $D$, the following diagram (in the notation
 above) commutes:
\begin{equation}\begin{split}\label{weakisomor}\xymatrix@R=1cm @C=2cm {
{{\otimes}}_{i=1}^m
U_{  p_i}^{\varepsilon_i}  \ar[r]^-{V_{   Q}}\ar[d]_{\otimes_{i=1}^m
f_{  p_i}^{\varepsilon_i}} &  {{\otimes}}_{j=1}^n U_{  p^j}^{\varepsilon^j}
 \ar[d]^{\otimes_{j=1}^n f_{  p^j}^{\varepsilon^j}}\\
 {{\otimes}}_{i=1}^m (U'_{  p_i})^{\varepsilon_i}  \ar[r]^-{V'_{Q}} &  {{\otimes}}_{j=1}^n (U'_{  p^j})^{\varepsilon^j}.
 }\end{split}\end{equation}

By {\it   isotopy} of a colored diagram, we mean an ambient
 isotopy   of the diagram  in  $\RR \times [0,1]$
 keeping the inputs, the outputs,  the orientations of 1-strata,
 the over/under-data in the crossings,  and all the colors.   We
 call two colored diagrams {\it color-equivalent} if they may be
 obtained from each other through  isotopy and isomorphism of
 colorings. Color-equivalent colored diagrams necessarily have
 the same source and the same target.

\subsection{The category ${\mathcal D}_{\mathcal C} $}
We define a
strict   monoidal category ${\mathcal  D}={\mathcal D}_{\mathcal C}
$ of $\mathcal C$-colored diagrams. This category has the same
objects as  the category $ {\mathcal G}_{\mathcal C} $
 from Section~\ref{sect-category-of-graphs}. The   monoidal product
of objects and the unit object  in ${\mathcal  D}$ are   the same as
in~${\mathcal G}_{\mathcal C}$. A morphism $((U_r,
\varepsilon_r))_{r=1}^k  \to ((U^s, \varepsilon^s))_{s=1}^l$ in
$\mathcal  D$ is a color-equivalence class of    $\mathcal
C$-colored diagrams with   source   $((U_r,
\varepsilon_r))_{r=1}^k$ and  target    $((U^s,
\varepsilon^s))_{s=1}^l$.   The identity morphism of an object
$((U_r, \varepsilon_r))_{r=1}^k$ is represented by the colored
diagram formed by $k$ disjoint vertical segments with source and
target $((U_r, \varepsilon_r))_{r=1}^k$.   The composition of
morphisms represented by colored diagrams $D, D'$ is obtained by
gluing $D$ on  top of $D'$.  The    monoidal    product of the morphisms
represented by  $D,D'$ is obtained by placing $D'$ to the right of
$D$. All axioms of a strict   monoidal category are
straightforward. By abuse of   language, we shall make no
      difference between a colored diagram and the corresponding
      morphism in $\mathcal D$.

\begin{lem}\label{lem-functorF}
   If $(\mathcal C, \varphi, \tau)$ is a  pivotal $G$-braided    category, then
   there is a unique strong monoidal functor ${\mathcal F}=({\mathcal F}, {\mathcal F}_2,
   {\mathcal F}_0)\colon {\mathcal  D}_{\mathcal C} \to  {\mathcal C}$ such
   that:
\begin{itemize}
\item ${\mathcal F}$ carries any object  $ ((U_r, \varepsilon_r))_{r=1}^k$
   of ${\mathcal D}_{\mathcal C}$ to $\otimes_{r=1}^k
   U_r^{\varepsilon_r}$;

\item  $ {\mathcal F}_0\co \un \to {\mathcal F}(\emptyset) =\un$ is the identity
    morphism;

\item for any objects $X=((U_r, \varepsilon_r))_{r=1}^k $ and
    $Y= ((V_s, \mu_s))_{s= 1}^l$ of $\mathcal D$, the morphism
    ${\mathcal F}_2 (X, Y) \co  {\mathcal F}(X)\otimes {\mathcal F}(Y) \to {\mathcal F}(X\otimes Y)$  is the
    canonical
   isomorphism
    $$  (U_1^{\varepsilon_1} \otimes \cdots \otimes
   U_k^{\varepsilon_k})  \otimes (V_1^{\mu_1} \otimes
   \cdots \otimes V_l^{\mu_l})  \cong
   U_1^{\varepsilon_1} \otimes \cdots \otimes
   U_k^{\varepsilon_k} \otimes V_1^{\mu_1} \otimes
   \cdots \otimes V_l^{\mu_l} $$ determined by the associativity constraints in
    $\cc$;

\item  ${\mathcal F}$ carries elementary diagrams to the following
     morphisms:
  $${\mathcal F}(\overrightarrow {\cap}_X)=\lev_X, \;{\mathcal F}(\overleftarrow {\cap}_X)=\rev_X,  \;{\mathcal F}(\overrightarrow {\cup}_X)=\lcoev_X, \; {\mathcal F}(\overleftarrow
   {\cup}_X)=\rcoev_X ,\;  {\mathcal F}(D_v)=v,$$
 \begin{equation}\label{Fother3} {\mathcal F}(\sigma_+(X,Y, X', \psi))=
 (\id_Y\otimes \psi^{-1}) \tau_{X,Y}\co  X\otimes Y \to Y\otimes
 X',\end{equation}
 \begin{equation}\label{Fother4}{\mathcal F}(\sigma_-(X,Y, Y', \psi))=  \tau_{ Y',X}^{-1} (\id_X\otimes \psi) \co  X\otimes Y \to Y'\otimes X.\end{equation}
 \end{itemize}
\end{lem}

\begin{proof}   The uniqueness of ${\mathcal F}$ is obvious because all morphisms in ${\mathcal  D}_{\mathcal C}$ can be obtained from the   elementary   diagrams using composition and    monoidal    product.
   The existence of ${\mathcal F}$ is a direct consequence of the axioms of
   a   pivotal  category, cf.\ Sections~\ref{pivotall} and~\ref{sect-pivot-cross-Gcat}.
\end{proof}

\section{Colored Reidemeister moves}

\subsection{The moves}
 We define  local transformations of colored   diagrams
 called {\it colored Reidemeister moves}. These moves preserve a
 diagram outside a 2-disk and modify the diagram in the disk as
  shown in Figures~\ref{fig-T1moves}--\ref{fig-T4moves2}.
There are four moves of type 1,  four moves of type 2, one move
 of type 3, and    four moves   of type 4.  We now specify the
 behavior of colorings under the moves.

Each type 1 or type 2 move creates two new crossings and
a new underpass with endpoints in these crossings.   The color of
 this   underpass  may be an arbitrary object of   $\mathcal  C$
 such that there is an isomorphism $\psi$ as in Figures~\ref{fig-T1moves} and~\ref{fig-T2moves}.
Both new crossings are colored with the same  $\psi$. Note
that under the type 1 moves,   $\vert X\vert=\vert X'\vert$.

 The morphisms $A,B,C,A',B' $ in the type 3 move are any
 isomorphisms in $\mathcal C$ as indicated such that  the
 following diagram   commutes:
 \begin{equation}\begin{split}\label{move3condition}\xymatrix@R=1cm @C=1.65cm {
\varphi_{\vert Y'\vert } (\widetilde X) \ar[d]^{\varphi_{\vert Y'\vert } (A')} & \ar[l]_{B'} X'' \ar[r]^-{B} &  \varphi_{\vert Z\vert} (X')
\ar[d]_{\varphi_{\vert Z \vert} (A)} \\
\varphi_{\vert Y'\vert } \varphi_{\vert Z\vert } (X)    \ar[r]^-{ \varphi_2(\vert Y'\vert, \vert Z\vert)_X} & \varphi_{\vert Z\vert\vert Y'\vert} (X)=\varphi_{\vert Y\vert\vert Z\vert} (X)  &  \ar[l]_{\quad \quad \quad \varphi_2(\vert Z\vert, \vert Y\vert)_X} \varphi_{\vert Z\vert} \varphi_{\vert Y\vert} (X) .}\end{split}\end{equation}
  The equality $\vert Z\vert\vert Y'\vert= \vert Y\vert\vert
Z\vert$  follows from the existence of the isomorphism $C$   and the fact that the objects $Y,Y'$ are non-zero.

 \begin{figure}[t]
\begin{center}
       \subfigure[{$\psi\co X'\to \varphi_{|X|}(X)$}]{\quad\psfrag{X}[Bl][Bl]{\scalebox{.9}{$X$}} \psfrag{B}[Bl][Bl]{\scalebox{.9}{$X'$}} \psfrag{p}[Br][Br]{\scalebox{.9}{$\psi$}} \rsdraw{.45}{.9}{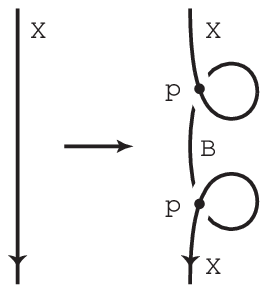}\quad} \qquad\;
   \subfigure[{$\psi\co X\to \varphi_{|X|}(X')$}]{\quad\psfrag{X}[Bl][Bl]{\scalebox{.9}{$X$}} \psfrag{B}[Bl][Bl]{\scalebox{.9}{$X'$}} \psfrag{p}[Br][Br]{\scalebox{.9}{$\psi$}} \rsdraw{.45}{.9}{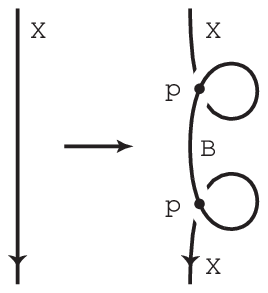}\quad} \\[.6em]
\subfigure[{$\psi\co X\to \varphi_{|X|}(X')$}]{\quad\psfrag{Y}[Bl][Bl]{\scalebox{.9}{$X$}}\psfrag{X}[Br][Br]{\scalebox{.9}{$X$}} \psfrag{B}[Br][Br]{\scalebox{.9}{$X'$}} \psfrag{p}[Bl][Bl]{\scalebox{.9}{$\psi$}} \rsdraw{.45}{.9}{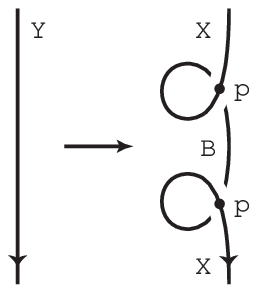}\quad} \qquad\;
   \subfigure[{$\psi\co X'\to \varphi_{|X|}(X)$}]{\quad\psfrag{Y}[Bl][Bl]{\scalebox{.9}{$X$}}\psfrag{X}[Br][Br]{\scalebox{.9}{$X$}} \psfrag{B}[Br][Br]{\scalebox{.9}{$X'$}} \psfrag{p}[Bl][Bl]{\scalebox{.9}{$\psi$}} \rsdraw{.45}{.9}{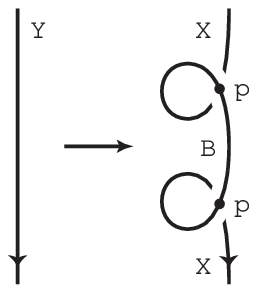}\quad}
\end{center}
\caption{Type 1 moves}
\label{fig-T1moves}
\end{figure}
\begin{figure}[t]
\begin{center}
      \subfigure[{$\psi\co Y\to \varphi_{|X|}(Y')$}]{\quad\psfrag{X}[Br][Br]{\scalebox{.9}{$X$}} \psfrag{A}[Br][Br]{\scalebox{.9}{$Y'$}} \psfrag{Y}[Bl][Bl]{\scalebox{.9}{$Y$}}\psfrag{B}[Bl][Bl]{\scalebox{.9}{$X$}} \psfrag{p}[Bl][Bl]{\scalebox{.9}{$\psi$}} \rsdraw{.45}{.9}{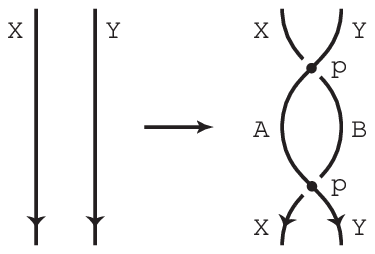}\quad} \qquad\;
                \subfigure[{$\psi\co X'\to \varphi_{|Y|}(X)$}]{\quad\psfrag{X}[Br][Br]{\scalebox{.9}{$X$}} \psfrag{A}[Br][Br]{\scalebox{.9}{$Y$}} \psfrag{Y}[Bl][Bl]{\scalebox{.9}{$Y$}}\psfrag{B}[Bl][Bl]{\scalebox{.9}{$X'$}} \psfrag{p}[Bl][Bl]{\scalebox{.9}{$\psi$}} \rsdraw{.45}{.9}{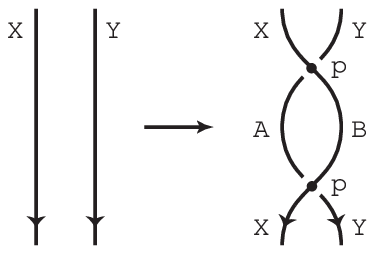}\quad}\\[.4em]
                \subfigure[{$\psi\co Y'\to \varphi_{|X|}(Y)$}]{\quad\psfrag{X}[Br][Br]{\scalebox{.9}{$X$}} \psfrag{A}[Br][Br]{\scalebox{.9}{$Y'$}} \psfrag{Y}[Bl][Bl]{\scalebox{.9}{$Y$}}\psfrag{B}[Bl][Bl]{\scalebox{.9}{$X$}} \psfrag{p}[Bl][Bl]{\scalebox{.9}{$\psi$}} \rsdraw{.45}{.9}{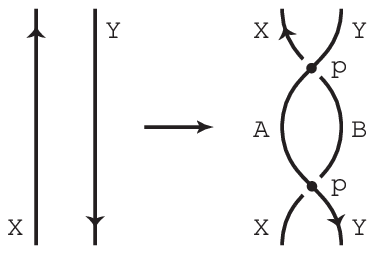}\quad} \qquad\;
                    \subfigure[{$\psi\co Y\to \varphi_{|X|}(Y')$}]{\quad\psfrag{X}[Br][Br]{\scalebox{.9}{$X$}} \psfrag{A}[Br][Br]{\scalebox{.9}{$Y'$}} \psfrag{Y}[Bl][Bl]{\scalebox{.9}{$Y$}}\psfrag{B}[Bl][Bl]{\scalebox{.9}{$X$}} \psfrag{p}[Bl][Bl]{\scalebox{.9}{$\psi$}} \rsdraw{.45}{.9}{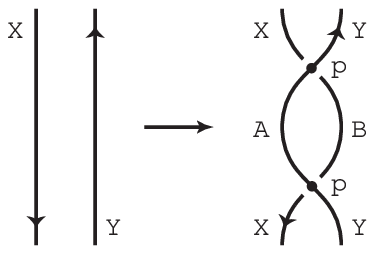}\quad}
\end{center}
\caption{Type 2 moves}
\label{fig-T2moves}
\end{figure}
\begin{figure}[t]
\begin{center}
\psfrag{X}[Br][Br]{\scalebox{.9}{$X$}}
\psfrag{Y}[Br][Br]{\scalebox{.9}{$Y$}}
\psfrag{Z}[Br][Br]{\scalebox{.9}{$Z$}}
\psfrag{R}[Bl][Bl]{\scalebox{.9}{$X'$}}
\psfrag{P}[Br][Br]{\scalebox{.9}{$X''$}}
\psfrag{T}[Br][Br]{\scalebox{.9}{$Y'$}}
\psfrag{S}[Bl][Bl]{\scalebox{.9}{$Y'$}}
\psfrag{A}[Br][Br]{\scalebox{.9}{$\widetilde{X}$}}
\psfrag{a}[Br][Br]{\scalebox{.9}{$A$}}
\psfrag{b}[Bl][Bl]{\scalebox{.9}{$B$}}
\psfrag{g}[Bl][Bl]{\scalebox{.9}{$C$}}
\psfrag{r}[Br][Br]{\scalebox{.9}{$A'$}}
\psfrag{l}[Bl][Bl]{\scalebox{.9}{$B'$}}
\psfrag{q}[Bl][Bl]{\scalebox{.9}{$C$}}
\psfrag{u}[Bl][Bl]{\scalebox{.9}{$A\co X'\to\varphi_{|Y|}(X)$}}
\psfrag{v}[Bl][Bl]{\scalebox{.9}{$B\co X''\to\varphi_{|Z|}(X')$}}
\psfrag{w}[Bl][Bl]{\scalebox{.9}{$C\co Y'\to\varphi_{|Z|}(Y)$}}
\psfrag{e}[Bl][Bl]{\scalebox{.9}{$A'\co \widetilde{X}\to\varphi_{|Z|}(X)$}}
\psfrag{s}[Bl][Bl]{\scalebox{.9}{$B'\co X''\to\varphi_{|Y'|}(\widetilde{X})$}}
\psfrag{o}[Bl][Bl]{\scalebox{.9}{}}
\rsdraw{.45}{.9}{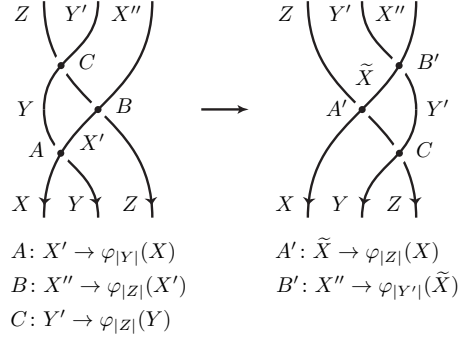}
\end{center}
\caption{Type 3 move}
\label{fig-T3move}
\end{figure}

\begin{figure}[t]
\begin{center}
{\label{fig-sub-T4a}%
\psfrag{v}[Bc][Bc]{\scalebox{.9}{$v_0$}}
\psfrag{n}[Bc][Bc]{\scalebox{.9}{$v_1$}}
\psfrag{Z}[Bc][Bc]{\scalebox{.9}{$Z$}}
\psfrag{X}[Bl][Bl]{\scalebox{.9}{$X_1$}}
\psfrag{Y}[Bl][Bl]{\scalebox{.9}{$X_2$}}
\psfrag{L}[Bl][Bl]{\scalebox{.9}{$X_m$}}
\psfrag{U}[Bl][Bl]{\scalebox{.9}{$Y^1$}}
\psfrag{V}[Bl][Bl]{\scalebox{.9}{$Y^2$}}
\psfrag{T}[Bl][Bl]{\scalebox{.9}{$Y^n$}}
\psfrag{P}[Br][Br]{\scalebox{.9}{$X^1$}}
\psfrag{S}[Br][Br]{\scalebox{.9}{$X^2$}}
\psfrag{R}[Br][Br]{\scalebox{.9}{$X^n$}}
\psfrag{A}[Br][Br]{\scalebox{.9}{$Y_1$}}
\psfrag{B}[Br][Br]{\scalebox{.9}{$Y_2$}}
\psfrag{C}[Br][Br]{\scalebox{.9}{$Y_m$}}
\psfrag{t}[Br][Br]{\scalebox{.9}{$\psi_1$}}
\psfrag{h}[Br][Br]{\scalebox{.9}{$\psi_2$}}
\psfrag{k}[Bl][Bl]{\scalebox{.9}{$\psi_m$}}
\psfrag{p}[Br][Br]{\scalebox{.9}{$\psi^1$}}
\psfrag{q}[Br][Br]{\scalebox{.9}{$\psi^2$}}
\psfrag{g}[Bl][Bl]{\scalebox{.9}{$\psi^n$}}
\rsdraw{.45}{.9}{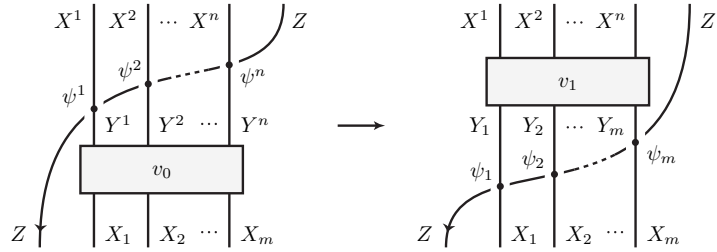}}\\[.6em]
\end{center}
\caption{The first   move  of type 4}
\label{fig-T4moves1}
\end{figure}

\begin{figure}[t]
\begin{center}
{\label{fig-sub-T4c}%
\psfrag{v}[Bc][Bc]{\scalebox{.9}{$v$}}
\psfrag{X}[Bl][Bl]{\scalebox{.9}{$Y_1$}}
\psfrag{Y}[Bl][Bl]{\scalebox{.9}{$Y_2$}}
\psfrag{L}[Bl][Bl]{\scalebox{.9}{$Y_m$}}
\psfrag{D}[Bl][Bl]{\scalebox{.9}{$X_m=X^n$}}
\psfrag{P}[Br][Br]{\scalebox{.9}{$Y^1$}}
\psfrag{S}[Br][Br]{\scalebox{.9}{$Y^2$}}
\psfrag{R}[Br][Br]{\scalebox{.9}{$Y^n$}}
\psfrag{T}[Bl][Bl]{\scalebox{.9}{$X^n$}}
\psfrag{Z}[Br][Br]{\scalebox{.9}{$X^0$}}
\psfrag{A}[Br][Br]{\scalebox{.9}{$X^0=X_0$}}
\psfrag{B}[Bc][Bc]{\scalebox{.9}{$X_1$}}
\psfrag{C}[Bc][Bc]{\scalebox{.9}{$X_2$}}
\psfrag{U}[Bc][Bc]{\scalebox{.9}{$X^1$}}
\psfrag{V}[Bc][Bc]{\scalebox{.9}{$X^2$}}
\psfrag{p}[Bl][Bl]{\scalebox{.9}{$\psi^1$}}
\psfrag{q}[Bl][Bl]{\scalebox{.9}{$\psi^2$}}
\psfrag{g}[Bl][Bl]{\scalebox{.9}{$\psi^n$}}
\psfrag{t}[Bl][Bl]{\scalebox{.9}{$\psi_1$}}
\psfrag{h}[Bl][Bl]{\scalebox{.9}{$\psi_2$}}
\psfrag{k}[Bl][Bl]{\scalebox{.9}{$\psi_m$}}
\rsdraw{.45}{.9}{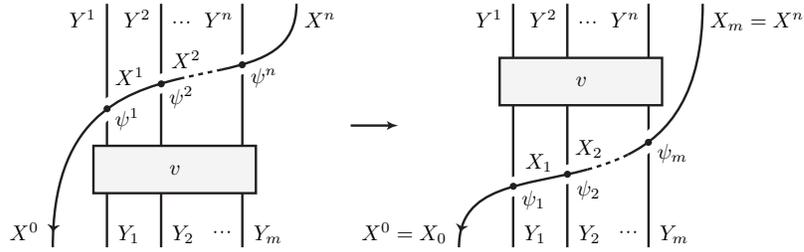}}
\end{center}
\caption{The   third  move  of type 4}
\label{fig-T4moves2}
\end{figure}

To describe the  type 4 moves, we   use   notation   $\overline{\psi}$ and $\psi^-$   introduced in Section~\ref{sect-pivot-cross-Gcat}.
 Let $Q$ be a coupon of a colored   diagram with $m$ entries
  and $n$ exits. The first    type 4  move  pushes   a
  $Z$-colored underpass behind $Q$, see
  Figure~\ref{fig-T4moves1}.  The  color $v_0$ of $Q$ is
  transformed into $v_1$.  There is only one requirement on
   $v_0$, $v_1$ and the isomorphisms $\psi_i, \psi^j$. Namely,
  the following diagram should commute:
\begin{equation}\begin{split}\label{fourthRmove}
 \xymatrix@R=1cm @C=2cm {
\otimes_{i=1}^m X_i^{\varepsilon_i} \ar[r]^-{\otimes_{i=1}^m  \psi_i^{\varepsilon_i}}\ar[d]_{v_0} & \otimes_{i=1}^m \varphi_{\mu}  (Y_i^{\varepsilon_i}) \ar[r]^-{(\varphi_{\mu})_m}  &
 \varphi_{\mu} ( \otimes_{i=1}^m    Y_i^{\varepsilon_i})  \ar[d]^-{\varphi_{\mu}(v_1)}\\
\otimes_{j=1}^n (Y^j)^{\varepsilon^j}   \ar[r]^-{ \otimes_{j=1}^n  (\psi^j)^{\varepsilon^j}} & \otimes_{j=1}^n \varphi_{\mu} ((X^j)^{\varepsilon^j})
\ar[r]^{(\varphi_{\mu})_n } &  \varphi_{\mu} (\otimes_{j=1}^n (X^j)^{\varepsilon^j}).}
\end{split}\end{equation}
 Here $\mu=\vert Z\vert\in G $ and      $\varepsilon_i ,
   \varepsilon^j=\pm$    are the signs determined by the $i$-th
   entry and $j$-th exit of $Q$ as in Section
   \ref{coloredGgraphs}. The second  type 4 move is obtained
   from the previous one  by inverting orientation on the
   $Z$-colored underpass (before and after the move) and
   replacing $\mu, \psi_i, \psi^j$ in   \eqref{fourthRmove}  by
   $\mu^{-1}=\vert Z\vert^{-1}$, $\overline   { \psi_i},
   \overline {\psi^j}$, respectively.

 The third   type 4 move pushes a branch of the diagram in front
of $Q$ keeping the color $v$ of $Q$,  see
Figure~\ref{fig-T4moves2}. There is only one
requirement on the  isomorphisms $\psi_i, \psi^j$. Set
$y_i=\vert Y_i\vert^{\varepsilon_i} \in G$ and $y^j= \vert
Y^j\vert^{\varepsilon^j} \in G$ for all $i,j$. Suppose first
that   $\varepsilon_i = \varepsilon^j=+$   for all $i, j$, i.e.,
that all the segments adjacent to $Q$ are directed downwards. We
require that
the composition
\begin{equation}\label{comp1}\begin{split}
\xymatrix@R=1cm @C=2cm {
X^n  \ar[r]^-{ \psi^n }  \ar[r]^-{\psi^n}  & \varphi_{y^n } (X^{n-1})   \ar[r]^-{
\varphi_{y^n } (\psi^{n-1})}  &    \varphi_{y^n }  \varphi_{y^{n-1} } (X^{n-2}) \ar[r]^-{
\varphi_{y^n } \varphi_{y^{n-1} } (\psi^{n-2})}   &  }\\
 \xymatrix@R=1cm @C=2cm {   \ldots    \ar[r]^-{
\varphi_{y^n }\cdots  \varphi_{y^{2} } (\psi^{1})}  &  \varphi_{y^n } \cdots  \varphi_{y^{ 1}  } (X^{0})  \ar[r]^-{
\varphi_{n}} &  \varphi_{y^1    \cdots  y^{ n}  } (X^{0}) }
\end{split}\end{equation}
is equal to   the composition
\begin{equation}\label{comp2}\begin{split}
\xymatrix@R=1cm @C=1.7cm {
X_m  \ar[r]^-{ \psi_m }  \ar[r]^-{\psi_m}  & \varphi_{y_{m} } (X_{m-1})   \ar[r]^-{
\varphi_{y_{m}} (\psi_{m-1})}  &    \varphi_{y_{m}}  \varphi_{y_{m-1} } (X_{m-2}) \,\,\,\, \ar[r]^-{
\varphi_{y_{m}} \varphi_{y_{m-1} } (\psi_{m-2})}  &  }\\
\xymatrix@R=1cm @C=2cm { \ldots    \ar[r]^-{
\varphi_{y_{m}}\cdots  \varphi_{y_{2}   } (\psi_{1})}  &
\varphi_{y_{m}} \cdots  \varphi_{y_{ 1}  } (X_{0})  \ar[r]^-{
\varphi_{m}} & \varphi_{y_1   \cdots   y_{ m} } (X_{0}) .}
\end{split}\end{equation}
In the general  case,  whenever
  $\varepsilon^j=-$ (resp.\@ $\varepsilon_i=-$),   one should
replace   here $\psi^j$ by $\overline {\psi^j}$ (resp.\@ replace
$\psi_i$ by ${\overline \psi_i}$).   Note that $X^n=X_m$,
$X^0=X_0$, and $y^1\cdots y^n=y_1\cdots y_m$ so that the source
and   target objects of both compositions are the same.

The fourth type 4 move is obtained from the previous
one by inverting orientation on the long branch  (before and
after the move). The rest of the notation and the condition on
the $\psi$'s are the same.

The  moves inverse  to the colored Reidemeister moves above
are also called colored Reidemeister moves.

We shall need one more move on colored  diagrams shown in Figure
 \ref{stabi} and called  {\it stabilization}. This move inserts
 a coupon inside   a downward-oriented branch of an  underpass.
 If this underpass is colored with $X$, then both underpasses
 adjoint to the new coupon are colored with $X$, and the coupon
 is colored with
$\id_X$. The rest of the diagram is preserved including the
coloring.

\begin{thm}\label{thm-functorF+}
For any $G$-ribbon  category  ${\mathcal C}$, the  functor
${\mathcal F}\colon {\mathcal  D}_{\mathcal C} \to  {\mathcal C}$ of Lemma~\ref{lem-functorF} is
invariant under the colored Reidemeister moves and under
stabilization.
\end{thm}
The proof of this theorem  is based on the following lemma:
\begin{figure}[t]
\begin{center}
\begin{align*}
T_+(X,X',\psi)& \psfrag{X}[Bl][Bl]{\scalebox{.9}{$X$}} \psfrag{B}[Bl][Bl]{\scalebox{.9}{$X'$}} \psfrag{p}[Br][Br]{\scalebox{.9}{$\psi$}}
\psfrag{r}[Bc][Bc]{\scalebox{.9}{$\psi\co X'\to \varphi_{|X|}(X)$}}=\;\;\rsdraw{.45}{.9}{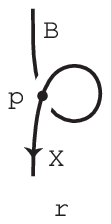}\,,
&
T_-(X,X',\psi)& \psfrag{X}[Bl][Bl]{\scalebox{.9}{$X$}} \psfrag{B}[Bl][Bl]{\scalebox{.9}{$X'$}} \psfrag{p}[Br][Br]{\scalebox{.9}{$\psi$}}
\psfrag{r}[Bc][Bc]{\scalebox{.9}{$\psi\co X\to \varphi_{|X|}(X')$}} =\;\;\rsdraw{.45}{.9}{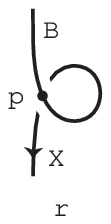}
\,,\\[.5em]
T'_+(X,X',\psi)& \psfrag{X}[Br][Br]{\scalebox{.9}{$X$}} \psfrag{B}[Br][Br]{\scalebox{.9}{$X'$}} \psfrag{p}[Bl][Bl]{\scalebox{.9}{$\psi$}}
\psfrag{r}[Bc][Bc]{\scalebox{.9}{$\psi\co X'\to \varphi_{|X|}(X)$}} =\;\;\rsdraw{.45}{.9}{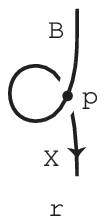}\;\,,
&
T'_-(X,X',\psi)& \psfrag{X}[Br][Br]{\scalebox{.9}{$X$}} \psfrag{B}[Br][Br]{\scalebox{.9}{$X'$}} \psfrag{p}[Bl][Bl]{\scalebox{.9}{$\psi$}}
\psfrag{r}[Bc][Bc]{\scalebox{.9}{$\psi\co X\to \varphi_{|X|}(X')$}}=\;\;\rsdraw{.45}{.9}{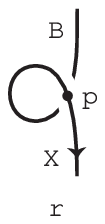}
\,\;,\\[.9em]
\sigma'_+(X,Y,Y',\psi)& \psfrag{X}[Br][Br]{\scalebox{.9}{$X$}} \psfrag{Y}[Bl][Bl]{\scalebox{.9}{$Y$}}
\psfrag{R}[Br][Br]{\scalebox{.9}{$Y'$}} \psfrag{Z}[Bl][Bl]{\scalebox{.9}{$X$}}
\psfrag{p}[Br][Br]{\scalebox{.9}{$\psi$}}
\psfrag{r}[Bc][Bc]{\scalebox{.9}{$\psi\co Y'\to\varphi_{|X|}(Y)$}} =\;\;\rsdraw{.6}{.9}{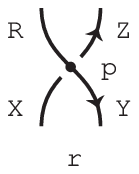}
\,\;,
&
\sigma'_-(X,Y,X',\psi)& \psfrag{X}[Br][Br]{\scalebox{.9}{$X$}} \psfrag{Y}[Bl][Bl]{\scalebox{.9}{$Y$}}
\psfrag{R}[Br][Br]{\scalebox{.9}{$Y$}} \psfrag{Z}[Bl][Bl]{\scalebox{.9}{$X'$}}
\psfrag{p}[Br][Br]{\scalebox{.9}{$\psi$}}
\psfrag{r}[Bc][Bc]{\scalebox{.9}{$\psi\co X\to\varphi_{|Y|}(X')$}} =\;\;\rsdraw{.6}{.9}{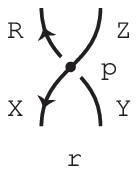}
\,\;,\\[.9em]
\sigma''_+(X,Y,Y',\psi)& \psfrag{X}[Br][Br]{\scalebox{.9}{$X$}} \psfrag{Y}[Bl][Bl]{\scalebox{.9}{$Y$}}
\psfrag{R}[Br][Br]{\scalebox{.9}{$Y'$}} \psfrag{Z}[Bl][Bl]{\scalebox{.9}{$X$}}
\psfrag{p}[Br][Br]{\scalebox{.9}{$\psi$}}
\psfrag{r}[Bc][Bc]{\scalebox{.9}{$\psi\co Y\to\varphi_{|X|}(Y')$}} =\;\;\rsdraw{.6}{.9}{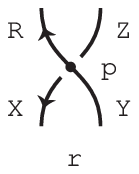}
\,\;,
&
\sigma''_-(X,Y,X',\psi)& \psfrag{X}[Br][Br]{\scalebox{.9}{$X$}} \psfrag{Y}[Bl][Bl]{\scalebox{.9}{$Y$}}
\psfrag{R}[Br][Br]{\scalebox{.9}{$Y$}} \psfrag{Z}[Bl][Bl]{\scalebox{.9}{$X'$}}
\psfrag{p}[Br][Br]{\scalebox{.9}{$\psi$}}
\psfrag{r}[Bc][Bc]{\scalebox{.9}{$\psi\co X'\to\varphi_{|Y|}(X)$}} =\;\;\rsdraw{.6}{.9}{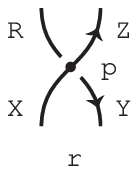}
\,\;,\\[.9em]
\sigma'''_+(X,Y,X',\psi)& \psfrag{X}[Br][Br]{\scalebox{.9}{$X$}} \psfrag{Y}[Bl][Bl]{\scalebox{.9}{$Y$}}
\psfrag{R}[Br][Br]{\scalebox{.9}{$Y$}} \psfrag{Z}[Bl][Bl]{\scalebox{.9}{$X'$}}
\psfrag{p}[Br][Br]{\scalebox{.9}{$\psi$}}
\psfrag{r}[Bc][Bc]{\scalebox{.9}{$\psi\co X\to\varphi_{|Y|}(X')$}} =\;\;\rsdraw{.6}{.9}{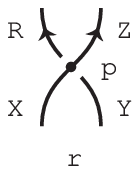}
\,\;,
&
\sigma'''_-(X,Y,Y',\psi)& \psfrag{X}[Br][Br]{\scalebox{.9}{$X$}} \psfrag{Y}[Bl][Bl]{\scalebox{.9}{$Y$}}
\psfrag{R}[Br][Br]{\scalebox{.9}{$Y'$}} \psfrag{Z}[Bl][Bl]{\scalebox{.9}{$X$}}
\psfrag{p}[Br][Br]{\scalebox{.9}{$\psi$}}
\psfrag{r}[Bc][Bc]{\scalebox{.9}{$\psi\co Y'\to\varphi_{|X|}(Y)$}} =\;\;\rsdraw{.6}{.9}{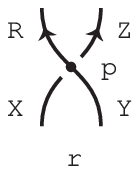}
\,\;.
\end{align*}
\end{center}
\caption{$\cc$-colored diagrams}
\label{fig-othergen}
\end{figure}
\begin{lem}\label{lem-F-othergen}
For any $G$-ribbon  category  ${\mathcal C}$, the images
under ${\mathcal F}\colon {\mathcal  D}_{\mathcal C} \to  {\mathcal
C}$ of the $\cc$-colored diagrams in Figure~\ref{fig-othergen} are:
\begin{align}
& {\mathcal F}\bigl(T_+(X,X',\psi)\bigr)={\mathcal F}\bigl(T'_+(X,X',\psi)\bigr)=\psi^{-1}\theta_X \quad ({\rm {here}} \quad \vert X\vert =\vert X'\vert), \label{Fother1}\\
& {\mathcal F}\bigl(T_-(X,X',\psi)\bigr)={\mathcal F}\bigl(T'_-(X,X',\psi)\bigr)=\theta_{X'}^{-1}\psi \quad ({\rm {here}} \quad \vert X\vert =\vert X'\vert),\label{Fother2}\\
& {\mathcal F}\bigl(\sigma'_+(X,Y,Y',\psi)\bigr)=\tau_{Y',X^*}^{-1}\bigl(\id_{X^*} \otimes \overline{\psi}\bigr), \label{Fother5}\\
& {\mathcal F}\bigl(\sigma'_-(X,Y,X',\psi)\bigr)=\bigl(\id_{Y^*} \otimes \overline{\psi}^{-1}\bigr)\tau_{X,Y^*}, \label{Fother6}\\
& {\mathcal F}\bigl(\sigma''_+(X,Y,Y',\psi)\bigr)= \tau_{Y'^*,X}^{-1}\bigl(\id_{X} \otimes \psi^-\bigr),\label{Fother7}\\
& {\mathcal F}\bigl(\sigma''_-(X,Y,X',\psi)\bigr)=\bigl(\id_{Y} \otimes (\psi^-)^{-1}\bigr)\tau_{X^*,Y},\label{Fother8}\\
& {\mathcal F}\bigl(\sigma'''_+(X,Y,X',\psi)\bigr)=\bigl(\id_{Y^*} \otimes (\overline{\psi^-})^{-1}\bigr)\tau_{X^*,Y^*} ,\label{Fother9}\\
& {\mathcal F}\bigl(\sigma'''_-(X,Y,Y',\psi)\bigr)=\tau_{Y'^*,X^*}^{-1}\bigl(\id_{X^*} \otimes \overline{\psi^-}\bigr).\label{Fother10}
\end{align}
\end{lem}
\begin{proof}
Equalities \eqref{Fother1} and \eqref{Fother2} follow from the
expressions for the twist and its inverse given in
   \eqref{eq-def-twist}, Lemma~\ref{lem-twist}, and
   Lemma~\ref{lem-twist-ribbon}. Since   $\sigma'_+(X,Y,Y',\psi)$ is color-equivalent to
$$
\psfrag{X}[Br][Br]{\scalebox{.9}{$X$}} \psfrag{Y}[Br][Br]{\scalebox{.9}{$Y$}}
\psfrag{R}[Bl][Bl]{\scalebox{.9}{$Y'$}} \psfrag{Z}[Bl][Bl]{\scalebox{.9}{$X$}}
\psfrag{p}[Br][Br]{\scalebox{.9}{$\psi$}} \rsdraw{.45}{.9}{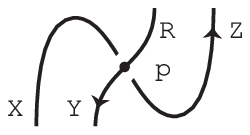}\;,
$$
we obtain from the definition of ${\mathcal F}$ that
$$
{\mathcal F}\bigl(\sigma'_+(X,Y,Y',\psi)\bigr)=(\lev_X \otimes \psi^{-1} \otimes \id_{X^*})(\id_{X^*} \otimes \tau_{Y,X} \otimes \id_X) (\id_{X^*\otimes Y} \otimes \lcoev_X).
$$
Note that
$(\varphi_0)_{Y'}^{-1}\varphi_2(|X|,|X|^{-1})_{Y'}\varphi_{|X|}(\overline{\psi})=\psi^{-1}$
by \eqref{crossing5} and \eqref{crossing6}. Now, by the first
equality of Lemma~\ref{lem-braiding}(d),
\begin{align*}
\tau_{Y',X^*}^{-1}=&\bigl(\lev_X \otimes (\varphi_0)_{Y'}^{-1}\varphi_2(|X|,|X|^{-1})_{Y'} \otimes \id_{X^*}\bigr)\circ \\
&\quad \circ (\id_{X^*}\otimes \tau_{\varphi_{|X|^{-1}}(Y'),X} \otimes \id_{X^*})(\id_{X^* \otimes \varphi_{|X|^{-1}}(Y')} \otimes \lcoev_X).
\end{align*}
Therefore we obtain   \eqref{Fother5}. Equality \eqref{Fother6}
is proved similarly. Since   $\sigma''_+(X,Y,Y',\psi)$ is color-equivalent to
$$
\psfrag{X}[Bl][Bl]{\scalebox{.9}{$X$}} \psfrag{Y}[Bl][Bl]{\scalebox{.9}{$Y$}}
\psfrag{R}[Br][Br]{\scalebox{.9}{$Y'$}} \psfrag{Z}[Br][Br]{\scalebox{.9}{$X$}}
\psfrag{p}[Br][Br]{\scalebox{.9}{$\psi$}} \rsdraw{.45}{.9}{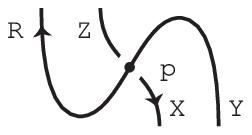}\;,
$$
we obtain from the definition of ${\mathcal F}$ that
\begin{align*}
{\mathcal F}\bigl(&\sigma''_+(X,Y,Y',\psi)\bigr)\\
&=(\id_{Y'^*\otimes X} \otimes \rcoev_Y)(\id_{X^*} \otimes (\id_X \otimes \psi^{-1})\tau_{Y,X} \otimes \id_X) (\lev_X \otimes \psi^{-1} \otimes \id_{X^*}).
\end{align*}
Note that $\phi_Y^{-1} (\psi^-)^*\varphi^1_{|X|}(Y'^*)
\varphi_{|X|}(\phi_{Y'})=\psi^{-1}$ by \eqref{FrFlcomp}, where
$\phi$ is the pivotal structure \eqref{pivotal-struct} of $\cc$.
Now, by the second equality of Lemma~\ref{lem-braiding}(d),
\begin{align*}
\tau_{Y'^*,X}^{-1}&=\bigl(\id_{Y'^* \otimes X} \otimes \lev_{\varphi_{|X|}(Y'^*)}(\varphi_{|X|}^1(Y'^*) \otimes \id_{\varphi_{|X|}(Y'^*)} )\bigr) \circ \\
& \qquad\circ (\id_{Y'^*} \otimes \tau_{Y'^{**},X} \otimes \id_{\varphi_{|X|}(Y'^*)})(\lcoev_{Y'^*} \otimes \id_{X \otimes\varphi_{|X|}(Y'^*)})\\
&=\bigl(\id_{Y'^* \otimes X} \otimes \lev_{\varphi_{|X|}(Y'^*)}(\varphi_{|X|}^1(Y'^*)\varphi_{|X|}(\phi_{Y'}) \otimes \id_{\varphi_{|X|}(Y'^*)} )\bigr) \circ \\
& \qquad\circ (\id_{Y'^*} \otimes \tau_{Y',X} \otimes \id_{\varphi_{|X|}(Y'^*)})(\rcoev_{Y'} \otimes \id_{X \otimes\varphi_{|X|}(Y'^*)}).
\end{align*}
Therefore we obtain  \eqref{Fother7}. Equality \eqref{Fother8}
is proved similarly.

Since $\sigma'''_+(X,Y,X',\psi)$, viewed as a morphism in
$\dd_\cc$,  is   dual to  $\sigma_+(X',Y,X,\psi)$,
$$
{\mathcal F}\bigl(\sigma'''_+(X,Y,X',\psi)\bigr)=(\tau_{X',Y})^*\bigl( (\psi^{-1})^* \otimes \id_{Y^*} \bigr).
$$
Using \eqref{FrFlcomp} and   two expressions of
$(\tau_{X'^*,Y})^{-1}$ given by Lemma~\ref{lem-braiding}(d), we
obtain
$$
(\tau_{X',Y})^*=\bigl( \id_{Y^*} \otimes  (\varphi_0)_{X'^*}^{-1}\varphi_2(|Y|^{-1},|Y|)_{X'^*} \bigr)\tau_{\varphi_{|Y|}(X'^*),Y^*} \bigl( \varphi^1_{|Y|}(X')^{-1} \otimes \id_{Y^*} \bigr).
$$
Therefore
\begin{align*}
{\mathcal F}\bigl(&\sigma'''_+(X,Y,X',\psi)\bigr)\\
&=\bigl( \id_{Y^*} \otimes  (\varphi_0)_{X'^*}^{-1}\varphi_2(|Y|^{-1},|Y|)_{X'^*} \bigr)\tau_{\varphi_{|Y|}(X'^*),Y^*} \bigl( \varphi^1_{|Y|}(X')^{-1}(\psi^{-1})^* \otimes \id_{Y^*} \bigr)\\
&=\bigl( \id_{Y^*} \otimes  (\varphi_0)_{X'^*}^{-1}\varphi_2(|Y|^{-1},|Y|)_{X'^*} \bigr)\tau_{\varphi_{|Y|}(X'^*),Y^*} ( \psi^- \otimes \id_{Y^*} )\\
&=\bigl( \id_{Y^*} \otimes  (\varphi_0)_{X'^*}^{-1}\varphi_2(|Y|^{-1},|Y|)_{X'^*} \varphi_{|Y|^{-1}}(\psi^-)\bigr)\tau_{X^*,Y^*}\\
&= \bigl( \id_{Y^*} \otimes  (\overline{\psi^-})^{-1} \bigr)\tau_{X^*,Y^*}.
\end{align*}
This gives~\eqref{Fother9}. Equality \eqref{Fother10} is proved
similarly.
\end{proof}

\begin{proof}[Proof of Theorem~\ref{thm-functorF+}]
We must prove that if two colored diagrams $D_1, D_2$ are
related by a (colored)   Reidemeister move or a
stabilization move, then ${\mathcal F}(D_1)={\mathcal F}(D_2)$. Invariance under
stabilization  is obvious. Invariance under the Reidemeister
  moves of types 1 and 2 is a direct consequence of
Lemma~\ref{lem-F-othergen}.  For example, the colored diagram $
T_-(X',X,\psi)\, T_+(X,X',\psi)$ on the right-hand side of
Figure~\ref{fig-T1moves}(a)  is carried by ${\mathcal F}$ to
$\theta_{X}^{-1}\psi\psi^{-1}\theta_X=\id_X$; hence the
invariance.

Let us prove the invariance of ${\mathcal F}$ under  the Reidemeister moves
of type 3. The left-hand  side of Figure~\ref{fig-T3move} is
carried by ${\mathcal F}$ to
$$
\bigl((\id_Z \otimes C^{-1})\tau_{Y,Z} \otimes
\id_Z\bigr)\bigl(\id_Y  \otimes (\id_Z \otimes
B^{-1})\tau_{X',Z}\bigr) \bigl((\id_Y \otimes A^{-1})\tau_{X,Y}
\otimes \id_Z\bigr)=$$ $$ \bigl(\id_Z \otimes C^{-1} \otimes
  B^{-1}\varphi_{|Z|} (A^{-1})\bigr)\bigl(\tau_{Y,Z} \otimes
  \varphi_{|Z|}\varphi_{|Y|}(X)\bigr)  \bigl(\id_Y \otimes
  \tau_{\varphi_{|Y|}(X),Z}\bigr)\bigl(\tau_{X,Y} \otimes
  \id_Z\bigr).$$
The right-hand  side of Figure~\ref{fig-T3move} is carried
by ${\mathcal F}$ to
$$
\bigl(\id_Z \otimes (\id_{Y'} \otimes
B'^{-1})\tau_{\widetilde{X},Y'}\bigr)\bigl((\id_Z \otimes A'^{-1})
\tau_{X,Z} \otimes \id_{Y'}\bigr)\bigl(\id_Y \otimes (\id_X \otimes
C^{-1})\tau_{Y,Z}\bigr)=$$ $$\bigl(\id_Z \otimes C^{-1} \otimes
B'^{-1}\varphi_{|Y'|}(A'^{-1} )\bigr) \bigl(\id_Z \otimes
\tau_{\varphi_{|Z|}(X),\varphi_{|Z|}(Y)}\bigr) \bigl(\tau_{X,Z}
\otimes \id_{\varphi_{|Z|}(Y)} \bigr) \bigl(\id_X \otimes
\tau_{Y,Z}\bigr).$$ We conclude using \eqref{move3condition} and the
quantum Yang-Baxter equality of Lemma~\ref{lem-braiding}(c).

Consider    the  first Reidemeister   move  of type 4 shown
in   Figure~\ref{fig-T4moves1}. Using \eqref{Fother4} and
\eqref{Fother7}, we obtain that ${\mathcal F}$ carries the  colored diagram
on the left to
$$
f=\bigl(\id_{\otimes_{j=1}^{n-1} (X^j)^{\varepsilon^j}} \otimes \tau^{-1}_{(X^n)^{\varepsilon^n},Z}  \bigr)   \cdots
\bigl(\tau^{-1}_{(X^1)^{\varepsilon^1},Z} \otimes \id_{\otimes_{j=2}^{n} (Y^j)^{\varepsilon^j}}  \bigr)\bigl(\id_Z \otimes (\otimes_{i=1}^{n} (\psi^i)^{\varepsilon^i})v_0\bigr).
$$
Now, setting $\mu=|Z|$ and   using \eqref{eq-braiding2}, we
obtain that
\begin{align*}
\bigl(\id_{\otimes_{j=1}^{n-1} (X^j)^{\varepsilon^j}} \otimes & \tau^{-1}_{(X^n)^{\varepsilon^n},Z}  \bigr) \circ \cdots \circ
\bigl(\tau^{-1}_{(X^1)^{\varepsilon^1},Z} \otimes \id_{\otimes_{j=2}^{n} (Y^j)^{\varepsilon^j}}  \bigr)\\
&=\tau^{-1}_{\otimes_{j=1}^{n} (X^j)^{\varepsilon^j},Z}\bigl(\id_Z \otimes (\varphi_{\mu})_n((X^1)^{\varepsilon^1}, \cdots, (X^n)^{\varepsilon^n})\bigr).
\end{align*}
Therefore, using \eqref{fourthRmove}, we obtain that
\begin{align*}
f&=\tau^{-1}_{\otimes_{j=1}^{n} (X^j)^{\varepsilon^j},Z}\bigl(\id_Z \otimes (\varphi_{\mu})_n((X^1)^{\varepsilon^1}, \cdots, (X^n)^{\varepsilon^n})(\otimes_{j=1}^{n} (\psi^j)^{\varepsilon^j})v_0\bigr)\\
&=\tau^{-1}_{\otimes_{j=1}^{n} (X^j)^{\varepsilon^j},Z}\bigl(\id_Z \otimes \varphi_{\mu}(v_1) (\varphi_{\mu})_m(X_1^{\varepsilon_1}, \cdots, X_m^{\varepsilon_m})(\otimes_{i=1}^{m} \psi_i^{\varepsilon_i})\bigr)\\
&=\bigl(v_1 \otimes \id_Z\bigr)\tau^{-1}_{\otimes_{i=1}^{m} Y_i^{\varepsilon_i},Z}\bigl(\id_Z \otimes (\varphi_{\mu})_m(X_1^{\varepsilon_1}, \cdots, X_m^{\varepsilon_m})(\otimes_{i=1}^{m} \psi_i^{\varepsilon_i})\bigr)\\
&=\bigl(v_1 \otimes \id_Z\bigr)
\bigl(\id_{\otimes_{i=1}^{m-1} Y_i^{\varepsilon_i}} \otimes \tau^{-1}_{X_m^{\varepsilon_m},Z}  \bigr)   \cdots
\bigl(\tau^{-1}_{Y_1^{\varepsilon_1},Z} \otimes \id_{\otimes_{i=2}^{m} X_i^{\varepsilon_i}}  \bigr)
\bigl(\id_Z \otimes (\otimes_{i=1}^{m} \psi_i^{\varepsilon_i})\bigr).
\end{align*}
The latter morphism is the image under ${\mathcal F}$ of the colored
diagram obtained by the move. This proves the invariance of ${\mathcal F}$
under the first  move of type 4. The   second    move of type~4
is treated similarly  using \eqref{Fother5} and
\eqref{Fother10}.

Consider now the third  move of type~4 shown in
Figure~\ref{fig-T4moves2}. Using  \eqref{Fother3} and
\eqref{Fother6},   we obtain that ${\mathcal F}$ carries the diagram on the
left to
\begin{align*}
g&= \bigl(\id_{\otimes_{j=1}^{n-1} (Y^j)^{\varepsilon^j}}  \otimes (\id_{(Y^n)^{\varepsilon^n}} \otimes (\xi^n)^{-1})\tau_{X^{n-1},(Y^n)^{\varepsilon^n}}\bigr)\circ \cdots \\
& \qquad \cdots \circ \bigl((\id_{(Y^1)^{\varepsilon^1}} \otimes (\xi^1)^{-1})\tau_{X^0,(Y^1)^{\varepsilon^1}} \otimes \id_{\otimes_{j=2}^{n} (Y^j)^{\varepsilon^j}}
  \bigr)\bigl(\id_{X^0} \otimes v\bigr).
\end{align*}
where $\xi^i=\psi^i$ if   $\varepsilon^i=+$   and
 $\xi^i=\overline{\psi^i}$ otherwise.  Set $y_i=\vert
 Y_i\vert^{\varepsilon_i} \in G$ and $y^j= \vert
 Y^j\vert^{\varepsilon^j} \in G$ for all $i,j$
and
 $$\rho=\bigl(\varphi_{y^n} \cdots
 \varphi_{y^2}(\xi^{1})\bigr)\circ\cdots\circ\bigl(\varphi_{y^n}(\xi^{n-1})
 \bigr) \circ\xi^n.$$ Then using \eqref{eq-braiding1}, we
 obtain
\begin{align*}
g&= \bigl(\id_{\otimes_{j=1}^{n} (Y^j)^{\varepsilon^j}} \otimes \rho^{-1}) \bigl(\id_{\otimes_{j=1}^{n-1} (Y^j)^{\varepsilon^j}}
 \otimes \tau_{\varphi_{y^{n-1}} \cdots \varphi_{y^1}(X^0),(Y^n)^{\varepsilon^n}}\bigr)\circ \cdots \\
& \qquad \qquad\qquad \cdots \circ \bigl(\tau_{X^0,(Y^1)^{\varepsilon^1}} \otimes \id_{\otimes_{j=2}^{n} (Y^j)^{\varepsilon^j}}
  \bigr)\bigl(\id_{X^0} \otimes v\bigr)\\
&= \bigl(\id_{\otimes_{j=1}^{n} (Y^j)^{\varepsilon^j}} \otimes \rho^{-1} \varphi^{-1}_n\bigr) \tau_{X^0,\otimes_{j=1}^n(Y^j)^{\varepsilon^j}}\bigl(\id_{X^0} \otimes v\bigr)\\
&=\bigl(v \otimes \rho^{-1} \varphi^{-1}_n\bigr) \tau_{X^0,\otimes_{i=1}^m Y_i^{\varepsilon_i}}.
\end{align*}
Set
$$\varrho=\bigl(\varphi_{y_m} \cdots \varphi_{y_2}(\xi_{1})\bigr)\circ\cdots\circ\bigl(\varphi_{y_m}(\xi_{m-1}) \bigr) \circ\xi_m$$
where $\xi_i=\psi_i$ if   $\varepsilon_i=+$    and
 $\xi_i=\overline{\psi_i}$ otherwise. Using the hypothesis
 $\varphi_n\rho=\varphi_m\varrho$ and \eqref{eq-braiding1}, we
obtain
\begin{align*}
g&=\bigl(v \otimes \varrho^{-1} \varphi^{-1}_m\bigr) \tau_{X^0,\otimes_{i=1}^m Y_i^{\varepsilon_i}}\\
&= \bigl(v \otimes \varrho^{-1}\bigr) \bigl(\id_{\otimes_{i=1}^{m-1} Y_i^{\varepsilon_i}}  \otimes
 \tau_{\varphi_{y_{m-1} } \cdots \varphi_{y_1 }(X_0) ,(Y_m)^{\varepsilon_m}}\bigr)  \cdots \bigl(\tau_{X_0,Y_1^{\varepsilon^1}} \otimes \id_{\otimes_{i=2}^{m} Y_i^{\varepsilon_i}}
  \bigr) \\
&= \bigl(v \otimes \id_{X_m}\bigr)\bigl(\id_{\otimes_{i=1}^{m-1} Y_i^{\varepsilon_i}}  \otimes (\id_{Y_m^{\varepsilon_m}} \otimes \xi_m^{-1})\tau_{X_{m-1},Y_m^{\varepsilon_m}}\bigr)\circ \cdots \\
& \qquad \qquad \cdots \circ \bigl((\id_{Y_1^{\varepsilon_1}} \otimes \xi_1^{-1})\tau_{X_0,Y_1^{\varepsilon_1}} \otimes \id_{\otimes_{i=2}^{m} Y_i^{\varepsilon_i}}
  \bigr).
\end{align*}
The latter morphism is the image under ${\mathcal F}$ of the diagram
obtained  by the move. This proves the invariance of ${\mathcal F}$ under
the third  move of type 4. The  fourth    move of type~4 is
treated similarly  using \eqref{Fother8} and \eqref{Fother9}.
 \end{proof}

 \section{The functor $F_\cc $}\label{From colored  diagrams to colored graphs}

 In this section, we  construct a canonical monoidal functor $ {\mathcal G}_\cc\to
 \cc$. We begin by discussing relations between colored ribbon
 graphs and colored   diagrams.

\subsection{Presentation of graphs by diagrams}\label{Colorings of diagrams vs. colorings of graphs}
Any graph diagram $D$ represents a  ribbon graph $
\Omega_D $ in the obvious way. Namely, we identify $\RR \times [0,1] $ with $ \RR\times \{0\}\times [0,1] \subset
\RR^2\times [0,1]$ and slightly push the interiors
of the underpasses of $D$ along the second axis into $\RR \times [0, 1) \times [0,1]$ keeping the rest of $D$.
This transforms the segments,
  circles, and coupons of $D$ into the edges,   circle
components,  and coupons of $ \Omega_D$, respectively.   The
framing of $\Omega_D$ is given by the constant vector field $(0,
\delta, 0)$ with small $\delta>0$.

Each underpass $p$ of a segment $d$ of $D$ determines a {\it
diagrammatic track} $\gamma_p$ of the edge of $\Omega=\Omega_D$
represented by $d$. The track $\gamma_p$ is represented by the
linear   path from the base point   $z\in C_\Omega$ to the point
of  $\widetilde d \subset \widetilde \Omega$ obtained from an
interior point of $p$ by shifting along the framing vector.
 By Section \ref{ribbredGgraphs}, the track $\gamma_p$
determines a  (negative) meridian $\mu_{\gamma_p} \in
\pi_1(C_\Omega)$ which we  call the {\it diagrammatic meridian}
of $p$  and denote by $\mu_p$. If the underpass $p$ is adjacent
to an input/output of $D$, then $\gamma_p$ is the corresponding
input/output  track of $\Omega$. In a similar way,  a coupon $Q$
of $D$  determines the  {\it diagrammatic track} $\gamma_Q$  of
the corresponding coupon of $\Omega $
and the associated {\it diagrammatic meridian}
$\mu_Q=\mu_{\gamma_Q}\in \pi_1(C_\Omega)$.

If $D$ has no circle 1-strata, then there is a direct
relationship between the colorings of $D$ and
$\Omega=\Omega_D$. Pick a homomorphism $g\colon \pi_1(C_\Omega)
\to G$. Each pre-coloring $u$ of    $ (\Omega, g) $ induces a
pre-coloring $U=U(u)$  of $D$ as follows: for every underpass
$p$ of $D$, set $U_p=u_{\gamma_p}\in \mathcal C_{g(\mu_p)}$  and
for every crossing $c$ of $D$, set
$$U_c=u_{\mu_{ {\underline c}}^{-1}, \gamma_{{c^{-}}}}\co  U_{  {{c^{+}}}}
=u_{  \gamma_{{c^{+}}}}=u_{\mu_{ {\underline c}}^{-1}  \gamma_{{c^{-}}}}\to
\varphi_{g(\mu_{ {\underline c}})} (u_{\gamma_{{c^{-}}}})=
\varphi_{ \vert U_{ {\underline c}} \vert  } (U_{  {{c^{-}}}})  .$$
Here we use the obvious equality $     \gamma_{{c^{+}}} =\mu_{
{\underline c}}^{-1}  \gamma_{{c^{-}}} $.
Similarly,  a coloring $(u,v)$ of   $(\Omega, g)$ induces a
coloring $( U=U(u),V)$ of $ D $ by $V_Q=v_{\gamma_Q}$ for
any coupon~$Q$ of~$D$.  We say that the colored diagram $(D, U ,V)$ {\it represents} the colored $G$-graph
$(\Omega, g, u, v)$.

To sum up, the structure $(\Omega, g)$ of a $G$-graph on $\Omega=\Omega_D$ together with a
coloring $(u,v)$ of this $G$-graph induce  a
coloring $( U ,V)$ of $ D $.       The homomorphism $g\colon \pi_1(C_\Omega)
\to G$ can  be  recovered
from   $U$  by
$g(\mu_p)= \vert U_p\vert$ for any underpass $p$ of $D$. Though we shall not need it, note
that the coloring
$(u,v)$ can be recovered from $(U,V)$   uniquely up to
isomorphism. Thus,  the colored $G$-graph $(\Omega, g, u,v)$ can be reconstructed
from the colored diagram $(D,U,V)$
uniquely up to isomorphism. Generally speaking, there are colorings of $D$ that
do not arise in this way from colorings of $\Omega$.  We emphasize
that these constructions apply only to diagrams and ribbon graphs without  circle components.

\begin{lem}\label{lem-functorF++}   Let $\Omega_r$ be  a colored  $G$-graph having no circle components and represented by a colored    diagram $D_r$ for $r=1, 2$.
\begin{enumerate}
  \labeli
\item  If $\Omega_1$, $\Omega_2$ are color-equivalent, then there
    is a finite sequence of colored Reidemeister moves,
    isotopies, and isomorphisms of colorings which transforms
    $D_1$ into  $D_2$.
\item If $\Omega_1$, $\Omega_2$ are stably
    color-equivalent, then there is a finite sequence of colored
    Reidemeister moves, isotopies,  isomorphisms of colorings,
      stabilizations, and   moves inverse to stabilizations
      which transforms $D_1$ into $D_2$.
\end{enumerate}
\end{lem}

\begin{proof} Claim  (ii) directly follows from Claim (i) and we focus on the latter.
  Recall that the color-equivalence of colored  $G$-graphs is
generated by isomorphisms of colorings and isotopies. It is clear
that isomorphisms of colorings of  graphs induce isomorphisms of the
induced colorings of  diagrams.  It remains   to handle isotopies of
graphs.

The existence of a color-preserving isotopy of $\Omega_1$ into
$\Omega_2$ implies the existence of a color-preserving isotopy of
$\Omega_1$ into $\Omega_2$ which keeps all the coupons parallel to
the strip $ \RR\times \{0\}\times [0,1]$ (this follows from
the surjectivity of the inclusion homomorphism $\pi_1(SO(2))\to
\pi_1(SO(3))$). Projecting such an isotopy into $\RR\times
[0,1]$, we obtain a finite sequence of colored Reidemeister moves
and  isotopies   transforming $D_1$ into $D_2$. Note that
the type 3 moves determined by various orientations of the branches
may be expanded as compositions of the type 3 move  of
Figure~\ref{fig-T3move} and the type 2 moves, see for instance
\cite{Po}. Therefore  it is enough to consider only the type 3 move
shown  in Figure~\ref{fig-T3move}. We need to prove that the
colorings of the diagrams are transformed as in the definition of
the colored Reidemeister moves. This is a consequence of
the following statement.

Claim.   {\it  A   Reidemeister move  $D\mapsto {\widetilde D}$ on
(uncolored) graph diagrams  without circle 1-strata   determines a
self-homeomorphism $f$ of $\RR^2\times [0,1]$ carrying
$\Omega=\Omega_D$ to $ {\widetilde \Omega} =\Omega_{{\widetilde
D}}$. Given a homomorphism $g\co \pi_1(C_\Omega)\to G$ and a coloring
$(u,v)$ of   $(\Omega,g) $, we transfer this data along $f$ to
obtain a homomorphism $ {\widetilde g}\colon \pi_1(C_{{\widetilde
\Omega}})\to G$ and a coloring $({\widetilde u},{\widetilde v})$ of
$({\widetilde \Omega}, {\widetilde g})$. Then the diagrams $D$ and $
{\widetilde D}$ with the   colorings $(U,V)$ and  $({\widetilde U},
{\widetilde V})$ induced from $(u,v)$ and  $({\widetilde
u},{\widetilde v})$ respectively, are related by the  corresponding
colored  Reidemeister move. }

  In the proof   we will   use the action of $f^{-1}$ on the
  tracks: for a  track $  \gamma$ of an edge/coupon of  $
  {\widetilde \Omega}$,  its pre-image $f^{-1} (  \gamma)$ is a
  track of  the corresponding edge/coupon of  $\Omega$. The
  isomorphism $ \pi_1(C_{{\widetilde \Omega}}) \to  \pi_1(C_{{
  \Omega}})$ induced by $f^{-1}$ will be also denoted  by
  $f^{-1}$. In this notation $\widetilde g = g f^{-1}$ and  $
  \widetilde u_{ \gamma}=u_{f^{-1} (  \gamma)}$  for any
  edge-track $  \gamma$ of  $ {\widetilde \Omega}$.

Consider the first type 1 move $D\mapsto {\widetilde D}$  in
Figure~\ref{fig-T1moves}. Let $p$ be the  underpass of $D$
modified by the move. The corresponding piece of ${\widetilde
D}$ contains two new  crossings  $c$, $ e$ and splits into 3
underpasses $c^{-}$, $e^{-}$, and $\underline c =\underline
e=c^+=e^+$.  Since $f^{-1} (\gamma_{{c^{-}}})= \gamma_p$,
$${\widetilde U}_{c^{-}}={\widetilde
u}_{\gamma_{c^{-}}}=u_{f^{-1}
(\gamma_{c^{-}})}=u_{\gamma_p}=U_p.$$ Similarly, $ f^{-1}
(\gamma_{{e^{-}}})=\gamma_p$ and ${\widetilde U}_{e^{-}}  =U_p$.
Also,
$${\widetilde U}_c={\widetilde u}_{\mu_{{\underline c}}^{-1}, \gamma_{{c^{-}}}}=u_{f^{-1} (\mu_{  {\underline c}}^{-1}), f^{-1} (\gamma_{{c^{-}}})} =u_{f^{-1} (\mu_{  {\underline e}}^{-1}), f^{-1} (\gamma_{{e^{-}}})}
={\widetilde u}_{\mu_{   {\underline e}}^{-1}, \gamma_{{e^{-}}}}=
 {\widetilde U}_e.$$
  Thus, the move $(D,U,V)\mapsto ({\widetilde D}, {\widetilde
U}, {\widetilde V})$ is as shown in
Figure~\ref{fig-T1moves}, where $X= U_p$, $X'={\widetilde
U}_{\underline c}={\widetilde U}_{\underline e}$, and $\psi=
{\widetilde U}_c={\widetilde U}_e$. The other   type  1 moves
are treated similarly.

A type 2 move $D\mapsto {\widetilde D}$    creates two new crossings
$c$, $ e$ such that  $\underline c=\underline e$. The ${\widetilde
U}$-colors in the   top and bottom   of ${\widetilde D}$ coincide
because ${\widetilde U}=U \circ f^{-1}$, $f=\id$ near the
top/bottom, and the $U$-colors  in the top and bottom  of   $D$
coincide. The equality ${\widetilde U}_c=  {\widetilde U}_e$ follows
from  the   formulas    $\mu_{ {\underline c}}=\mu_{  {\underline
e}}$  and $f^{-1} (\gamma_{{c^{-}}}) =f^{-1} (\gamma_{{e^{-}}})$.
The latter holds because  either $c^{-}=e^{-}$ or  both tracks
$f^{-1} (\gamma_{{c^{-}}}) $ and $f^{-1} (\gamma_{{e^{-}}})$ are
equal to the  diagrammatic track  of  one and   the same  underpass
of $D$.

Consider a type 3 move $D\mapsto {\widetilde D}$. Denote by $ c$
the crossing of $D$ colored with     $ C$  and  by $ c'$ the
crossing  of ${\widetilde D}$ colored with     $  C'$. It is
clear that $f$ carries $\gamma_{c^-}$  and
$\mu_{\gamma_{\underline c}}$ to $\gamma_{(c')^-}$ and
$\mu_{\gamma_{\underline c'}}$, respectively. Therefore
$$C'={\widetilde U}_{c'}={\widetilde u}_{\mu_{\gamma_{\underline
c'}}^{-1}, \gamma_{(c')^-}} =u_{f^{-1} (\mu_{\gamma_{\underline
c'}}^{-1}), f^{-1} (\gamma_{(c')^-})}=
u_{\mu_{\gamma_{\underline c}}^{-1}, \gamma_{c^-}}=U_c=C.$$ Let
$p,   q, r$  be the underpasses of $D$ colored with $X,  Y, Z$,
respectively. Thus, $X=u_{\gamma_p}$, $Y=u_{\gamma_q}$, and
$Z=u_{\gamma_r}$.  Set $\gamma=\gamma_p$,
$\delta=\mu_{\gamma_q}^{-1} $, and $\beta= \mu_{\gamma_r}^{-1}
$. Then
 $$\varphi_2(\vert Z\vert, \vert Y\vert)_X \varphi_{\vert Z\vert} (A) B=
\varphi_2(g(\beta^{-1}) , g(\delta^{-1}))_{u_\gamma} \varphi_{g(\beta^{-1})} (u_{\delta, \gamma}) u_{\beta, \delta \gamma}=u_{\beta  \delta, \gamma}
$$
where the last equality follows from the commutativity of the
diagram   \eqref{condweak} and we use  the formulas $\vert
Z\vert=g(\beta^{-1})$, $\vert Y\vert=g(\delta^{-1})$. Let $p',
q', r'$  be the underpasses of ${\widetilde D}$ colored with $X,
Y', Z$, respectively. A similar computation gives
$$\varphi_2(\vert Y'\vert, \vert Z\vert)_X \varphi_{\vert Y'\vert} (A') B'=
 {\widetilde u}_{\beta'  \delta', \gamma'} ,
$$
where $\beta' = \mu_{\gamma_{q'}}^{-1}$, $\delta' =
\mu_{\gamma_{r'}}^{-1}$, and $\gamma'=\gamma_{p'}$.  Clearly,
$f^{-1}( \beta'  \delta')= \beta  \delta$ and
$f^{-1}(\gamma')=\gamma$. Hence ${\widetilde u}_{\beta'
\delta', \gamma'} = u_{\beta  \delta, \gamma} $ which proves the
commutativity of the diagram \eqref{move3condition}.

Consider now a type 4 move   $D\mapsto {\widetilde D}$. Denote
by $Q$ the coupon of $D$ subject to the  move  and by
$\widetilde Q$ the corresponding coupon of $\widetilde D$. We
begin with the first type 4 move. We must prove that the
morphisms $v_0, v_1, \psi_i, \psi^j$ determined by the colorings
$(u,v)$ and $(\widetilde u, \widetilde v)$ turn
\eqref{fourthRmove} into a commutative diagram. Let $\beta\in
\pi_1(C_{\widetilde \Omega})$ be the inverse of the diagrammatic
(negative) meridian of the $Z$-colored underpass of $\widetilde
D$. Clearly, $ {\widetilde g} (\beta^{-1})=\vert  Z\vert =\mu
\in G$. Let $\gamma$ be the diagrammatic track of   $\widetilde
Q$.    We shall identify the diagram \eqref{fourthRmove} with
the  diagram \eqref{weakiso---} associated with these   $\beta$,
$\gamma$ and the coloring $(\widetilde u, \widetilde v)$ of
$\widetilde D$.    For   $i=1, \ldots, m$, the track $\gamma_i$
derived from $\gamma$ as in Section \ref{coloredGgraphs}   is
the diagrammatic track of the $  {Y_i}$-colored underpass  of
${\widetilde D}$ and  $\beta \gamma_i$ is the diagrammatic track
of the $ {X_i}$-colored underpass  of ${\widetilde D}$.
Therefore $\widetilde u_{\gamma_i}= {Y_i}$ and  $\widetilde
u_{\beta \gamma_i}=   {X_i}$ for all $i$. For   $j=1, \ldots,
n$, the track $\gamma^j$  derived from $\gamma$ as in Section
\ref{coloredGgraphs}   is the diagrammatic track of the $
{X^j}$-colored underpass of ${\widetilde D}$ and so  $\widetilde
u_{\gamma^j}=   {X^j}$ for  all $j $. To compute $\widetilde
u_{\beta\gamma^j} =u_{f^{-1} (\beta \gamma^j)}$, note that
$f^{-1} (\beta \gamma^j)$ is the diagrammatic track of the $
{Y^j}$-colored underpass of $D$. Hence  $\widetilde
u_{\beta\gamma^j}=  {Y^j}$ for all $j$. Thus,  the   diagrams
\eqref{weakiso---}  and \eqref{fourthRmove}  have the same
objects. It follows directly from the definitions  that the
morphisms are also the same. Now,  the commutativity of
\eqref{weakiso---}  implies the commutativity of
\eqref{fourthRmove}. The second   type 4 move is treated
similarly using in the role of   $\beta\in \pi_1(C_{\widetilde
\Omega})$   the diagrammatic  meridian of the $Z$-colored
underpass (rather than its inverse as above).    One should also
use the identity $\overline{u_{\delta^{-1}, \delta \gamma}}=
u_{\delta, \gamma}$ which  holds for any $\delta\in
\pi_1(C_\Omega)$ and any edge-track $\gamma$  of $\Omega$. This
identity is a  consequence of the definition of  a coloring of
$\Omega$.

Consider the third type 4 move $D\mapsto {\widetilde D}$.  We
need to prove the equality of the   compositions \eqref{comp1}
and \eqref{comp2}.  Let    $\gamma^j$ be  the diagrammatic track
of the $X^j$-colored underpass of $D$ for $j= 0,1, \ldots, n$.
By the definition of the induced coloring, $X^j=u_{\gamma^j}$
for all $j$. Note that  $\gamma^n= \beta^{-1}   \gamma^0$ where
$\beta\in \pi_1(C_\Omega)$ is the diagrammatic meridian of $Q$.
We claim that the  composition  \eqref{comp1}    is    equal to
$u_{\beta^{-1}, \gamma^0}\colon X^n \to \varphi_{g(\beta)}
(X^0)$.  Suppose first that $\varepsilon^j=+$ for all $j=1,
\ldots, n$.  For $n=1$    our claim follows from the definition
of $\psi^1$. For $n \geq 2$ the claim is deduced by induction
from the commutativity of the diagram \eqref{condweak}. The case
$\varepsilon^j=-$   can be reduced to the case $\varepsilon^j=+$
by inverting the orientation of the $Y^j$-colored underpath and
changing its color  to $(Y^j)^*$. Indeed, under this
transformation  the color of the $j$-th crossing $\psi^j=
u_{y^{-1}, \gamma^j} \co X^{j-1}\to \varphi_{y } (X^j)$ (where
$y=\vert Y^j\vert\in G$)  changes to $u_{y ,
\gamma^{j-1}}\colon X^j \to \varphi_{y^{-1} } (X^{j-1})$ and we
need only to observe that  $$u_{y ,  \gamma^{j-1}}= u_{y, y^{-1}
\gamma^j} = \overline {u_{y^{-1}, \gamma^j}}= \overline {\psi^j}
.$$ A similar computation shows that the  composition
\eqref{comp2}  is    equal to ${\widetilde  u}_{\mu^{-1},
\gamma_0}\colon X_m \to \varphi_{g(\mu)} (X_0)$ where $ \mu \in
\pi_1(C_{\widetilde  \Omega})$ is the diagrammatic meridian of $
\widetilde Q$  and  $\gamma_0$ is  the diagrammatic track of the
$X_0$-colored underpass of $\widetilde  D$. It remains to
observe that $X_m=X^n$, $X_0=X^0$, and
$${\widetilde  u}_{\mu^{-1}, \gamma_0}= { u}_{f^{-1}(\mu^{-1}), f^{-1}(\gamma_0)}
= u_{\beta^{-1}, \gamma^0}.$$
The fourth type 4 move is treated similarly.
\end{proof}

\subsection{Functor  $F_\cc$}\label{The functorR} For  any    $G$-ribbon category
$\cc$, we define a functor $F_{\mathcal C}\co {\mathcal
G}_\cc\to \cc$ as follows. Consider  the category of colored
diagrams ${\mathcal D}={\mathcal D}_{\mathcal C}$ and let $\overline
{\mathcal D} $ be its quotient by the equivalence relation on the
set of morphisms generated by the colored Reidemeister moves and
stabilization. The category $ \overline {\mathcal D} $ has the same
objects as ${\mathcal  D} $, and the structure of a strict monoidal
category in $ {\mathcal D}$ induces a structure of a strict monoidal
category in $ \overline {\mathcal D} $. Theorem~\ref{thm-functorF+}
implies that the strong monoidal functor $\mathcal F  \colon
{\mathcal  D}  \to  {\mathcal C}$ of Lemma~\ref{lem-functorF} is
invariant under the colored Reidemeister moves and stabilization.
Therefore $\mathcal F $ induces  a   strong monoidal functor
$\overline {\mathcal F} \colon \overline {\mathcal D}   \to
{\mathcal C}$.

Consider the  strict
monoidal functor   $P\co  \mathcal G_{ \mathcal
C}\to \overline {\mathcal D}$ that    carries each object to itself
  and carries a morphism
represented by a colored $G$-graph
  into the morphism
  represented by a  diagram of this graph with induced
coloring. Lemma~\ref{lem-functorF++} implies that $P$ is well
defined. Set   $$F_{\mathcal C}=\overline {\mathcal F} P\co \mathcal
G_{ \mathcal C} \to {
 \mathcal C}.$$
  Since $P$ is   strict
monoidal  and $\overline {\mathcal F}$ is strong monoidal, their
composition   $F_{\mathcal C}$ is   strong monoidal.   We summarize the relationships between these functors in the following commutative diagram:
$$
\xymatrix@R=1cm @C=1.5cm {
 & \dd_\cc \ar@{->>}[d] \ar[rd]^{{\mathcal F}_\cc}& \\
{\mathcal G}_\cc \ar@/_1.5pc/[rr]_{F_\cc} \ar[r]^{P}& \overline{\dd}_\cc \ar[r]^{\overline{\mathcal F}_\cc} & \cc
}
$$

\subsection{Remarks} 1.  The category ${\mathcal  D}_\cc$  becomes a pivotal $G$-graded category  by
setting $$ \vert U\vert =\prod_{r=1}^k \vert U_r\vert^{\varepsilon_r} \quad {\text {and}} \quad U^*=((U_k,-\varepsilon_k), \dots
,(U_1,-\varepsilon_1)) $$ for any object $U=((U_1,\varepsilon_1),
\dots,(U_k,\varepsilon_k))$ of $\dd$,  and
\begin{center}
\psfrag{A}[Bc][Bc]{\scalebox{.75}{$(U_1,\varepsilon_1)$}}
\psfrag{B}[Bc][Bc]{\scalebox{.75}{$(U_k,\varepsilon_k)$}}
\psfrag{U}[Bc][Bc]{\scalebox{.75}{$(U_k,-\varepsilon_k)$}}
\psfrag{V}[Bc][Bc]{\scalebox{.75}{$(U_1,-\varepsilon_1)$}}
$\lev_{U}=$\rsdraw{.45}{.9}{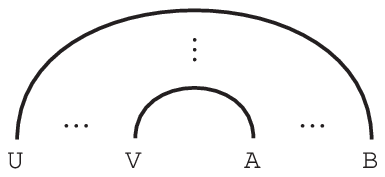}, \qquad $\lcoev_{U}=$\rsdraw{.45}{.9}{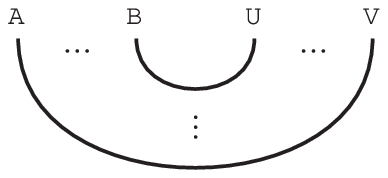},\\[.8em]
\psfrag{U}[Bc][Bc]{\scalebox{.75}{$(U_1,\varepsilon_1)$}}
\psfrag{V}[Bc][Bc]{\scalebox{.75}{$(U_k,\varepsilon_k)$}}
\psfrag{A}[Bc][Bc]{\scalebox{.75}{$(U_k,-\varepsilon_k)$}}
\psfrag{B}[Bc][Bc]{\scalebox{.75}{$(U_1,-\varepsilon_1)$}}
$\rev_{U}=$\rsdraw{.45}{.9}{ev-diag}, \qquad
$\rcoev_{U}=$\rsdraw{.45}{.9}{coev-diag},
\end{center}
where the orientation of the arcs is uniquely determined by the
 signs $\varepsilon_i$.

 It is easy to check  that the   functor $\mathcal F=\mathcal F_{\mathcal C} \colon {\mathcal  D}_{\mathcal C} \to  {\mathcal C}$ is grading preserving and pivotal. We have
$$
{\mathcal F}^1((U_1,\varepsilon_1), \dots,(U_n,\varepsilon_n))={\mathcal F}^1(U_n,\varepsilon_n) \otimes \cdots \otimes {\mathcal F}^1(U_1,\varepsilon_1)
$$
where
$$
{\mathcal F}^1(U_i,\varepsilon_i)= \left\{\begin{array}{ll} \id_{U_i^*}\co  {U_i^*}\to {U_i^*} & \text{if   $\varepsilon_i=+$, }\\
 (\rev_{U_i} \otimes \id_{U_i^{**}})(\id_{U_i} \otimes \lcoev_{U_i^*})\co {U_i}\to U_i^{\ast \ast} & \text{if   $\varepsilon_i=-$. }\end{array}\right.
$$

Similarly, using   colored $G$-graphs represented by the diagrams
above, one can turn $\mathcal G_{ \mathcal C}$ into a pivotal
$G$-graded category, and then the functor $F_\cc\co {\mathcal
G}_\cc\to \cc$  is grading preserving and pivotal.

2. Using appropriate colored braids, one can define a $G$-braiding
in $\mathcal G_{ \mathcal C}$ turning   $\mathcal G_{ \mathcal
C}$  into a $G$-ribbon category  so that the functor
$F_\cc$ preserves both the $G$-braiding and the $G$-twist. We
shall not use this $G$-braiding.

\section{Conjugation of colorings}\label{Conjugation}
In this section, we  define conjugation of colorings and describe the behavior of the functor $F_\cc\co {\mathcal G}_\cc\to \cc$ under conjugation.

\subsection {Conjugation of colorings}\label{Conjugation of the colorings}
We can conjugate $G$-graphs and their   colorings    by any
$\eta\in G$. For a  $G$-graph $(\Omega, g)$, its {\it $\eta$-conjugate} $\Omega^{\eta}=(\Omega,
g^{\eta})$ is the same ribbon graph $\Omega$ endowed with the
homomorphism
 $g^{\eta}=\eta^{-1} g \eta\colon \pi_1(C_\Omega)\to G$. Given
 a pre-coloring $u$ of $(\Omega, g)$, we define the {\it
 $\eta$-conjugate  pre-coloring} $u^{\eta}$ of $(\Omega,
 g^{\eta})$ by
$u^{\eta}_\gamma=\varphi_\eta (u_\gamma)\in \mathcal
 C_{g^{\eta}(\mu_\gamma)}$ for any edge-track $\gamma$   of
 $\Omega$. For $\beta \in \pi_1(C_\Omega)$, we define the
 isomorphism
$u^{\eta}_{\beta, \gamma}\colon u^{\eta}_{\beta  \gamma} \to
\varphi_{g^{\eta}(\beta^{-1})} (u^{\eta}_\gamma)$ as the
composition of the isomorphisms
\begin{gather*}
\xymatrix@R=1cm @C=2cm {
u^{\eta}_{\beta  \gamma}=\varphi_{\eta} (u_{\beta \gamma})  \ar[r]^-{\varphi_{\eta} (u_{\beta, \gamma})}  &
 \varphi_{\eta} \varphi_{g(\beta^{-1})} (u_\gamma)   \ar[r]^-{\varphi_2 ( \eta, g(\beta^{-1}) )_{u_\gamma}}  &
   \varphi_{g(\beta^{-1})\eta} (u_\gamma) }\\
\xymatrix@R=1cm @C=3cm { { =\varphi_{\eta g^{\eta}(\beta^{-1})} (u_\gamma) }\ar[r]^-{(\varphi_2 ( g^{\eta}(\beta^{-1}), \eta)_{u_\gamma})^{-1}}
  & \varphi_{  g^{\eta}(\beta^{-1}) } \varphi_\eta (u_\gamma)=\varphi_{  g^{\eta}(\beta^{-1})  }  (u^{\eta}_\gamma).}
\end{gather*}

\begin{lem}\label{lem-conjugation}
$u^{\eta}$ is a pre-coloring of $\Omega^{\eta}$.
\end{lem}

\begin{proof}
Consider the commutative diagram obtained     by setting $\alpha=\eta$ and $X=u_\gamma$ in  \eqref{crossing6}.
 Since   $(\varphi_0)_{u_\gamma}= u_{1, \gamma}$, we deduce that
$u^{\eta}_{1, \gamma} =(\varphi_0)_{\varphi_\eta (u_{  \gamma})
}=(\varphi_0)_{u^{\eta}_\gamma}$.

We now check the commutativity of the diagram \eqref{condweak}
where $g$ and $u$ are  replaced by $g^{\eta}$ and $u^{\eta}$
respectively. It follows from the definitions that
\begin{gather*}
\varphi_{g^{\eta}(\beta^{-1})}  (u^{\eta}_{\delta, \gamma}) \,  u^{\eta}_{ \beta, \delta \gamma}= \varphi_{g^{\eta}(\beta^{-1})} ((\varphi_2 ( g^{\eta}(\delta^{-1}), \eta)_{u_{ \gamma}})^{-1}) \circ
 \varphi_{g^{\eta}(\beta^{-1})} ( \varphi_2 ( \eta, g(\delta^{-1}) )_{u_{ \gamma}} ) \\
  \circ\,
  \varphi_{g^{\eta}(\beta^{-1})} \varphi_{\eta} (u_{\delta,   \gamma} )\circ
(\varphi_2 ( g^{\eta}(\beta^{-1}), \eta)_{u_{\delta\gamma}})^{-1} \circ  \varphi_2 ( \eta, g(\beta^{-1}) )_{u_{\delta\gamma}} \circ     \varphi_{\eta} (u_{\beta, \delta \gamma}).
\end{gather*}
 We  rewrite the composition of the  three leftmost morphisms in
 the last row using twice the  naturality of $\varphi$. This
 gives
\begin{gather*}
\varphi_{g^{\eta}(\beta^{-1})}  (u^{\eta}_{\delta, \gamma}) \circ  u^{\eta}_{ \beta, \delta \gamma}= \varphi_{g^{\eta}(\beta^{-1})} ((\varphi_2 ( g^{\eta}(\delta^{-1}), \eta)_{u_{ \gamma}})^{-1}) \circ
 \varphi_{g^{\eta}(\beta^{-1})} ( \varphi_2 ( \eta, g(\delta^{-1}) )_{u_{ \gamma}} ) \\
  \circ
  (\varphi_{2} (g^{\eta}(\beta^{-1}), \eta)_{\varphi_{g(\delta^{-1})} (u_\gamma)} )^{-1}   \circ
\varphi_{g(\beta^{-1})  \eta} (u_{\delta,\gamma})  \circ  \varphi_2 ( \eta, g(\beta^{-1}) )_{u_{\delta\gamma}} \circ     \varphi_{\eta} (u_{\beta, \delta \gamma})\\
= \varphi_{g^{\eta}(\beta^{-1})} ((\varphi_2 ( g^{\eta}(\delta^{-1}), \eta)_{u_{ \gamma}})^{-1}) \circ
 \varphi_{g^{\eta}(\beta^{-1})} ( \varphi_2 ( \eta, g(\delta^{-1}) )_{u_{ \gamma}} ) \circ\\
  (\varphi_{2} (g^{\eta}(\beta^{-1}), \eta)_{\varphi_{g(\delta^{-1})} (u_\gamma)} )^{-1}   \circ
\varphi_{2} ( \eta, g(\beta^{-1}))_{\varphi_{g(\delta^{-1})} (u_\gamma)}  \circ  \varphi_\eta \varphi_{g(\beta^{-1})}   (u_{\delta,\gamma}) \circ     \varphi_{\eta} (u_{\beta, \delta \gamma}).
\end{gather*}
Using the commutativity of   \eqref{condweak}  for $g,u$, we
 replace the composition of the  two rightmost morphisms    with
 $\varphi_\eta (\varphi_2 (g(\beta^{-1}),
 g(\delta^{-1}))_{u_\gamma})^{-1}\circ   \varphi_\eta (u_{\beta
 \delta, \gamma})$. At each of the next three steps we use
 \eqref{crossing5}. First, we replace  $$\varphi_{2} ( \eta,
 g(\beta^{-1}))_{\varphi_{g(\delta^{-1})} (u_\gamma)}  \circ
 \varphi_\eta (\varphi_2 (g(\beta^{-1}),
 g(\delta^{-1}))_{u_\gamma})^{-1}$$
with $$(\varphi_{2} (g(\beta^{-1}) \eta, g( \delta^{-1}))_{
u_\gamma} )^{-1} \circ \varphi_{2} ( \eta,
g(\delta^{-1}\beta^{-1}))_{  u_\gamma} .$$ Next, we replace
$$  \varphi_{g^{\eta}(\beta^{-1})} (\varphi_{2} (\eta, g(\beta^{-1}))_{ u_\gamma } ) \circ (\varphi_{2} (g^{\eta}(\beta^{-1}), \eta)_{\varphi_{g(\delta^{-1})} (u_\gamma)} )^{-1}  \circ (\varphi_{2} (g(\beta^{-1}) \eta, g( \delta^{-1}))_{  u_\gamma} )^{-1}$$
with $(\varphi_{2} (g^{\eta}(\beta^{-1}), g(\delta^{-1})\eta)_{
 u_\gamma } )^{-1} $.
Finally, we replace
$$\varphi_{g^{\eta}(\beta^{-1})} ((\varphi_2 ( g^{\eta}(\delta^{-1}), \eta)_{u_{ \gamma}})^{-1}) \circ (\varphi_{2} (g^{\eta}(\beta^{-1}), g(\delta^{-1})\eta)_{ u_\gamma } )^{-1} $$
with
$$(\varphi_{2} (g^{\eta}(  \beta^{-1}), g^{\eta}(  \delta^{-1}))_{ u^{\eta}_\gamma } )^{-1} \circ (\varphi_{2} (g^{\eta}(\delta^{-1} \beta^{-1}),  \eta)_{ u_\gamma } )^{-1}.$$
The resulting expression is nothing but $(\varphi_{2} (g^{\eta}(
 \beta^{-1}), g^{\eta}(  \delta^{-1}))_{ u^{\eta}_\gamma }
 )^{-1} \, u^{\eta}_{\beta \delta, \gamma}$.
\end{proof}

It is clear that the source and the target of $\Omega^\eta$ can
be computed from the source and the target of $\Omega$ by
applying $\varphi_\eta$ to the objects of $\cc$ while keeping the signs.

   For a coloring $(u,v)$ of $(\Omega, g)$, we define the {\it
   $\eta$-conjugate  coloring} $(u^{\eta},v^{\eta})$ of $(\Omega,
   g^{\eta})$. Here
    $u^{\eta}$ is the pre-coloring defined  above, and for any
    coupon-track $\gamma$ of $\Omega$, the morphism
    $v^{\eta}_\gamma$  is defined as the composition   of
    morphisms
\begin{gather*}
\xymatrix@R=1cm @C=1.8cm {
\otimes_{i=1}^m (\varphi_\eta (u_{\gamma_i}))^{\varepsilon_i}
\ar[r]^-{ \otimes_i \rho_i^{-1}}  &  \otimes_{i=1}^m \,
\varphi_\eta (u_{\gamma_i}^{\varepsilon_i} )
\ar[r]^-{(\varphi_\eta)_m  }  &    \varphi_\eta (\otimes_{i=1}^m
\,u_{\gamma_i}^{\varepsilon_i} ) \ar[r]^-{
\varphi_\eta(v_\gamma)  } & {}}\\
\xymatrix@R=1cm @C=2cm
{\varphi_\eta (\otimes_{j=1}^n \,  u_{\gamma^j}^{\varepsilon^j} )  \ar[r]^-{(\varphi_\eta)_n^{-1}  }
  &   \otimes_{j=1}^n   \,\varphi_\eta (  u_{\gamma^j}^{\varepsilon^j} )  \ar[r]^-{ \otimes_j \rho^j}
   & \otimes_{j=1}^n  ( \varphi_\eta   (u_{\gamma^j}))^{\varepsilon^j} .}
\end{gather*}
Here we use  notation of Section \ref{coloredGgraphs}   and   set
$$
\rho_i= \left\{\begin{array}{ll} \id_{\varphi_\eta (u_{\gamma_i})}\co \varphi_\eta (u_{\gamma_i})\to \varphi_\eta (u_{\gamma_i}) & \text{if $\varepsilon_i=+$,}\\[.4em]
 \varphi_\eta^{1}(u_{\gamma_i})\colon \varphi_\eta
   (u_{\gamma_i}^{\ast} ) \to (\varphi_\eta
(u_{\gamma_i}))^{\ast} & \text{if $\varepsilon_i=-$. }\end{array}\right.
$$
and similarly
$$
\rho^j= \left\{\begin{array}{ll} \id_{\varphi_\eta (u_{\gamma^j})}\co \varphi_\eta (u_{\gamma^j})\to \varphi_\eta (u_{\gamma^j}) & \text{if $\varepsilon^j=+$,}\\[.4em]
 \varphi_\eta^{1}(u_{\gamma^j})\colon \varphi_\eta
   (u_{\gamma^j}^{\ast} ) \to (\varphi_\eta
(u_{\gamma^j}))^{\ast} & \text{if $\varepsilon^j=-$. }\end{array}\right.
$$
In the case where $m=n=1$ and $
\varepsilon_1= \varepsilon^1=+$, the definition of
$v^{\eta}_\gamma$ simplifies to $v^{\eta}_\gamma=\varphi_\eta
(v_\gamma)$.

\begin{lem}\label{lem-conjugation+}
$(u^{\eta}, v^{\eta})$ is a coloring of $\Omega^{\eta}$.
\end{lem}
\begin{proof}
We need to prove the commutativity of the diagram \eqref{weakiso---} for all $\gamma$,  $\beta$. For simplicity, we restrict ourselves to   coupon-tracks $\gamma$ with
   $m=n=1$ and $ \varepsilon_1=  \varepsilon^1=+$. Then the
   commutativity of  \eqref{weakiso---} follows from the
   commutativity of the diagram
$$
\xymatrix@R=1cm @C=1.15cm
{
\varphi_\eta (  u_{\beta \gamma_1 } )  \ar[r]^-{\varphi_\eta (  u_{\beta, \gamma_1}  ) }   \ar[d]_{v^{\eta}_{\beta \gamma}}
  &   \varphi_\eta \varphi_{g(\beta^{-1})} (  u_{ \gamma_1}  )  \ar[r]^-{\varphi_2 }  \ar[d]_{
   \varphi_\eta \varphi_{g(\beta^{-1})} (v_\gamma)}
  &   \varphi_{g(\beta^{-1})\eta} (  u_{  \gamma_1}  )  \ar[r]^-{(\varphi_2 )^{-1}} \ar[d]_{ \varphi_{g(\beta^{-1})\eta} (v_\gamma)}
   &   \varphi_{g^{\eta}(\beta^{-1}) } \varphi_{ \eta}  (  u_{  \gamma_1}  ) \ar[d]_{\varphi_{g^{\eta}(\beta^{-1}) } (v^{\eta}_\gamma)}\\
  \varphi_\eta (  u_{\beta \gamma^1}  )  \ar[r]^-{\varphi_\eta (  u_{\beta, \gamma^1}  ) }
  &   \varphi_\eta \varphi_{g(\beta^{-1})} (  u_{ \gamma^1}  )  \ar[r]^-{\varphi_2 }
  &   \varphi_{g(\beta^{-1})\eta} (  u_{  \gamma^1}  )  \ar[r]^-{(\varphi_2 )^{-1}}
   &   \varphi_{g^{\eta}(\beta^{-1}) } \varphi_{ \eta}  (  u_{  \gamma^1}  )
    .}
$$
\end{proof}

Isomorphisms   of   colorings can also  be
conjugated in the obvious way and yield isomorphisms   of the
conjugate  colorings.


\subsection{Behavior of   $F=F_{\cc} $
  under conjugation of colorings}

\begin{thm}\label{thm-conjugation}  Let $\cc$ be a $G$-ribbon  category, and let $\Omega$ be a   colored  $G$-graph  with source
$( u_{1}, \varepsilon_1),\ldots , ( u_{k}, \varepsilon_k)$ and
   target $( u^{1}, \varepsilon^1),\ldots ,( u^{l},
   \varepsilon^l)$. Let $\Omega^\eta$ be the
colored $G$-graph
   obtained from $\Omega$ through conjugation by $\eta\in G$.
Then the morphism   $F(\Omega^\eta)$  is equal to the following
 composition
\begin{gather*}
\xymatrix@R=0cm @C=1.7cm {
\otimes_{r=1}^k (\varphi_\eta (u_{ r}))^{\varepsilon_r}   \ar[r]^-{ \otimes_r \rho_r^{-1}}  &  \otimes_{r=1}^k \, \varphi_\eta (u_{ r}^{\varepsilon_r} )  \ar[r]^-{(\varphi_\eta)_k  }  &    \varphi_\eta (\otimes_{r=1}^k   \,u_{r}^{\varepsilon_r} ) \ar[r]^-{ \varphi_\eta(F(\Omega) )  } & {}}\\
\xymatrix@R=1cm @C=1.7cm
{\varphi_\eta (\otimes_{j=1}^l \,  (u^{ s})^{\varepsilon^s} )  \ar[r]^-{(\varphi_\eta)_l^{-1}  }
  &   \otimes_{s=1}^l   \,\varphi_\eta (  (u^{ s})^{\varepsilon^s} )  \ar[r]^-{ \otimes_s \rho^s}
   & \otimes_{s=1}^l  ( \varphi_\eta   (u^{ s}))^{\varepsilon^s} ,}
\end{gather*}
where    $\rho_r=\id_{\varphi_\eta (u_r)}$ if $\varepsilon_r=+$,
   $\rho_r=\varphi_\eta^{1}(u_r )$ if $\varepsilon_r=-$
   and similarly $\rho^s=\id_{\varphi_\eta   (u^{ s})}$ if $\varepsilon^s=+$,
   $\rho^s=\varphi_\eta^{1}(u^s)$  if $\varepsilon^s=-$.
\end{thm}

For $m=n=0$, we obtain   $ F(\Omega^\eta)=(\varphi_\eta)_0^{-1}
\varphi_\eta(F(\Omega) )   (\varphi_\eta)_0 $. In particular, if
the morphism $F(\Omega)\in \End_\cc(\un)$
is a scalar multiple of $\id_\un$, then
$F(\Omega^\eta)=F(\Omega)$.

\begin{proof} We give   the main lines of the proof leaving the details to the reader. First of all, a coloring $(U,V)$ of an arbitrary graph diagram $D$  determines an $\eta$-conjugate coloring   $(U^\eta, V^\eta)$ of $D$ as follows. For an underpass $p$ of $D$, set $U^\eta_p=\varphi_\eta (U_p)$.
For a crossing $c$ of $D$, let $U^\eta_c$ be the composition
\begin{gather*}
\xymatrix@R=1cm @C=2cm {
U^{\eta}_{c^+}=\varphi_{\eta} (U_{c^+})  \ar[r]^-{\varphi_{\eta} (U_c)}  &
\varphi_{\eta} \varphi_{\vert U_{\underline c} \vert } (U_{c^-})   \ar[r]^-{\varphi_2 ( \eta, \vert U_{\underline c} \vert )_{U_{c^-}}}  &    \varphi_{\vert U_{\underline c} \vert \eta} (U_{c^-})  }\\
\xymatrix@R=1cm @C=3cm {
 =\varphi_{\eta  \vert U^{\eta}_{\underline c} \vert} (U_{c^-})    \ar[r]^-{(\varphi_2 (  \vert U^{\eta} _{\underline c}\vert , \eta)_{U_{c^-}})^{-1}}
  & \varphi_{  \vert U^{\eta} _{\underline c} \vert } \varphi_\eta (U_{c^-})= \varphi_{  \vert U^{\eta} _{\underline c} \vert }  (U^{\eta}_{c^-}).
  }
\end{gather*}
The coloring $V^\eta$ of the coupons of $D$ is defined similarly to $v^\eta$ in Section
\ref{Conjugation of the colorings}.  The colored  diagram
$(D, U^\eta, V^\eta)$ is  denoted by $D^\eta$.

We next define the conjugation endofunctor $ \Phi_\eta$ of the
category $ \mathcal D_{\mathcal C}$. It transforms an object $
((U_r, \varepsilon_r))_{r=1}^k$  of $ \mathcal D_{\mathcal C}$
into the object $ (( \varphi_\eta (U_r), \varepsilon_r
))_{r=1}^k$. A morphism of $ \mathcal D_{\mathcal C}$
represented by a colored  diagram $D$ is transformed into
the   morphism represented by $D^\eta$. It is easy to see that $
\Phi_\eta$  is a strict monoidal functor. Comparing the values
on the elementary  diagrams,
one easily observes that  for any
morphism $f\colon ((U_r, \varepsilon_r))_{r=1}^k \to ((U^s ,
\varepsilon^s))_{s=1}^l$ in $ \mathcal D_{\mathcal C}$, the
following diagram commutes:
$$
\xymatrix@R=1cm @C=2cm {
\otimes_{r=1}^k (\varphi_\eta (U_r))^{\varepsilon_r}  \ar[r]^-{\otimes_r \rho_r^{-1}}   \ar[d]^-{ F \Phi_\eta (f)} &
 \otimes_{r=1}^k (\varphi_\eta (U_r^{\varepsilon_r}))  \ar[r]^-{(\varphi_\eta)_k}     &
 \varphi_\eta (\otimes_{r=1}^k    U_r^{\varepsilon_r} )      \ar[d]^-{ \varphi_\eta ( F(f))}   \\
\otimes_{s=1}^l (\varphi_\eta (U^s))^{\varepsilon^s}     \ar[r]^-{\otimes_s (\rho^s )^{-1}}   &
\otimes_{s=1}^l  \varphi_\eta ((U^s)^{\varepsilon^s} )    \ar[r]^-{(\varphi_\eta)_l }  &
\varphi_\eta ( \otimes_{s=1}^l  (U^s)^{\varepsilon^s} )   .}
$$
Here   $\rho_r=\id_{\varphi_\eta (U_r)}$ if $\varepsilon_r=+$,
   $\rho_r=\varphi_\eta^{1}(U_r)$ if $\varepsilon_r=-$
   and similarly $\rho^s=\id_{\varphi_\eta   (U^{ s})}$ if $\varepsilon^s=+$,
   $\rho^s=\varphi_\eta^{1}(U^s)$  if $\varepsilon^s=-$.

Observe finally that if a colored $G$-graph $\Omega$ is
represented by a colored     diagram $D$, then
$\Omega^\eta$ is represented by $D^\eta$. Now, the claim of
the theorem directly follows from the commutativity of the
previous diagram.
\end{proof}

\section{Invariants of special   $G$-graphs}\label{Invariants of special colored $G$-graphs}

In this section, we derive from a $G$-ribbon  category  $\cc$ an invariant of so-called special colored $G$-graphs (such graphs possibly have circle components). This construction  will be crucial in the definition of the surgery HQFT.

\subsection {Insertion of coupons}\label{sect-insert-coupon}  The  functor $\mathcal F$ of  Lemma~\ref{lem-functorF}  does not apply to colored $G$-graphs having circle components. In particular, this functor does not apply to knots and links. We show   how to transform   circle components into graphs with coupons and   to derive  from  $\mathcal F$  invariants of some $G$-links.

Let $(\Gamma, g\colon \pi_1(C_\Gamma)\to G)$
be a $\cc$-colored $G$-graph with circle components $(\ell_r)_r$. We transform $\Gamma$ into a colored $G$-graph $\Omega_\Gamma$ without circle   components as follows. Insert into each circle
component $\ell_r$ of $\Gamma$ a coupon $Q_r$ with one input and
one output, see   Figure~\ref{insertion-coupon}. In this figure
$Q_r$ is oriented counterclockwise, its bottom base  is the
bottom horizontal side, and the framing  is given (on the
boldface portions) by the   vector field orthogonal to the page
of the picture and directed behind the page.  In this way,
$\ell_r$ is transformed into a union of $Q_r$ and an oriented
segment $e_r $ for all $r$. These coupons and segments together
with $\Gamma \setminus \cup_r \ell_r$ form a ribbon graph
$\Omega=\Omega_\Gamma$. Clearly, $\pi_1(C_{\Omega })=\pi_1
(C_\Gamma)$ so that the homomorphism $g\colon \pi_1
(C_\Gamma)\to G $  induces a homomorphism $  \pi_1(C_{\Omega })
\to G$ also denoted~$g$.

\begin{figure}[t]
\begin{center}
\psfrag{e}[Br][Br]{\scalebox{.9}{$\ell_r$}}
\psfrag{a}[Br][Br]{\scalebox{.9}{$e_r$}}
\psfrag{Q}[Bc][Bc]{\scalebox{.9}{$Q_r$}}
\rsdraw{.45}{.9}{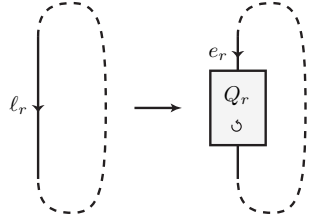}
\end{center}
\caption{Insertion of a coupon}
\label{insertion-coupon}
\end{figure}

The given coloring of $\Gamma$ determines a coloring of all
tracks of $\Omega$ except the tracks of the edges $(e_r)_r$ and
the coupons $(Q_r)_r$. By      colorings  of $\Omega$ we shall
mean only the colorings   extending this \lq\lq partial
coloring". To color the  tracks of   $(e_r)_r$ and $(Q_r)_r$,   fix a
$z$-path $\gamma (r)$ for each
$Q_r$ and set $\mu_r=\mu_{\gamma(r)}\in \pi_1(C_{\Omega })$.
Pushing the endpoint of $\gamma (r)$ on $\widetilde Q_r$ towards
the top (resp.\ bottom) base, we obtain a $z$-path  for $e_r$,
cf. Section \ref{coloredGgraphs}.
 Let   $\gamma^r$ and $\gamma_r$  be the corresponding tracks of
$e_r$. Clearly,    $\mu_{\gamma^r}=\mu_{\gamma_r}=\mu_r$ and
$\gamma_r=\lambda_r \gamma^r $ where $\lambda_r \in
\pi_1(C_\Omega)$ is the longitude  of $\ell_r$ determined by
$\gamma (r)$ and  the  orientation and the framing of $\ell_r$.
Given     objects $( X_{r} \in {{\mathcal C}}_{g(\mu_r)})_{r }
$, Lemma \ref{lem-weakcolorings} yields an edge-coloring $u$ of
$\Omega $ such that $u_{\gamma_r}= X_{r}$ for all $r$.     Note
that   $u$ gives an isomorphism $u_{\lambda_r^{-1}, \gamma_r}
\colon u_{\gamma^r}  \to \varphi_{ g(\lambda_r) } (X_{r}) $ for
all~$r$. We color  each coupon-track  $\gamma (r)$ with the
composition of a morphism $     X_{r}  \to \varphi_{
g(\lambda_r) } (X_{r}) $ with $ (u_{\lambda_r^{-1}, \gamma_r})^{-1}  $.
By  Section~\ref{coloredGgraphs}, this extends to a coloring  of $\Omega$.

\subsection{Special colored $G$-graphs}\label{sect-special-G-graphs}
We call a colored $G$-graph $(\Gamma, g)$  {\it special} if
if the longitudes of all circle components of $\Gamma$ lie in $\Ker \, g$. An example of a
special $G$-graph is provided by the
{\it trivial $G$-knot} defined as a framed oriented unknot
in $S^3$ with trivial homomorphism of the fundamental group
of the complement to $G$. A more general example is provided
by any framed oriented link   $\ell\subset S^3$ endowed with
homomorphism $\pi_1(S^3\setminus \ell)\to G$ carrying the
longitudes of all components of $\ell$ to $1$. We call such
links {\it special $G$-links}.

Let $(\Gamma, g\colon \pi_1(C_\Gamma)\to G)$ be a special $\cc$-colored $G$-graph with circle components $(\ell_r)_r$.
As in Section~\ref{sect-insert-coupon}, consider an edge-coloring $u$ of  $\Omega=(\Omega_\Gamma, g)$ such
that $u_{\gamma_r}= X_{r}\in {{\mathcal C}}_{g(\mu_r)}$ for all
 $r$. Consider the   isomorphism $u_{\lambda_r^{-1}, \gamma_r}
 \colon u_{\gamma^r}  \to \varphi_{ g(\lambda_r) }
 (X_{r})=\varphi_{1 } (X_{r}) $. Given morphisms  $ (f_r\in
 \End_{\mathcal  C}(X_r))_{r=1}^n$, we color  each    $\gamma
 (r)$ with
the morphism
\begin{equation}\label{vvv}
v_{\gamma(r)}= (u_{\lambda_r^{-1}, \gamma_r})^{-1}
(\varphi_0)_{X_r}  f_r \colon u_{\gamma_r} \to
u_{\gamma^r}.
\end{equation}
This extends uniquely to
a  coloring $(u,v)$ of $\Omega$.  The resulting colored
$G$-graph is denoted  $\Omega  (X_r, f_r, \gamma(r))
 $. Different choices of $u$ lead to isomorphic
colored $G$-graphs    so that $F( \Omega  (X_r, f_r,\gamma(r)) )
 \in \End_{\mathcal  C}(\un)$ does not depend on the choice of
 $u$. The map $(X_r, f_r)_r\mapsto F( \Omega  (X_r,
 f_r,\gamma(r)) ) $ extends by $\kk$-linearity to an $n$-linear
 form
\begin{equation}\label{eqq+--}
\otimes_{r=1}^n \widetilde L_{g(\mu_r)} \to \End_{\mathcal  C}(\un),
\end{equation}
where $\widetilde L$ is defined in Section \ref{sect-verlinde-algebra}. Generally speaking, the form \eqref{eqq+--}
depends on the choice of the tracks $(\gamma(r))_r$. Before
exploring this form note that every coloring $(u,v)$ of
$\Omega=(\Omega_\Gamma, g)$ is obtained as above from a unique
family of {\it associated morphisms} $(f_r \in \End_{\mathcal
C}(X_r))_r$ computed   by $f_r=(\varphi_0)_{X_r}^{-1} \,
u_{\lambda_r^{-1}, \gamma_r} \, v_{\gamma(r)}$.

We now study the form   \eqref{eqq+--}.  We begin   with the
following lemma.

\begin{lem}\label{lem-smallomega--}
For any $r=1, \ldots, n$ and any $f_r, h_r \in \End_{\mathcal  C}(X_r)$ the morphism $f_r h_r -h_r f_r \in \End_{\mathcal  C}(X_r)$  lies in the annihilator of   \eqref{eqq+--}.
\end{lem}

\begin{proof}  Assume for simplicity that $\Gamma=\ell_1 $ is a knot so that $r=n=1$; the general case is  similar.  Set $\pi=\pi_1(C_\Gamma)=\pi_1(C_\Omega)$. Let $\gamma=\gamma (1)$ be the fixed track of the coupon $Q=Q_1$ of $\Omega$
 and let $\mu=\mu_\gamma \in \pi$ be the associated meridian.
  We must prove that $F( \Omega  (X, fh,\gamma) )=F( \Omega  (X,
hf,\gamma) )$ for  any object $  X  \in {{\mathcal C}}_{g(\mu )}
$ and any endomorphisms $f,h$ of $X$.  This equality is well
known in the non-crossed case; the   proof  goes by replacing
the $fh$-colored coupon   with two coupons colored with $f, h$,
then pushing the $h$-colored coupon  along $\Gamma$ so that it
comes on top of the $f$-colored coupon and finally replacing the
two resulting coupons with an $hf$-colored coupon. These
operations preserve the invariant $F$ and yield the required
equality. The crossed case is similar but needs a more careful
treatment as follows.

A schematic picture of the colored $G$-graph $\Omega=\Omega  (X,
fh,\gamma)$ is shown in Figure~\ref{fig-eee}. Here  $\gamma^1$
(resp.\ $\gamma_1$) is the   track  of the   edge $e=e_1$ of
$\Omega$ obtained by slightly pushing $\gamma$  to the top
(resp.\ bottom) of $Q$. We have $\gamma^1=\lambda^{-1}
\gamma_1$, where $\lambda  \in \pi $ is the longitude  of
$\Gamma $ determined by $\gamma$. By definition, the coloring
$(u,v)$ of  $\Omega $ satisfies $u_{\gamma_1}=X$ and $v_\gamma=
(u_{\lambda^{-1}, \gamma_1})^{-1} \, (\varphi_0)_{X} \, fh
\colon X \to Y$ where $Y=u_{\gamma^1} \in {{\mathcal C}}_{g(\mu
)} $.

\begin{figure}[t]
\begin{center}
      \psfrag{x}[Bc][Bc]{\scalebox{.8}{$\gamma^1$}}
      \psfrag{c}[Bc][Bc]{\scalebox{.8}{$\gamma$}}
      \psfrag{n}[Bc][Bc]{\scalebox{.8}{$\gamma_1$}}
      \psfrag{e}[Bl][Bl]{\scalebox{.9}{$e$}}
      \psfrag{Q}[Bc][Bc]{\scalebox{.9}{$Q$}}
      \psfrag{a}[Bc][Bc]{\scalebox{.8}{$\eta$}}
      \psfrag{v}[Bl][Bl]{\scalebox{.8}{$\rho^1\!=\!\eta_1$}}
      \psfrag{u}[Bc][Bc]{\scalebox{.8}{$\rho$}}
      \psfrag{o}[Bc][Bc]{\scalebox{.8}{$\eta^1$}}
      \psfrag{s}[Bc][Bc]{\scalebox{.8}{$\rho_1$}}
      \psfrag{Q}[Bc][Bc]{\scalebox{.9}{$Q$}}
      \psfrag{z}[Br][Br]{\scalebox{.9}{$z$}}
      \psfrag{r}[Bl][Bl]{\scalebox{.9}{$\zeta$}}
      \psfrag{l}[Bl][Bl]{\scalebox{.9}{$d'$}}
      \psfrag{d}[Bl][Bl]{\scalebox{.9}{$d$}}
      \psfrag{A}[Bc][Bc]{\scalebox{1}{$\Omega$}}
      \psfrag{B}[Bc][Bc]{\scalebox{1}{$\Omega'$}}
\rsdraw{.45}{1}{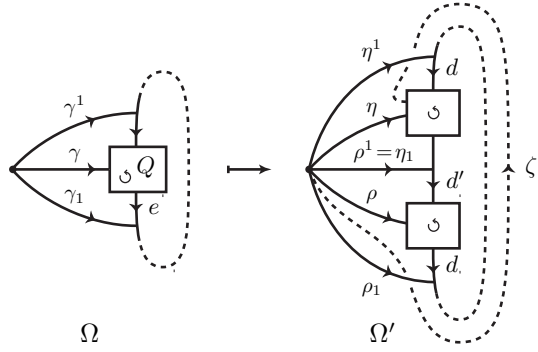}
\end{center}
\caption{The ribbon graphs $\Omega$ and $\Omega'$}
\label{fig-eee}
\end{figure}

We replace the coupon $Q$ with two coupons as  in Figure
\ref{fig-eee}. The resulting  ribbon graph $\Omega'$   has   two
edges $d, d'$
   where $d\subset e$. We identify $ \pi_1(C_{\Omega'})=\pi$ in
   the obvious way so that $\Omega'$ becomes a $G$-graph. The
   track $\gamma$ determines in the obvious way the tracks
    $\rho$, $\eta$ of  the two coupons  of $\Omega'$  so that
    $\rho^1=\eta_1$  and $\rho_1=\gamma_1$, $\eta^1=\gamma^1$
    (as tracks of $e$). The edge-coloring $u$ of $\Omega$
    induces an edge-coloring of $\Omega'$ as follows. Each track
    of $d$ in $C_{\Omega'}$ determines a track of $e$ in
    $C_{\Omega}$ and keeps its $u$-color;
     the same holds for  the isomorphisms determined by pairs
(an element of $\pi $, a track of $d$). We   color the track
$\rho^1=\eta_1$ of $d'$ with $X$. This data extends to an
edge-coloring $u'$ of~$\Omega'$. Note that
$u'_{\rho_1}=u_{\gamma_1}=X$ and $u'_{\eta^1}=u_{\gamma^1}=Y$.
The edge-coloring $u'$ extends to a coloring $(u',v')$ of
$\Omega'$ such that $v'_\rho=h$ and  $v'_\eta= (u_{\lambda^{-1},
\gamma_1})^{-1} \, (\varphi_0)_{X} \, f \colon X \to Y$. Note
for the record   that  $\eta^1=\lambda^{-1} \rho_1$  and
$u'_{\lambda^{-1}, \rho_1}=u_{\lambda^{-1}, \gamma_1}\colon Y
\to \varphi_1(X)$.

Pushing the upper coupon   along $d'$ and the lower coupon along
$d$,  we obtain a self-homeomorphism $j$ of  ${\mathbb
R}^2\times [0,1]$ transforming $\Omega'$ into itself and
permuting its edges and   coupons. The homeomorphism $j$ induces
the identity automorphism of  $\pi=\pi_1(C_{\Omega'})$.
Transporting the coloring $(u',v')$ of $\Omega'$ along $j$ we
obtain a    coloring $(u'', v'')$ of   $\Omega'$.  Note that $j$
transforms the tracks $\eta, \eta_1, \eta^1$  into the   tracks
$\rho, \rho_1, \rho^1$, respectively. Therefore
 $u''_{\rho_1} =u'_{\eta_1}=X$, $u''_{\rho^1} =u'_{\eta^1}=Y$,
and $  v''_{\rho} =v'_{\eta}  \colon X \to Y$.  Consider the
coupon-track $\zeta= j(\rho)$. Clearly,  $\zeta^1=j(\rho^1)$ and
$\zeta_1= j(\rho_1)$. Thus,  $u''_{\zeta^1}=u'_{\rho^1}=X$,
$u''_{\zeta_1}=u'_{\rho_1}=X$ and $v''_{\zeta}=v'_{\rho}=h$. The
definition of $j$ implies that the tracks $\zeta^1$ and $\rho_1$
of   $d$ are equal: $\zeta^1=\rho_1$.

Next, we contract the edge $d'$ of  $\Omega'$ to obtain the same
$G$-graph $\Omega$ as before. The coloring $(u'', v'')$ of
$\Omega'$ induces a coloring $(\overline u, \overline v)$ of
$\Omega$ such that $\overline u_{\gamma_1}= u''_{\rho_1}=X$,
$\overline u_{\gamma^1}= u''_{\eta^1}$, and $\overline
v_\gamma=v''_\eta v''_\rho \colon X\to \overline u_{\gamma^1}=
u''_{\eta^1}$.  It follows from the definitions that   all these
transformations of $\Omega$   preserve the invariant $F$:
$$
F(\Omega,  u,v)= F(\Omega',  u',v')  = F(\Omega',  u'',v'') = F(\Omega,  \overline u, \overline v).
$$
The first equality is obtained by applying the definition of $F$
to a diagram representing $\Omega$ and chosen so  that    the
diagrammatic track of $Q$ is equal to~$\gamma$ (such a diagram
exists for any $\gamma$). The third equality is proven
similarly. The second equality holds because  the colored
$G$-graphs $(\Omega',  u',v')  $ and $(\Omega',  u'',v'')$ are
isotopic.

To finish the proof, we identify the colored $G$-graph
$(\Omega,  \overline u, \overline v)$ with $\Omega(X,
hf,\gamma)$. It is enough to show that the endomorphism, $x$,
of $X$ associated with the coupon-track $\gamma$ of $(\Omega,
\overline u, \overline v)$  is equal to $hf$. By definition,
\begin{gather*}
x=(\varphi_0)_{X}^{-1} \, \overline u_{\lambda^{-1}, \gamma_1} \, \overline v_{\gamma} =(\varphi_0)_{X}^{-1} \,   u''_{\lambda^{-1}, \rho_1} v''_\eta v''_\rho=(\varphi_0)_{X}^{-1} \,   u''_{\lambda^{-1}, \rho_1} v''_\eta v'_\eta  \\
= (\varphi_0)_{X}^{-1} \,   u''_{\lambda^{-1}, \rho_1} v''_\eta (u_{\lambda^{-1}, \gamma_1})^{-1} \, (\varphi_0)_{X} \, f = (\varphi_0)_{X}^{-1} \,   u''_{\lambda^{-1},  \zeta^1} v''_\eta (u''_{\lambda^{-1}, \zeta_1})^{-1} \, (\varphi_0)_{X} \, f
\end{gather*}
where the last equality follows from the formulas
$\zeta^1=\rho_1$ and $u''_{\lambda^{-1},
\zeta_1}=u'_{\lambda^{-1}, \rho_1}=u_{\lambda^{-1}, \gamma_1}$.
Using  the equality of coupon-tracks $\eta= \lambda^{-1} \zeta$
and the definition of a coloring, we obtain
$$
u''_{\lambda^{-1},  \zeta^1} v''_\eta (u''_{\lambda^{-1}, \zeta_1})^{-1}=u''_{\lambda^{-1},  \zeta^1} v''_{\lambda^{-1} \zeta} (u''_{\lambda^{-1}, \zeta_1})^{-1}= \varphi_{g (\lambda )} (v''_\zeta)=\varphi_1(h).
$$
Hence $x=(\varphi_0)_{X}^{-1} \,   \varphi_1(h) \,
(\varphi_0)_{X} \, f =hf$.
\end{proof}

Lemma \ref{lem-smallomega--} implies that  the $n$-linear form
\eqref{eqq+--} induces   an $n$-linear form
\begin{equation}
\label{eqq+} \otimes_{r=1}^n   L_{g(\mu_r)} \to \End_{\mathcal  C}(\un),
\end{equation}
where $L=L(\cc)$ is the fusion algebra of $\cc$ (see Section~\ref{sect-verlinde-algebra}). Given a  family of vectors $\omega=(\omega_\alpha \in
L_\alpha)_{\alpha\in G}$, we can  evaluate   \eqref{eqq+} on  $
\otimes_{r=1}^n \,   \omega_{g(\mu_r)}$. This yields an element
of $ \End_{\mathcal  C}(\un)$ denoted $F(\Gamma,  g,\omega,
(\gamma(r))_r)$.

We say that $\omega$ is {\it conjugation invariant} if
$\varphi_\beta(\omega_\alpha) =\omega_{\beta\alpha\beta^{-1}}$
for all $\alpha, \beta \in G$, where $\varphi\co G \to \mathrm{Aut}(L)$ is defined in Section~\ref{sect-verlinde-algebra}.

\begin{lem}\label{lem-smallomega}
If $\omega$ is conjugation invariant, then $ F(\Gamma,  g,\omega,
(\gamma(r))_r)$     does not depend on the      tracks
$(\gamma(r))_r$ and  $F(\Gamma,  g,\omega )=F(\Gamma,  g,\omega,
(\gamma(r))_r)$ is an isotopy invariant of the special colored $G$-graph
$(\Gamma, g)$.
\end{lem}

\begin{proof}
We need to prove that for any  $(\beta_r \in \pi_1(C_\Omega))_{r=1, \ldots, n}$,
\begin{equation}\label{FFF}
F(\Gamma,  g,\omega, (\gamma(r))_r)=F(\Gamma,  g,\omega, (\beta_r \gamma(r))_r).
\end{equation}
Pick   any objects $( X_{r} \in {{\mathcal C}}_{g(\mu_r)})_{r } $ and their endomorphisms $(f_r)_r$. For  all $r  $, set
$$Y_r=\varphi_{g(\beta_r^{-1})} (X_r)\in \mathcal C_{g(\beta_r \mu_r \beta_r^{-1})} \quad {\text {and}} \quad j_r=\varphi_{g(\beta_r^{-1})} (f_r) \in  \End_{\mathcal  C}(Y_r).$$ Consider the colored $G$-graphs
  $\Psi =\Omega  (X_r, f_r,\gamma(r)) $ and $\Psi'=\Omega  (Y_r,
  j_r, \beta_r \gamma(r)) $  with the same underlying $G$-graph
  $\Omega$. We construct below an isomorphism of colored graphs
  $\Psi \approx \Psi '$. The existence of such an isomorphism
  implies that $F(\Psi )=F(\Psi ')$.  Since this holds for all
  choices of $X_r, f_r$,  the forms   \eqref{eqq+} determined by
  the coupon-tracks  $(\gamma(r))_r$  and $(\beta_r
  \gamma(r))_r$ are obtained from each other via the isomorphism
  $$ \otimes_{r=1}^n   \varphi_{g(\beta_r^{-1})} \colon \otimes_{r=1}^n   L_{g(\mu_r)} \to \otimes_{r=1}^n   L_{g(\beta_r \mu_r \beta_r^{-1})}.$$
  By assumption on $\omega$, this  isomorphism  carries $
 \otimes_{r=1}^n   \omega_{g(\mu_r)}$ into $ \otimes_{r=1}^n
 \omega_{g(\beta_r \mu_r \beta_r^{-1})}$. This implies
 \eqref{FFF}.

Denote the colorings   of $\Psi $ and $\Psi'$  by  $(u,v)$
  and $(u',v')$ respectively. By definition ${u}_{\gamma_r}=X_r$
  and $u'_{\beta_r\gamma_r}=Y_r=\varphi_{g(\beta_r^{-1})} (X_r)$
  for all $r$.  The isomorphisms $$h_{\beta_r
  \gamma_r}={u}_{\beta_r, \gamma_r}\colon {u}_{\beta_r
  \gamma_r}\to \varphi_{g(\beta_r^{-1})} ({u}_{\gamma_r})
  =u'_{\beta_r\gamma_r}$$ with $r=1,\ldots, n$ extend  uniquely
  to an isomorphism of   edge-colorings $h\colon {u}\to u'$.
  That $h$ transforms $v$  into $v'$ follows from the fact that
  $h$  conjugates $v_{\beta_r\gamma (r)}$ and
  $v'_{\beta_r\gamma (r)}$ for all $r$. To verify this fact, fix
  $r$ and note that  it is enough to show that    $h_{\beta_r
  \gamma_r}={u}_{\beta_r, \gamma_r}$ conjugates the   morphisms
  $f\in  \End_{\mathcal
C}({u}_{\beta_r\gamma_r})$ and $f'\in \End_{\mathcal
C}(u'_{\beta_r\gamma_r})=\End_{\mathcal C}(Y_r)$ associated with
$v_{\beta_r \gamma (r)}$ and    $v'_{\beta_r \gamma (r)}$,
respectively. By the definition of $v'$, we have $f' =
j_r=\varphi_{g(\beta_r^{-1})} (f_r)$. The homomorphism $f $ is
computed by
\begin{equation}\label{forf}
f= (\varphi_0)^{-1}_{u_{\beta_r\gamma_r}} \,   u_{\beta_r \lambda_r^{-1}\beta_r^{-1}, \beta_r \gamma_r} \,  v_{\beta_r \gamma (r)} .
\end{equation}
The naturality of $\varphi_0$ yields the  commutative diagram
\begin{equation}\label{morphismf-}
\begin{split}
\xymatrix@R=1cm @C=2cm {
{u}_{\beta_r \gamma_r} \ar[d]_{(\varphi_0)_{u_{\beta_r\gamma_r}}} \ar[r]^{u_{\beta_r, \gamma_r}} &
   \varphi_{g(\beta_r^{-1}) }  (  u_{ \gamma_r} )   \ar[d]^{(\varphi_{0} )_{\varphi_{g(\beta_r^{-1}) }  (  u_{ \gamma_r} )} }  \\
\varphi_1({u_{\beta_r\gamma_r}})  \ar[r]^{\varphi_1({u}_{\beta_r, \gamma_r})}  &
\varphi_1 \varphi_{g(\beta_r^{-1})} ({u}_{\gamma_r})  .}
\end{split}
\end{equation}
 Expanding $(\varphi_0)^{-1}_{u_{\beta_r\gamma_r}}$ from this
diagram and substituting in \eqref{forf} we obtain that  $$
u_{\beta_r,\gamma_r} f=      \varphi_2(1, g(\beta_r^{-1})
)_{u_{\gamma_r}} \,   \varphi_1({u_{\beta_r,\gamma_r}})   \,
u_{\beta_r \lambda_r^{-1}\beta_r^{-1}, \beta_r \gamma_r}
v_{\beta_r \gamma (r)} =     u_{\beta_r \lambda_r^{-1},
\gamma_r}    v_{\beta_r \gamma (r)} .$$ Here the first equality
holds because    the inverse of the right vertical arrow in
\eqref{morphismf-} is   $\varphi_2(1, g(\beta_r^{-1})
)_{u_{\gamma_r}}$ and the second equality follows from the
commutativity of \eqref{condweak}.

Consider now the   following  commutative diagram of
isomorphisms:
\begin{equation}\label{morphismf}
\begin{split}
\xymatrix@R=1cm @C=2.3cm {
{u}_{\beta_r \gamma_r} \ar[d]^{v_{\beta_r \gamma (r)}} \ar[r]^{{u}_{\beta_r, \gamma_r}}&
 \varphi_{g(\beta_r^{-1})} ({u}_{\gamma_r})  \ar[r]^-{f'=\varphi_{g(\beta_r^{-1})} (f_r ) } \ar[d]^{\varphi_{g(\beta_r^{-1})} (v_{\gamma(r)})} &    \varphi_{g(\beta_r^{-1})} ({u}_{\gamma_r}) \ar[d]^{\varphi_{g(\beta_r^{-1})} ((\varphi_0)_{{u}_{\gamma_r}})}\\
{u}_{\beta_r \gamma^r}   \ar[r]^{{u}_{\beta_r, \gamma^r}}  &
 \varphi_{g(\beta_r^{-1})} ({u}_{\gamma^r})  \ar[r]^-{\varphi_{g(\beta_r^{-1})} ({u}_{\lambda_r^{-1}, \gamma_r}) } &
 \varphi_{g(\beta_r^{-1})} \varphi_1 ({u}_{\gamma_r}) .}
\end{split}
\end{equation}
Here the left square is the diagram \eqref{weakiso---} for the
  $r$-th coupon   of $\Omega$  (this coupon has   one entry and
  one exit so that $\otimes$  and     $\varphi_2$ do not come
  up.)  The  right square    is commutative by    \eqref{vvv}.
  Note that the inverse of the rightmost vertical arrow is equal
  to $\varphi_2(g(\beta_r^{-1}), 1)_{u_{\gamma_r}}$. The
  commutativity of the diagram \eqref{condweak} implies that
  the composition of the isomorphisms in  \eqref{morphismf},
  going from the bottom left to the bottom right and then to the
  top right is equal to $u_{\beta_r \lambda_r^{-1}, \gamma_r}$.
Therefore
\begin{equation*}
f' \,   {u}_{\beta_r, \gamma_r} =u_{\beta_r \lambda_r^{-1}, \gamma_r}  \,  v_{\beta_r \gamma (r)}=u_{\beta_r,\gamma_r} \,  f. \qedhere
\end{equation*}
\end{proof}

\begin{lem}\label{lem-smallomega--ee}
Assume $\End_\cc(\un)=\kk\,\id_\un$ and $\omega$ is conjugation invariant. Then
$F(\Gamma,  \eta^{-1} g \eta ,\omega )=F(\Gamma,  g,\omega )$ for all $\eta\in G$.
\end{lem}

\begin{proof}
We   prove a stronger claim which concerns an arbitrary, not necessarily conjugation invariant  family of vectors $\omega=(\omega_\alpha \in L_\alpha)_{\alpha\in G}$.
Consider the family   $\omega^\eta =(\omega^\eta_\alpha\in
L_\alpha)_\alpha$ defined by $\omega^\eta_\alpha=\varphi_\eta
(\omega_{\eta \alpha\eta^{-1}})$ for all $\alpha$.  We claim
that
\begin{equation}\label{conji}
F(\Gamma,  \eta^{-1} g \eta,\omega^\eta, (\gamma(r))_r)=F(\Gamma,  g,\omega, (\gamma(r))_r)
\end{equation}
where $(\gamma(r))_r$ are the coupon-tracks as above. If
$\omega$ is conjugation invariant, then $\omega^\eta=\omega$ and
the   lemma follows.

To prove \eqref{conji}, consider the $G$-graphs
$\Omega=(\Omega_\Gamma, g)$ and $\Omega^\eta=(\Omega_\Gamma,
\eta^{-1} g \eta)$. Let $ \gamma_r,  X_{r}  , f_r $  be as
above.         The $\eta$-conjugation transforms the colored
$G$-graph $  (\Omega, u,v) $ $=\Omega  (X_r, f_r,\gamma(r))$
into a colored $G$-graph $ (\Omega^\eta, u^\eta, v^\eta)$. The
endomorphism of  $u^\eta_{\gamma_r}= \varphi_\eta
(u_{\gamma_r})=  \varphi_\eta (X_r)$ associated with $(u^\eta,
v^\eta)$ is computed from the definitions to be
\begin{gather*}
(\varphi_0)_{ \varphi_\eta (X_r)}^{-1} \, u^\eta_{\lambda_r^{-1}, \gamma_r} \,  \varphi_\eta(v_{\gamma(r)})\\
=(\varphi_0)_{ \varphi_\eta (X_r)}^{-1} \, (\varphi_2(1, \eta)_{X_r})^{-1} \, \varphi_2( \eta, 1 )_{X_r} \, \varphi_\eta (u_{\lambda_r^{-1}, \gamma_r})\, \varphi_\eta( u^{-1}_{\lambda_r^{-1}, \gamma_r}  (\varphi_0)_{X_r} f_r)\\
=(\varphi_0)_{ \varphi_\eta (X_r)}^{-1} \, (\varphi_2(1, \eta)_{X_r})^{-1} \, \varphi_2( \eta, 1 )_{X_r} \,   \varphi_\eta( (\varphi_0)_{X_r})\, \varphi_\eta(f_r)= \varphi_\eta(f_r).
\end{gather*}
Therefore    $ (\Omega^\eta, u^\eta, v^\eta)=\Omega^\eta
(\varphi_\eta (X_r), \varphi_\eta ( f_r),\gamma(r))$.  Since
$\End_\cc(\un)=\kk\,\id_\un$, the  remark after the statement
of Theorem \ref{thm-conjugation} implies  that
$$
F(\Omega^\eta  (\varphi_\eta (X_r), \varphi_\eta (
f_r),\gamma(r)))= F(\Omega^\eta, u^\eta, v^\eta)= F(\Omega,
u,v)=F( \Omega  (X_r, f_r,\gamma(r))) .
$$
Extending by linearity to the fusion algebra, we obtain  \eqref{conji}.
\end{proof}

\begin{lem}\label{lem-smallomega+}
If $\omega$ is conjugation invariant and  $\omega_\alpha^* =\omega_{ \alpha^{-1}}$ for all
$\alpha  \in G$, then $F(\Gamma,  g,\omega )$  does not
depend on the orientation of the circle components of $\Gamma$.
\end{lem}

\begin{proof}  This  lemma follows from a stronger claim which says that
$F( \Omega  (X_r, f_r,\gamma(r)) )$ is preserved when the
orientation of   $\ell_i$   is reversed and simultaneously   the
pair $(X_i, f_i\in \End (X_i)) $ is replaced with  $(X_i^*,
f_i^*\in \End (X_i^*)) $ (keeping the rest of the data). This
claim is well known in the non-crossed setting.   The proof in
the crossed setting is left to the reader as an exercise.
\end{proof}

An example of a  family of vectors $\omega$ satisfying
all requirements of the last three lemmas will be given in the
next section.

\subsection{Remark} \label{RemarkAssMorphism}     The morphism $ (u_{\lambda_r^{-1}, \gamma_r})^{-1} (\varphi_0)_{X_r}\colon X_r=u_{\gamma_r}\to u_{\gamma^r}$ appearing in~\eqref{vvv} may be rewritten as
$(\varphi_0)^{-1}_{u_{\gamma^r}}  u_{\lambda_r, \gamma^r}
$. Indeed,
\begin{gather*}
u_{\lambda_r^{-1}, \gamma_r} (\varphi_0)^{-1}_{u_{\gamma^r}}   u_{\lambda_r, \gamma^r} = (\varphi_0)^{-1}_{\varphi_1(X_r)} \varphi_1( u_{\lambda_r^{-1}, \gamma_r} )  u_{\lambda_r, \gamma^r}\\
=(\varphi_0)^{-1}_{\varphi_1(X_r)}(\varphi_2(1,1)_{\varphi_1(X_r)})^{-1}  (\varphi_0)_{X_r}=(\varphi_0)_{X_r}.
\end{gather*}
These equalities follow respectively  from the naturality of
$\varphi_0$,   the commutativity of the diagram
\eqref{condweak}, and the  commutativity of the left triangle in
\eqref{crossing6}.

\section{$G$-modular categories}\label{$G$-modular categories}

\subsection{Pre-fusion and fusion categories}\label{sect-pre-fusion-cat}
We call an object $U$ of a $\kk$-additive category $\cc$
\emph{simple} if $\End_\cc(U)$ is a free \kt module of rank 1 (and
so has the basis $\{\id_U\}$).     It is clear that   an object
isomorphic to a simple object is itself  simple. If~$\cc$ is pivotal, then the dual of a   simple object of $\cc$ is  simple.

A \emph{split semisimple category (over $\kk$)} is a  $\kk$-additive
category $\cc$ such that each object of $\cc$ is a finite direct sum of simple objects and $\Hom_\cc(i,j)=0$ for any non-isomorphic
simple objects $i,j  $ of $\cc$.

Clearly, the Hom spaces in
such a $\cc$ are   free $\kk$-modules of finite rank. For  $X\in\cc$
and a  simple object $i \in \cc$, the  modules $\Hom_\cc(X,i)$
and $\Hom_\cc(i,X)$ have same   rank denoted $N^i_X $ and
called the \emph{multiplicity number}. A set $I$ of simple objects of $\cc$ is \emph{representative} if every
simple object of $\cc$ is isomorphic to a unique element of~$I$.

A \emph{pre-fusion category (over $\kk)$} is a split semisimple
$\kk$-additive pivotal category $\cc$ such that the unit object
$\un$ is  simple. In such a   category, the map $\kk \to
\End_\cc(\un), k \mapsto k \, \id_\un$  is a \kt algebra isomorphism
which we use  to identify $\End_\cc(\un)=\kk$. The left and right
dimensions of any simple object of a pre-fusion category are
invertible (see, for example, Lemma 4.1 of \cite{TVi1}).

If $I$ is representative set of simple objects of pre-fusion category $\cc$, then for any object $X$ of $\cc$, $N_X^i=0$ for all but a finite number of $i\in I$, and
\begin{equation}\label{eq-dim-mult-numbers}
\dim_l(X)=\sum_{i\in I} \dim_l(i) N_X^i, \qquad \dim_r(X)=\sum_{i\in I} \dim_r(i) N_X^i.
\end{equation}

A \emph{fusion category} is a pre-fusion category such that the
set of isomorphism classes of simple objects is finite. The
\emph{dimension} $\dim(\cc)$ of a   fusion   category $\cc$ is
$$
\dim(\cc)=\sum_{i\in I} \dim_l(i)\dim_r(i) \in \kk,
$$
where $I$ is a (finite)  set   of   simple objects of $\cc$ such
that every simple object of~$\cc$ is isomorphic to a unique element
of~$I$. The sum on the right-hand side   does not depend on the
choice of~$I$.

\subsection{$G$-fusion categories}\label{sect-G-fusion}
In a pre-fusion $G$-graded category
$\cc=\oplus_{g \in G} \, \cc_g$, every simple object is isomorphic
to a simple object of $\cc_g$ for a unique $g\in G$. Moreover, for
all $g\in G$, each object of $\cc_{g}$ is a finite direct sum of
simple objects of $\cc_{g}$.

A \emph{$G$-fusion category} is a
pre-fusion $G$-graded category~$\cc$ such that   the set of isomorphism
classes of simple objects of~$\cc_g$   is finite and non-empty for
every $g\in G$. For $G=1$, we obtain the    notion of a fusion
category (see Section~\ref{sect-pre-fusion-cat}).   The neutral
component $\cc_1$ of a $G$-fusion category $\cc$ is a fusion
category. A $G$-fusion category is a fusion category if and only if
$G$ is finite.

  The argument   in \cite[Section VII.1]{Tu1}  shows  that if $\cc=\oplus_{g \in G}\, \cc_g$ is a $G$-fusion category, then for all $g\in G$,
\begin{equation}\label{eq-dim-Cg}
\sum_{i \in I_g} \dim_l(i)\dim_r(i)=\dim(\cc_1),
\end{equation}
where $I_g$ is any representative set of simple objects of $\cc_g$.

The fusion algebra $L$ of a $G$-fusion category $\cc$ (see Section~\ref{sect-verlinde-algebra}) is
a  free \kt module with basis $ (\langle i \rangle )_{i\in  I }$,
where $I$ is an arbitrary representative set of simple objects of
$\cc$.  For $X\in {\cc} $, the vector  $\langle X \rangle\in L $
expands as  $\langle X \rangle= \sum_{i\in I } N^i_X
\,\langle i \rangle  $.

\subsection{$G$-modular categories}\label{$G$-modular categories++}
A $G$-ribbon category $\cc$ is \emph{spherical} in the sense that the left and right traces of any endomorphism $g$ in $\cc$  coincide. This fact is a consequence of Theorem~\ref{thm-functorF+}, since the graph diagrams representing $\tr_l(g)$ and $\tr_r(g)$ (see Section~\ref{sect-penrose}) are related by Reidemeister moves of type 1 and 2. Then $\tr_{l}(g)=\tr_r(g)$ is denoted $\tr(g)$ and called the \emph{trace} of $g$. Consequently, the left and right dimensions of an object $X$ of $\cc$ coincide. Then $\dim_{l}(X)=\dim_{r}(X)$ is denoted $\dim(X)$ and called the \emph{dimension} of $X$.

The neutral component $\cc_1$ of a  $G$-ribbon   $G$-fusion category $\cc$ is a ribbon fusion
category. Let $I_1$ be a (finite) representative set of simple
objects of $\cc_1$. For $i,j \in I_1$,   set
$$
S_{i,j}= \tr(c_{j,i}\circ c_{i,j}\colon  i\otimes  j\to
i\otimes j)\in \End_{\cc}(\un)=\kk
$$
where $c_{i,j}\co i\otimes j\to j\otimes i$ is the braiding in $\cc_1$  (see Section~\ref{sect-cas-C1}).
The matrix  $S=[S_{i,j}]_{i,j\in I_1} $ does not depend on the
choice of $I_1$ and is called the {\it $S$-matrix} of $\cc$.

A \emph{$G$-modular category} is a   $G$-ribbon   $G$-fusion category whose $S$-matrix is invertible (over $\kk$).
In other words, a $G$-modular category is a   $G$-ribbon   $G$-fusion category whose neutral component is
modular  in the sense of \cite{Tu0}.

Let $\cc$ be a $G$-modular category. Consider the fusion algebra $L=\oplus_{\alpha\in G}\, L_\alpha$ of~$\cc$ (see Section~\ref{sect-verlinde-algebra}). For $\alpha\in G$, set
$$
\omega_\cc^\alpha=\sum_{i\in I_\alpha}
 \dim  ( i) \,\langle i \rangle\in L_\alpha ,
$$
where
$I_\alpha$ is  a representative set  of simple objects
of~$\cc_\alpha$.  The vector $\omega_\cc^\alpha$
does not depend on the choice of $I_\alpha$. Since the action of
any $\beta\in G$ on ${\cc}$ transforms   simple objects in
${\cc}_\alpha$ into simple objects in ${\cc}_{\beta \alpha
\beta^{-1}}$ and preserves their dimension, the family
$\omega_\cc=(\omega_\cc^\alpha)_\alpha$ is conjugation
invariant.  By Lemma  \ref{lem-smallomega},   $\omega_\cc$
determines an isotopy invariant  $F(\ell,  \omega_\cc)\in
 \kk$ of special $G$-links  $\ell$.

We shall need several elements   of $\kk$ associated with  a
$G$-modular category ${\cc}$. Since each $ i\in I_1$ is
a simple object, the   twist  $ v_i\co i \to   i$ in $\cc_1$ (see Section~\ref{sect-cas-C1}) is
equal to  $v_i=\nu_i \,\id_{ i}$ for some $\nu_i\in \kk$.  Since $v_i$ is an isomorphism, $\nu_i\in \kk^*$.
Set
$$
\Delta_{\pm} =\sum \limits_{i\in I_1}{\nu_i^{\pm 1}(\dim  (i))^2}\in
\kk.
$$
Formulas \eqref{Fother1} and \eqref{Fother2} imply that
$\Delta_{\pm}=F(\ell^\pm, \omega_\cc)$ where $\ell^\pm\subset
S^3$ is a trivial $G$-knot   with framing $\pm 1$, see
Section~\ref{sect-special-G-graphs}. The properties of modular categories
imply that  $\Delta_{\pm} \in \kk^*$     and   $\Delta_+
 \Delta_-=\dim(\cc_1)$, see \cite[Formula II.2.4.a.]{Tu0}. In particular, $\dim(\cc_1)$ is invertible in $\kk$.
A {\it rank}   of ${\cc} $ is a square root $\mathcal  D\in \kk^*$ of $\dim(\cc_1)$.

\section{The surgery HQFT}\label{sect-modular-Gcat}

\subsection{An invariant of $G$-manifolds}
 By a   {\it closed connected $G$-manifold}  we    mean  a closed connected oriented
  $3$-dimensional manifold   whose fundamental
 group is endowed with   a conjugacy class of homomorphisms  to $G$. A   {\it closed   $G$-manifold} is a disjoint union of a finite number of
  closed connected $G$-manifolds. A {\it homeomorphism} of closed  $G$-manifolds   is an orientation preserving homeomorphism whose action in $\pi_1$ commutes with the maps to $G$ (up to conjugation).

  We   derive from a $G$-modular category $\cc$  with rank  $\mathcal  D $
a multiplicative $\kk$-valued homeomorphism invariant    $\tau_{\cc} $ of     closed
$G$-manifolds. The multiplicativity of $\tau_{\cc} $  means that $\tau_{\cc}(M\amalg N)= \tau_{\cc}(M) \,  \tau_{\cc}(N)$ for all closed $G$-manifolds $M, N$.

It suffices to define $\tau_{\cc}$ for closed connected $G$-manifolds. Present such a manifold $M$ as the
result of surgery on $S^3=\RR^3\cup\{\infty\}$ along a framed
 link $\ell\subset \RR^2\times (0,1)$ with  $\#\ell$ components.
Thus,  $M$ is obtained by gluing $\#\ell$
 solid tori to the exterior $E_\ell$ of $\ell$ in $S^3$. Pick
a base point $z\in E_\ell\setminus \{\infty\}  $ with big second
coordinate. Composing the inclusion map $\pi_1 (E_\ell, z)\to \pi_1(M,z)$ with a homomorphism  $\pi_1(M,z)\to G$ in the given conjugacy class, we obtain   a homomorphism  $ g \colon
\pi_1(C_\ell,z)=\pi_1(E_\ell,z) \to G$.  We orient $\ell$ in an
arbitrary way. It is clear that the $G$-link $(\ell, g)$ is
special  in the sense of Section \ref{sect-special-G-graphs}. Recall the elements $\Delta_+, \Delta_- \in \kk$ from Section~\ref{$G$-modular categories++}.
Set
$$
\tau_{\cc}(M)= \Delta_{-}^{\sigma(\ell)}
\, {\mathcal D}^{ -\sigma(\ell)-\#\ell-1} F(\ell, g , \omega_\cc)\in \kk
$$
where  $\sigma(\ell)$ is the signature of the compact oriented
4-manifold $B_\ell$ bounded by $M$ and obtained  from the 4-ball
$B^4$ by attaching  2-handles    along  tubular neighborhoods of
the components of $\ell$ in $S^3=\partial B^4$. Here   $B_\ell$
is oriented so that $\partial B_\ell=M$ in the category of
oriented manifolds.

\begin{thm}\label{MAINTH}
$\tau_{\cc}(M)$ is a homeomorphism invariant of  the $G$-manifold $M$.
\end{thm}

\begin{proof}
 We
should prove that $\tau_{\cc}(M)$ does not depend on the choices
made in its definition. First of all, the homomorphism $g$ is   well
defined only up to conjugation.
  Lemma \ref{lem-smallomega--} implies that    $ F(\ell,
 g,\omega_\cc )$ is preserved under conjugation of~$g$ so that
 $\tau_{\cc}(M)$
does not depend on the choice of $g$  in its conjugacy class. Since the duality
  $V\mapsto V^*$ preserves the dimension and  transforms simple
  objects in
${\cc}_\alpha$ into simple objects in ${\cc}_{\alpha^{-1}}$, we
have $(\omega_\cc^\alpha)^*=\omega_\cc^{{\alpha}^{-1}}$ for
all $\alpha \in G $. Lemma  \ref{lem-smallomega+} implies that
$F(\ell, g, \omega_\cc)$ is independent of  the choice of
orientation of $\ell$ and so is $\tau_{\cc}(M)$.

To prove the independence of the choice of $\ell$ we   use
Kirby's theory of moves on   links. By \cite{Ki}, any two framed
links   in $S^3$ yielding via surgery homeomorphic 3-manifolds
can be related by certain transformations called
  Kirby moves. There are moves of two kinds. The first    move
adds to a framed link $\ell\subset S^3$ a distant unknot
$\ell^{\pm}$  with framing $ \pm 1$; under this move the
4-manifold  $B_\ell$ is transformed into    $B_\ell \# \CCC P^2$.   The second   move preserves   $B_\ell$ and is induced by
a sliding  of a 2-handle  of $B_\ell$ across another 2-handle.
We   need a more precise version of this theory. Denote the
result of surgery on a framed link $\ell\subset S^3$ by
$M_\ell$. A {\it surgery presentation}  of a closed connected
oriented 3-manifold $M$  is a pair (a framed link $\ell\subset
S^3$, an isotopy class of   orientation preserving
homeomorphisms $ f\colon M\to M_\ell$). Note that any framing
preserving isotopy of $\ell$ onto itself induces a homeomorphism
$j_0\colon M_\ell\to M_\ell$. For any $ f\colon M\to M_\ell$ as
above, the pair  $(\ell, j_0 f)$ is  a surgery presentation of
$M$; we say that it is obtained from $(\ell, f)$ by isotopy. The
first Kirby move $\ell\mapsto \ell'= \ell \amalg \ell^{\pm}$
induces a homeomorphism $j_1\colon M_\ell\to M_{\ell'}  $ which
is the identity outside a small 3-ball containing $\ell^{\pm}$.
The second  Kirby move ${\ell}\mapsto {\ell}'$ induces  a
homeomorphism $B_{\ell}\to B_{{\ell}'}$ which restricts to a
homeomorphism of boundaries
   $j_2\colon M_{\ell}\to M_{{\ell}'}$. In both cases we say
that the surgery presentation $(\ell', j_k f\colon M\to
M_{\ell'})$ (where $k=1,2$)
  is obtained from $(\ell, f\colon M\to M_\ell)$ by the $k$-th
Kirby move. The arguments in   \cite{Ki}, Section 2  show that
for any surgery presentations
 $({\ell}_1, f_1\colon M_1\to M_{\ell_1})$  and $({\ell}_2,
f_2\colon M_2\to M_{\ell_2})$ of  closed  connected oriented
3-manifolds $M_1,M_2$ and for any isotopy class of orientation
preserving  homeomorphisms   $f\colon M_{1}\to M_{2}$ there is a
sequence of Kirby moves and isotopies transforming $({\ell}_1,
f_1)$ into $ ({\ell}_2,f_2 f)$.

The 3-manifold $M_\ell$ obtained by surgery  on a  special
$G$-link $(\ell\subset S^3, g\colon \pi_1(C_{\ell})\to
G)$ is a $G$-manifold  in the obvious way.  (Warning:
by definition,  $G$-links are oriented but their
orientations play  no role in the surgery construction.)  A
Kirby move on a special   $G$-link   $(\ell,g)$ yields a
special $G$-link   $(\ell'\subset S^3, g'\colon
\pi_1(C_{\ell'})\to G)$ where $g'$ is   the composition of
the inclusion homomorphism $\pi_1(C_{\ell'})\to
\pi_1(M_{\ell'})$, the isomorphism $\pi_1(M_{\ell'})
=\pi_1(M_{\ell})$ induced by the homeomorphism $j\colon
M_{\ell}\to M_{\ell'}$ as above and the homomorphism
$\pi_1(M_{\ell})\to G$ induced by   $g$.  The results of
the previous paragraph imply that  if surgeries on two special
$G$-links   in $S^3$ yield   homeomorphic
$G$-manifolds,  then these $G$-links  can be related
by  a finite sequence of
  Kirby moves,  isotopies, and orientation reversions on link
  components.

It is clear that  $\tau_{\cc}(M_\ell)$ is invariant under
isotopies on $\ell$. To prove the theorem it is   enough to show
that $\tau_{\cc}(M_\ell)$ is invariant under the  Kirby moves on
$\ell$.  Under the first Kirby  move $\ell\mapsto \ell'=\ell
\amalg \ell^{\pm}$ the meridian of $\ell^{\pm}$ is contractible
in    $M_{\ell'}$ and  the $G$-link   $ \ell'  $ is a
disjoint union of $\ell $ and the $G$-unknot  $\ell^{\pm} $ with
framing $\pm 1$.  Therefore  $$F(\ell', \omega_\cc)=
F(\ell^{\pm}, \omega_\cc)\, F(\ell, \omega_\cc)=\Delta_{\pm
}\, F(\ell, \omega_\cc) .$$
   This formula and the equalities $\#\ell'=\#\ell +1, \sigma
(\ell')=\sigma(\ell) \pm 1$,   $\Delta_+ \Delta_-=\mathcal D^2$
imply that
 $\tau_{\cc}(M_\ell)=\tau_{\cc}(M_{\ell'})$.

  We consider the second Kirby moves in the restricted form
studied by Fenn and Rourke  \cite{FR}. The  Kirby-Fenn-Rourke
moves split into    positive   and   negative ones. It is
explained in \cite{RT2} that  (modulo the first Kirby moves) it
is enough to consider only   the negative Kirby-Fenn-Rourke
moves. Such a move   $\ell \mapsto \ell'$ replaces a piece
$\Gamma$
  of $\ell$ lying in a closed 3-ball $B^3 $ by another piece
$\Gamma'$ lying in $B^3$ and having the same endpoints.  Here
$\Gamma=B^3\cap \ell$ is
  a system of $k\geq 1$ parallel strings with parallel framings
and  $\Gamma'=\Gamma^-\cup t$, where $\Gamma^-$ is obtained from
$\Gamma$ by applying a full left-hand twist and $t$ is an unknot
encircling $\Gamma$ and having the framing $-1$. (This move
$\ell \mapsto \ell'$ can  be expanded as a composition of $k$
Kirby moves of type 2 and a single Kirby move of type 1.) Note
that $\#\ell'=\#\ell +1$ and $\sigma (\ell')=\sigma (\ell)-1$.
We must prove that $F(\ell', \omega_\cc)=\Delta_{-}\, F(\ell,
\omega_\cc) $. This will follow  from a \lq\lq local" equality
which we now formulate.

 Let   $\Gamma$ be a trivial braid on $k$ strings in ${\mathbb
R}^2\times [0,1]$ with constant framing. Let
$\Gamma'=\Gamma^-\cup t\subset {\mathbb R}^2\times [0,1]$ be the
framed tangle obtained from $\Gamma$ as above.  Fix an arbitrary
orientation of $\Gamma'$ and  the induced orientation of
$\Gamma$. Let us transform $\Gamma'$ into a ribbon graph
$\Omega=\Gamma^-\cup \Omega_t$ by inserting a coupon in $t$ as
 in Figure \ref{insertion-coupon}. Clearly, $C_\Gamma= ({\mathbb
 R}^2\times [0,1])- \Gamma$ is obtained from $C_{\Gamma^-}=
 ({\mathbb R}^2\times [0,1])- {\Gamma^-}$ by surgery on
 $t\subset C_{\Gamma^-}$. This yields inclusions
 $C_\Omega\subset C_{\Gamma'}=C_{\Gamma^-} \setminus t \subset
 C_\Gamma$. (The reader uncomfortable with open manifolds may
replace the compliments by exteriors throughout the argument.)
These inclusions induce  a bijection $g\leftrightarrow g'$
between homomorphisms $  g \colon \pi_1(C_\Gamma)\to G$ and
homomorphisms $g'\colon \pi_1(C_{\Omega})=\pi_1(C_{\Gamma'})\to
G$ carrying the homotopy class of the $(-1)$-longitude of
$t$  to $1\in G$.   Any coloring $u$ of the $G$-graph
$(\Gamma,g)$ induces a coloring $u'$ of the $G$-graph
$(\Omega,g')$ such that
\begin{enumerate}
\labela
\item the values of $u'$ on the edge-tracks of   $\Gamma^-\subset
\Omega$ are equal to the values of $u$ on the corresponding
edge-tracks of $\Gamma$ (and the same for the isomorphisms
associated with pairs (a track of     $\Gamma^-$,  an   element
of $\pi_1(C_{\Omega}$));
\item $\Omega_t$ is colored as in Section \ref{sect-special-G-graphs}  using
the canonical color $\omega=\omega_\cc$.
\end{enumerate}
Clearly, the colored $G$-graphs $\Gamma$ and $\Omega$ have the
same source  and the same target. The properties of the functor
$F:\mathcal G_{\cc}\to \cc$  imply that to prove the equality
 $F(\ell', \omega_\cc)=\Delta_{-}\, F(\ell, \omega_\cc) $,
it is enough to show   that for any orientation  of $\Gamma'$
and any
  $ g $,  $u$ as above,
\begin{equation}\label{mainformulla}
F(\Omega,g',u')= \Delta_{-}\, F(\Gamma,g,u)\, .
\end{equation}
Let us prove this formula. Let
$(U_1,\varepsilon_1),\ldots ,(U_k,\varepsilon_k)$ be the
source of $\Gamma$. Using the standard technique of coupons
colored with identity morphisms, we can reduce the proof of
\eqref{mainformulla} to the case where $\Gamma$ is a single
string oriented from top to bottom and colored with
$\otimes_{r=1}^k U_r^{\varepsilon_r} \in {\mathcal C}$.
Similarly, using   a decomposition of this object as a direct
sum of simple objects, we can further reduce ourselves to the
case where the input and the output of  $\Gamma$ is a 1-term
sequence $(V,+)$ where $V$ is a simple object of ${\mathcal
C}$. Suppose that $V\in {\mathcal C}_\alpha$ where $\alpha \in
G$. By the argument above in this proof, $F(\Gamma',g',u')$ does
not change if we invert the orientation of $t$. Therefore we can
assume that $t$ is oriented so that its linking number with the
string $\Gamma$ is equal to $+1$. Clearly,
$F(\Gamma,g,u)=\id_V$. Since $V$ is simple,
$$
F(\Omega,g',u')=(\dim(V))^{-1}\tr \bigl ( F(\Omega,g',u') \bigr )\, \id_V.
$$
To establish
\eqref{mainformulla}, we need  only  to prove that $\tr \bigl ( F(\Omega,g',u') \bigr )=\dim(V) \Delta_{-}
$. By construction,
\begin{equation}\label{eq-pf-kirby1}
F(\Omega,g',u')=\sum_{i \in I_{\alpha^{-1}}} \dim(i)\, {\mathcal F}(D_i),
\end{equation}
where $D_i\in \mathcal{D}_\cc$ is the $\cc$-colored diagram of Figure~\ref{fig-kirby-pf1} and $I_{\alpha^{-1}}$ is a representative set of simple objects of $\cc_{\alpha^{-1}}$.
   \begin{figure}[t]
\begin{center}
\psfrag{v}[Br][Br]{\scalebox{.9}{$\id_{\varphi_\alpha(V)}$}}
\psfrag{u}[Br][Br]{\scalebox{.9}{$\id_{\varphi_{\alpha^{-1}}(i)}$}}
\psfrag{B}[Br][Br]{\scalebox{.9}{$\varphi_{\alpha^{-1}}(i)$}}
\psfrag{X}[Br][Br]{\scalebox{.9}{$\varphi_\alpha(V)$}}
\psfrag{V}[Bl][Bl]{\scalebox{.9}{$V$}}
\psfrag{Y}[Bl][Bl]{\scalebox{.9}{$\varphi_\alpha(V)$}}
\psfrag{q}[Bl][Bl]{\scalebox{.9}{$\psi=\varphi_2(\alpha,\alpha^{-1})_i^{-1}(\varphi_0)_i$}}
\psfrag{i}[Bl][Bl]{\scalebox{.9}{$i$}}
\psfrag{p}[Bl][Bl]{\scalebox{.9}{$\phi=\varphi_2(\alpha^{-1},\alpha)_V^{-1}(\varphi_0)_V$}}
$D_i=$\; \rsdraw{.45}{.9}{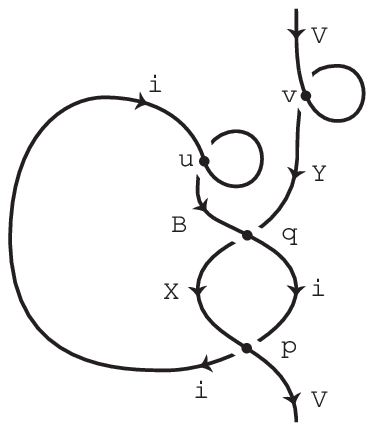}
\end{center}
\caption{}
\label{fig-kirby-pf1}
\end{figure}
Let $i \in I_{\alpha^{-1}}$. From Lemmas~\ref{lem-functorF} and \ref{lem-F-othergen}, we obtain $\tr\bigl({\mathcal F}(D_i) \bigr)=\tr(T_{i,V})$ with
$$
T_{i,V}=(\theta_i^{-1} \otimes \theta_V^{-1})\tau^{-1}_{\varphi_{\alpha^{-1}}(i),\varphi_\alpha(V)} (\id_{\varphi_\alpha(V)} \otimes \psi) \tau^{-1}_{\varphi_\alpha(V),i} (\id_{i} \otimes \phi).
$$
Using the naturality of $\tau$,   Lemma~\ref{lem-twist-mult},   and \eqref{crossing7}, we obtain
\begin{align*}
T_{i,V}& =(\theta_i^{-1} \otimes \theta_V^{-1})\tau^{-1}_{\varphi_{\alpha^{-1}}(i),\varphi_\alpha(V)} \tau^{-1}_{\varphi_\alpha(V),\varphi_{\alpha}\varphi_{\alpha^{-1}}(i)} (\psi \otimes \phi)\\
&= \theta_{i \otimes V}^{-1} (\varphi_1)_2(i,V)((\varphi_0)_i \otimes (\varphi_0)_V)= \theta_{i \otimes V}^{-1} (\varphi_0)_{i \otimes V}=v_{i \otimes V}^{-1},
\end{align*}
where $\{v_X=(\varphi_0)_X^{-1}\theta_X\co X \to X\}_{X \in \cc_1}$ is the (standard) twist of $\cc_1$ (see Section~\ref{sect-cas-C1}). Recall that for any simple object $X$ of $\cc_1$, $v_X=\nu_X\id_X$ for some $\nu_X \in \kk^*$. Let $I_1$ be a representative set of simple objects of $\cc_1$. Since $i \otimes V\in\cc_1$
splits as a (finite) direct sum of objects of $I_1$, there
exists a finite family of morphisms $( p_a\co i \otimes V \to i_a  , q_a \co i_a \to i \otimes V )_{a \in A}$
such that
\begin{equation*}
\id_{i \otimes V}=\sum_{a \in A} q_a p_a  \quad
\text{and} \quad p_a q_b=\delta_{a,b} \, \id_{i_a} \quad
\text{for all} \quad a,b\in A.
\end{equation*}
Then
\begin{gather*}
\tr\bigl({\mathcal F}(D_i) \bigr)=\tr(v^{-1}_{i \otimes V})=\tr(v^{-1}_{i \otimes V}\id_{i \otimes V})
=\sum_{a \in A} \tr(v^{-1}_{i \otimes V}p_aq_a)\\
 =\sum_{a \in A} \tr(p_av^{-1}_{i_a}q_a)
=\sum_{a \in A} \nu_{i_a}^{-1} \tr(p_aq_a) =\sum_{a \in A} \nu_{i_a}^{-1} \dim(i_a)\\
 =\sum_{k \in I_1} \sum_{\substack{a \in A\\ i_a=k}} \nu_{k}^{-1} \dim(k)=\sum_{k \in I_1} N_{i \otimes V}^k \nu_{k}^{-1} \dim(k).
\end{gather*}
Finally, using \eqref{eq-pf-kirby1} and \eqref{eq-dim-mult-numbers}, we obtain
\begin{gather*}
\tr \bigl ( F(\Omega,g',u') \bigr )=\sum_{i \in I_{\alpha^{-1}}} \sum_{k \in I_1} \dim(i) N_{i \otimes V}^k \nu_{k}^{-1} \dim(k)\\
=\sum_{k \in I_1} \sum_{i \in I_{\alpha^{-1}}}  \dim(i^*) N_{V \otimes k^*}^{i^*} \nu_{k}^{-1} \dim(k)
=\sum_{k \in I_1} \left ( \sum_{j \in I_{\alpha}}  \dim(j) N_{V \otimes k^*}^{j} \right ) \nu_{k}^{-1} \dim(k)\\
=\sum_{k \in I_1} \dim(V \otimes k^*) \nu_{k}^{-1} \dim(k) = \sum_{k \in I_1} \dim(V) \nu_{k}^{-1} (\dim(k))^2=\dim(V) \Delta_{-}.
\end{gather*}
This concludes the proof of Theorem~\ref{MAINTH}.
\end{proof}

\subsection{Remarks}
1.  The 3-sphere $S^3$ has a unique structure of a   closed
$G$-manifold. It can be obtained by the surgery  on $S^3$
along an empty link. Hence, $\tau_{\cc}(S^3)={\mathcal
D}^{-1}$.

2. We
have $\tau_\cc  (S^1\times S^2,f )=1$ for any homomorphism $f \colon \pi_1(S^1\times S^2) \approx \ZZ\to G$. Indeed,  the closed $G$-manifold
$(S^1\times S^2, f )$ can be obtained by the surgery on $S^3$ along
an unknot $\ell$ with framing $0$ and with homomorphism $g\colon
\pi_1(C_\ell)\to G$ carrying a meridian of $\ell$  to a certain
$\alpha\in G$. Then $\sigma (\ell)=0$ and by   \eqref{eq-dim-Cg},
$$\tau_\cc  (S^1\times S^2,f )= \mathcal D^{-2} F(\ell, g , \omega_\cc)=\mathcal D^{-2}
\sum_{i \in I_\alpha} (\dim( i))^2= \mathcal D^{-2} \dim(\cc_1) =1.$$

3.  The definition of $\tau_\cc  (M)$ can be rewritten in a more
 symmetric form:   $$
\tau_\cc  (M)= \mathcal D^{-b_1(M)-1}\Delta_-^{-\sigma_-}
\Delta_+^{-\sigma_+} F(\ell, g, \omega_\cc)$$ where
$b_1(M)=\#\ell -\sigma_+-\sigma_-$ is the first Betti number of
$M$ and $\sigma_+$ (resp.\ $ \sigma_-$) is the number of
positive (resp.\ negative) squares in the diagonal decomposition
of the intersection form $H_2(B_\ell) \times H_2(B_\ell) \to
\mathbb  Z$. The invariant
$$ \mathcal  D^{b_1(M)+1} \tau_{\cc}(M)=\Delta_-^{-\sigma_-} \Delta_+^{-\sigma_+}
F(\ell, g, \omega_\cc)$$
does not depend on the choice of $\mathcal D$.

4. The invariant  $\tau_\cc (M)$   can be defined without the
invertibility assumption  on the $S$-matrix   of $\cc$. It
suffices to require the  weaker  condition $\Delta_+,
\Delta_-\in \kk^*$.    The invertibility of  the $S$-matrix  is
needed in order to extend    $\tau_\cc$   to an HQFT in the next subsection.

\subsection{The   HQFT}\label{VI.2.2+} Consider again a $G$-modular category $\cc$  with rank  $\mathcal  D $. The invariant $\tau_{\cc}$
of closed   $G$-manifolds extends to a 3-dimensional HQFT   with target an Eilenberg-MacLane space $K(G,1)$,   i.e., to a    symmetric
  monoidal projective functor  from the category of $G$-surfaces and 3-dimensional $G$-cobordisms $ \mathrm{Cob}^G$ to $
  \mathrm{vect}_\kk$. For   precise definitions of $G$-surfaces, $G$-cobordisms, and HQFTs, we refer to  \cite{TVi2}.
 The resulting projective functor $ \mathrm{Cob}^G \to
  \mathrm{vect}_\kk$   is still denoted by $\tau_{\cc}$.    The   construction of    $\tau_{\cc}$ is given in
   \cite[Chapter VII]{Tu1}   when $\cc$ belongs to the class of
   strict $G$-modular categories considered there; the same method
  applies to   $G$-modular categories in the sense
  of this paper. The projectivity of $\tau_{\cc}$ may be described more precisely: the homomorphism associated with any
  $G$-cobordism obtained by gluing two $G$-cobordisms is equal
  to the composition of the corresponding homomorphisms times an integer power of $\Delta_+ \Delta_-^{-1}\in \kk^*$. If $\Delta_+ = \Delta_-$, then   $\tau_{\cc}$  is   a functor.
  If $\Delta_+
  \neq \Delta_-$, then  the multiplicative ambiguity of $\tau_{\cc}$  may be resolved by     enriching $G$-surfaces (and in
  particular the bases of $G$-cobordisms) with Lagrangian
  subspaces in real 1-dimensional homology. The  projective
  functor $\tau_{\cc}$ lifts to a symmetric monoidal functor
  from the category of such enriched $G$-cobordisms  to
   $\mathrm{vect}_\kk$, cf.\ \cite[Chapter VII]{Tu1}.

For completeness, we give an explicit expression for (the isomorphism type
 of) the $\kk$-module $\tau_{\cc}(\Sigma)\in \mathrm{vect}_\kk$
associated with a (closed connected) $G$-surface $\Sigma$ of genus
$n\geq 0$. Such a surface carries a base point, $\bullet$, and a homomorphism $ \pi_1(\Sigma,\bullet)\to G$. If $n=0$, then $\tau_{\cc}(\Sigma)\simeq \kk$. If
$n\geq 1$, then   a skeleton of $\Sigma$ is  formed by   $2n$ loops beginning
and ending at $\bullet$   as   in Figure~\ref{fig-Surface}.
\begin{figure}[t]
\begin{center}
 \psfrag{i}[Bc][Bc]{\scalebox{.9}{$\alpha_1$}}
 \psfrag{j}[Bc][Bc]{\scalebox{.9}{$\beta_1$}}
  \psfrag{c}[Bc][Bc]{\scalebox{.9}{$\alpha_2$}}
 \psfrag{g}[Bc][Bc]{\scalebox{.9}{$\beta_2$}}
  \psfrag{e}[Bc][Bc]{\scalebox{.9}{$\alpha_n$}}
 \psfrag{q}[Bc][Bc]{\scalebox{.9}{$\beta_n$}}
\rsdraw{.45}{.9}{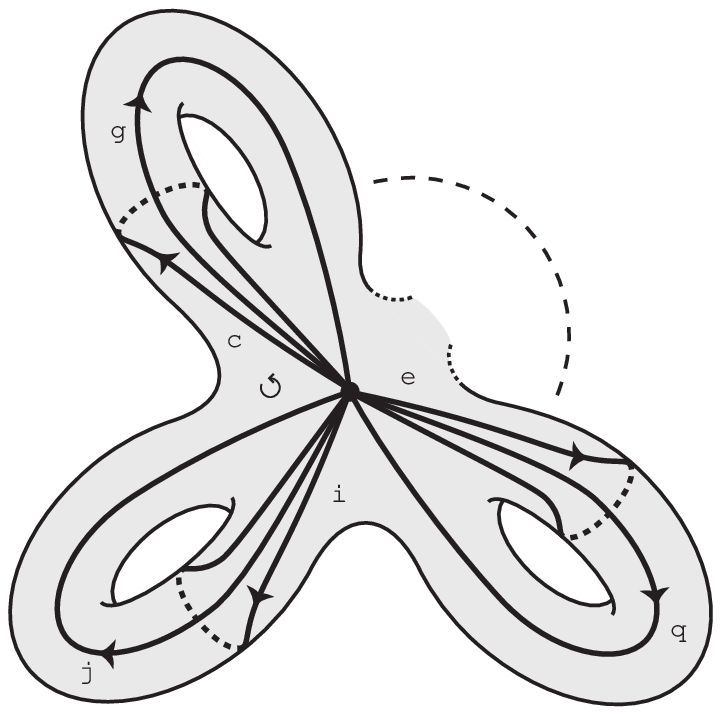}\,.
\end{center}
\caption{}
\label{fig-Surface}
\end{figure}
Let $\alpha_1,\beta_1, \ldots, \alpha_n, \beta_n\in G$ be the evaluations of the given homomorphism $ \pi_1(\Sigma,a)\to G$ on these  loops,    as indicated in
the figure. Note that
$   \prod_{i=1}^n   \alpha_i^{-1}
\beta_i^{-1} \alpha_i \beta_i  =1$.
 Given a representative set $ \amalg_{\alpha \in G}\,
 {\mathcal I}_\alpha$ of simple objects of $\cc$, we have
\begin{equation}\label{eq-formulaforthemodule}
\tau_{ \cc }(\Sigma)\simeq\!\!\!\!\!\!\bigoplus_{J_1 \in  {\mathcal I}_{\beta_1}, \dots, J_n \in {\mathcal I}_{\beta_n}} \!\!\!\!\!\! \Hom_{ \cc }(\un_{ \cc},
 \varphi_{\alpha_1}(J_1^*)  \otimes J_1 \otimes \cdots \otimes \varphi_{\alpha_n}(J_n^*) \otimes J_n).
\end{equation}
This formula directly follows from the definition of $\tau_{ \cc
}(\Sigma)$ and allows one to compute $\rank_{\kk} \, \tau_{ \cc
}(\Sigma)$ via a  version of the standard  Verlinde formula,
cf.\@ \cite{Tu1}.

\end{document}